\documentclass[a4paper,11pt]{amsart}

\usepackage{graphicx}
\usepackage{amsmath}
\usepackage{amssymb,amsfonts}

\numberwithin{equation}{section}

\newtheorem{theorem}{Theorem}[section]
\newtheorem{lemma}[theorem]{Lemma}
\newtheorem{proposition}[theorem]{Proposition}

\theoremstyle{definition}

\newtheorem*{rem}{Remark}

 \newcommand{\Z}{\mathbb{Z}}

\newcommand{\starsum}{\sideset{}{^*}\sum}

\renewcommand{\phi}{\varphi}

\newcommand{\0}{\mathbf{0}}

\newcommand{\PP}{\mathbb{P}}
\renewcommand{\AA}{\mathbb{A}}
\newcommand{\FF}{\mathbb{F}}
\newcommand{\ZZ}{\mathbb{Z}}

\newcommand{\NN}{\mathbb{N}}
\newcommand{\QQ}{\mathbb{Q}}
\newcommand{\RR}{\mathbb{R}}
\newcommand{\CC}{\mathbb{C}}

\newcommand{\cS}{\mathcal{S}}

\newcommand{\cD}{\mathcal{D}}
\newcommand{\cM}{\mathcal{M}}
\newcommand{\cW}{\mathcal{W}}
\newcommand{\cR}{\mathcal{R}}

\newcommand{\cI}{\mathcal{I}}
\newcommand{\cL}{\mathcal{L}}

\newcommand{\cT}{\mathcal{T}}

\renewcommand{\leq}{\leqslant}

\renewcommand{\geq}{\geqslant}

\renewcommand{\bar}{\overline}

\newcommand{\m}{\mathbf{m}}

\newcommand{\x}{\mathbf{x}}
\newcommand{\y}{\mathbf{y}}
\renewcommand{\c}{\mathbf{c}}
\newcommand{\f}{\mathbf{f}}
\renewcommand{\v}{\mathbf{v}}
\renewcommand{\u}{\mathbf{u}}
\newcommand{\z}{\mathbf{z}}

\renewcommand{\a}{\mathbf{a}}

\newcommand{\h}{\mathbf{h}}
\newcommand{\g}{\mathbf{g}}

\renewcommand{\r}{\mathbf{r}}
\renewcommand{\t}{\mathbf{t}}
\newcommand{\s}{\mathbf{s}}
\newcommand{\w}{\mathbf{w}}

\newcommand{\fm}{\mathfrak{m}}

\newcommand{\ve}{\varepsilon}

\DeclareMathOperator{\supp}{supp}

\DeclareMathOperator{\Spec}{Spec}

\newcommand{\fS}{\mathfrak{S}}
\newcommand{\cH}{\mathcal{H}}

\DeclareMathOperator{\Mod}{mod} 
\renewcommand{\bmod}[1]{\,(\Mod{#1})}

\newcommand{\dl}{\Delta}

\DeclareMathOperator{\Sing}{Sing}
\DeclareMathOperator{\content}{cont}
\DeclareMathOperator{\resultant}{Res}

\DeclareMathOperator{\vol}{vol}

\newcommand{\cN}{\mathcal{N}}
\newcommand{\cY}{\mathcal{Y}}
\newcommand{\cZ}{\mathcal{Z}}

\newcommand{\cP}{\mathcal{P}}

\newcommand{\cU}{\mathcal{U}}

\newcommand{\cX}{\mathcal{X}}

\newcommand{\tR}{B_2}
\newcommand{\bR}{B}
\newcommand{\fI}{\mathfrak{I}}
\newcommand{\fM}{\mathfrak{M}}
\newcommand{\R}{\mathbf{R}}
\newcommand{\e}{\mathbf{e}}

\newcommand{\dashsum}{\sideset{}{'}\sum}
\newcommand{\dashmax}{\sideset{}{'}\max}

\newcommand{\GG}{\mathbb{G}}

\begin{document}
\title{On the Hasse principle for quartic hypersurfaces}
\author{O. Marmon}
 \address{Centre for Mathematical Sciences \\ Lund University \\ Box 118 \\ 221 00 Lund \\ Sweden}
 \email{oscar.marmon@math.lu.se}
\author{P. Vishe}
\address{
Department of Mathematical Sciences,\\
Durham University\\
Durham\\ DH1 3LE\\ UK}
\email{pankaj.vishe@durham.ac.uk}
\begin{abstract}
We establish the Hasse principle for smooth projective quartic hypersurfaces of dimension greater than or equal to $28$ defined over $\QQ$.
\end{abstract}
 \maketitle
\section{Introduction}
\label{sec:intro}

Let $X\subset 
\PP^{n-1}_\QQ $ be a quartic
hypersurface corresponding to the zero locus of a homogeneous quartic polynomial
$F\in \ZZ[x_1,...,x_{n}] $. Determining whether $X$ contains a 
rational 
point is a fundamental question in Diophantine geometry. The variety $X$ is said to satisfy 
the \textit{Hasse principle} if $X$ contains a rational point provided that it contains an ad\`elic point. In other words, $X(\AA_\QQ)  \neq \emptyset\Rightarrow X(\QQ)\neq \emptyset$, where $X(\AA_\QQ) = X(\RR) \times \prod_{p} X(\QQ_p)$ is the set of ad\`elic points of $X$. The aim of this paper is to 
establish conditions on $n$ under which $X$ satisfies the Hasse principle.

A counterexample due to Swinnerton-Dyer \cite{Swinnerton_Dyer00}: $F(\x) = 7x_1^4+8x_2^4-9x_3^4-14x_4^4$, implies that the Hasse principle cannot be expected to be true for all quartic hypersurfaces.
However, this and other known counterexamples are explained by the Brauer-Manin obstruction. By a result of Colliot-Th\'el\`ene \cite[Appendix]{Poonen_Voloch02}, the Brauer-Manin obstruction is void for non-singular hypersurfaces in $\PP^{n-1}_\QQ$ provided that 
$n\geq 5$. As a result, it is conjectured that the Hasse principle holds for a non-singular quartic hypersurface $X$ defined over $\QQ$
as long as $n\geq 5$. 

A long-standing result by Birch \cite{Birch61} shows that $X(\QQ) \neq \emptyset$ provided that $X$ possesses a non-singular ad\`elic point and 
\[
n - \dim \Sing(X) \geq 50,
\]
where $\Sing(X)$ denotes the singular locus of the projective variety $X$. In particular, this establishes the Hasse principle for non-singular quartic hypersurfaces as soon as $n \geq 49$ (recall that the empty set is declared to have dimension $-1$). Birch in fact provides an admissible range of $n$ for a 
hypersurface of arbitrary degree $d$, with a bound depending on $d$. While Birch's 
result has been improved significantly in the cubic case over the years, 
improving upon it when $d\geq 4$ has turned out to be a much more formidable task. A 
breakthrough was achieved by Browning and Heath-Brown \cite{Browning-Heath-Brown09}, when they 
reduced the lower bound for $n$ in the quartic case from $49$ to $41$.  
Hanselmann \cite{Hanselmann} then established the case $n=40$. The methodology in 
\cite{Browning-Heath-Brown09} 
has since been generalised by Browning and Prendiville \cite{Browning_Prendiville}, thus improving 
upon Birch's bounds for every degree $d\geq 5$. In the special case of diagonal forms $F = a_1 x_1^4 + \dotsc + a_n x_n^4$ with all $a_i \neq 0$, Vaughan \cite{Vaughan89} shows that 12 variables suffice.

The main theorem of the present paper records a major improvement in the range of $n$ for which the Hasse principle holds. Let $X_{\mathrm{ns}} = X \setminus \Sing(X)$ denote the non-singular locus of $X$. As in \cite{Browning-Heath-Brown09}, our result takes the quantitative form of a lower bound for the counting function
\[
N(X,P) = \#\{x \in X(\QQ) \mid H(x) \leq P\},
\]
where $H(\cdot)$ is the usual height on $\PP^{n-1}(\QQ)$ defined by $H(x_1:...:x_n)=\max_i\{|x_i|\}$, where $x_1,...,x_n\in\ZZ$ such that $\gcd(x_1,...,x_n)=1$.

\begin{theorem}
\label{thm:main thm}
Let $X\subset \PP^{n-1}_\QQ$ be a quartic hypersurface satisfying 
$n-\dim\Sing(X)\geq 31$. Assume that $X_{\mathrm{ns}}(\AA_\QQ) \neq \emptyset$.
Then there exist constants $P_0 \geq 1$ and $c > 0$ such that
\[
N(X,P) \geq cP^{n-4} 
\]
as soon as $P \geq P_0$.
\end{theorem}

In particular, this establishes the Hasse principle for non-singular quartic hypersurfaces as soon as $n \geq 30$, saving 10 variables over the previously best known result. We expect that a more elaborate version of our approach will allow us to save one further variable. We shall devote a follow-up paper dedicated to achieving this improvement. Adapting an idea of Hooley \cite{Hooley-Octonary}, the result can possibly be improved even further under the assumption of a generalised Riemann hypothesis for a class of Hasse-Weil $L$-functions.  It is likely that in the vein of \cite{Browning_Vishe14} and \cite{Browning_Vishe15}, the methods here can be generalised to obtain the Hasse principle and weak approximation in the number field and function field setting. Moreover, it is likely that the techniques will be able to generalise to the setting of homogeneous polynomials of degree $d$.

\subsection{Key ideas}\label{sub:key}Let us briefly discuss the key ideas in the proof of Theorem \ref{thm:main thm}. To begin with, we replace the counting function $N(X,P)$ by the smoothed version
$$N_W(F,P)=\sum\limits_{\substack{\x\in 
\ZZ^n\\
F(\x)=0}}W(\x/P)$$
for a suitably chosen smooth weight function $W: \RR^n \to \RR_{\geq 0}$ with compact support. One clearly has $N(X,P) \gg N_W(F,P)$ if $W$ is chosen appropriately.
The estimation of $N_W(F,P)$ proceeds via a variant of the Hardy-Littlewood circle method. In its classical form, this begins by writing
\begin{equation}\label{circle}
 N_W(F,P)=\int_0^1 S(\alpha)d\alpha,
 \end{equation}
where $S(\alpha)$ is the 
generating function
\begin{equation}
 \label{Salpha}
 S(\alpha)=\sum_{\x\in \ZZ^n}W(\x/P)e(\alpha F(\x)),
\end{equation} 
and splitting the unit interval into major and minor arcs as usual.
Most modern versions of the circle method start with an application of the Poisson 
summation formula to estimate $S(\alpha)$. However, in the present setting, if $F$ is a polynomial of 
degree $4$ or more, the bounds for the exponential integral which emerge out of 
this process turn out to be too large to obtain an admissible bound for the minor arc contribution. This fundamental issue was overcome by Browning and Heath-Brown \cite{Browning-Heath-Brown09}. They used a point-wise van 
der Corput differencing to bound the exponential sum $S(\alpha)$ in the minor arcs. Hanselmann \cite{Hanselmann} further incorporated the averaging 
trick 
introduced by Heath-Brown \cite{Heath-Brown07} along with the van der Corput differencing 
to save an extra variable.

Our main improvement over previous results comes from achieving  non-trivial cancellation in the averages $$\starsum_{a=1}^q S(a/q+z) $$ in the minor arcs, as pioneered by Kloosterman \cite{Kloosterman}. Here the $*$ over the sum indicates that $a$ and $q$ are co-prime.
Let
$$
\delta_0(n) =
\begin{cases}
1,  & \mbox{if $n = 0$}, \\
0,  & \mbox{otherwise},
\end{cases}
$$
denote the delta function detecting when an integer $n=0$. We begin by rewriting the delta symbol method of Duke, Friedlander and Iwaniec \cite{Duke_Friedlander_Iwaniec93} in the following, possibly a little bit more familiar form.
\begin{proposition}
\label{prop:Kloos}
 Let $Q\geq 1$ and let $n$ be an integer. Then, given any $\theta>0$, one has
$$
 \delta_0(n)=
 \sum_{q=1}^Q
~\starsum_{a=1}^q\int_{|z|<(qQ)^{-1+\theta}} p_{q}(z)e((a/q+z)n)\,dz+O_{N,\theta}(Q^{-N\theta}),
$$
where $p_q(z)$ is a smooth function satisfying
\begin{equation}p_q(z)\ll_N 1,\label{eq:pbound1}\end{equation}                                            
and
\begin{equation}p_q(z)=1+O_{N}\left((q/Q)^N\right)\quad \textrm{ for } \quad |z|\leq Q^{-2},\label{eq:pbound2}\end{equation} 
for any $N\geq 0$. 
\end{proposition}
The proof of Proposition \ref{prop:Kloos} will be carried out in Section \ref{sec:setup}. The functions $p_q(z)$ can be viewed as {\em smooth} symmetrically placed arcs around points $\{a/q: \gcd(a,q)=1\}$ of an approximate length $O((qQ)^{-1+\theta})$. Thus, the proposition 
can be viewed as an exact smooth version of Kloosterman's circle method.  This reinterpretation of the delta symbol method was already implicitly a key idea in \cite{Heath-Brown96},   \cite{Munshi15} and \cite{Browning_Vishe14} etc. The version stated here suppresses the dependence on the mysterious $h(x,y)$ function appearing in the previous works completely.  This is achieved by providing a finer analysis of the functions $p_q(z)$ for ``small'' values of $q$ as well. Moreover, it also allows us to choose $Q$ independent of the degree $d$ of the polynomial $F$. This is crucial in our work as we choose $Q=P^{8/5+\ve}$ which is significantly less than the natural choice $P^2$ permitted by the term $h(q/Q,F(\x)/Q^2)$ arising from the earlier versions. It should be noted that one may also analogously obtain bounds for the derivatives of $p_y(z)$ with respect to $y$, which is often necessary to obtain extra cancellations in the $q$-sum. Given how versatile the delta symbol method has been in its applications, it is likely that our version in Proposition \ref{prop:Kloos} will be of independent interest to the readers.

Applying Proposition \ref{prop:Kloos} to the expression
\begin{align*}
N_W(F,P)
=\sum\limits_{\substack{\x\in 
\ZZ^n}}W(\x/P)\delta_0(F(\x)),
\end{align*}
we obtain the following corollary, which takes the place of the identity \eqref{circle}.
\begin{proposition}
\label{prop:delta_counting}
For any $Q,P\geq 1$,
\begin{align}
\label{eq:kloosterman}
 N_W(F,P)
=\sum_{q=1}^Q \int_{|z|<(qQ)^{-1+\theta}} p_{q}(z)S(q,z)\,dz+O_{N,\theta}(Q^{-N\theta}P^n),
\end{align}
where 
\begin{equation}
 \label{Su}
 S(q,z)=\starsum_{a=1}^q S(a/q +z),
\end{equation}
and $p_q(z)$ is a smooth function satisfying \eqref{eq:pbound1} and \eqref{eq:pbound2}.
\end{proposition}

The success of our method relies on combining this Kloosterman type circle method with van der Corput differencing process from \cite{Browning-Heath-Brown09} as well as the averaging procedure in \cite{Hanselmann}.
We apply van der Corput differencing process 
to the exponential sum $S(q,z)$ defined in \eqref{Su} rather than to $S(\alpha)$.
This approach still allows us to maintain the key feature of the method in 
\cite{Browning-Heath-Brown09} with regard to the exponential integral arising from the Poisson summation. The resulting exponential sums, however, are of a different nature. In the case where $q$ is squarefree, they may be interpreted as exponential sums on varieties over finite fields which are intersections of a quartic and a cubic hypersurface. To estimate these, Deligne type bounds due to Katz \cite{Katz} come into play. For squarefull $q$ we are able to recycle the bounds in \cite{Browning-Heath-Brown09}
for the averages of cubic exponential sums. We also need to consider sums of such exponential sums over sparser sets, corresponding to certain dual varieties, as in \cite{Heath-Brown83}. The fact that these dual varieties now vary with the parameter $\h$ in van der Corput differencing process, provides an additional difficulty over the situation in \cite{Heath-Brown83}.  

\subsection{Acknowledgements} 
While working on this paper, the first author was supported by the Knut and Alice Wallenberg Foundation. In the course of this work, we have benefited from discussions with Tim Browning, Chris Hall, Nick Katz, Fabien Pazuki, Dan Petersen and Will Sawin. Their help is greatly appreciated. We also thank the anonymous referees for their comments. Their comments/suggestions have improved the overall exposition significantly.

\section{Setup of the circle method}
\label{sec:setup}

We will begin by establishing the proof of Proposition \ref{prop:Kloos} which provides a stepping stone in proving the results in this paper. We start by recalling Heath-Brown's version \cite[Thm 1]{Heath-Brown96} of the delta symbol method: 
\begin{lemma}
\label{lem:DFI}
For any $Q>1$ there is a positive constant $c_Q$, and a smooth function $h(x,y)$ defined on $(0,\infty)\times\mathbb R$, such that
\begin{equation}
\label{eq:DFI}
\delta_0(n)=\frac{c_Q}{Q^2}\sum_{q=1}^{\infty}\;\starsum_{a=1}^q e_q\left(an\right)h\left(\frac{q}{Q},\frac{n} {Q^2}\right)
\end{equation}
for $n\in\mathbb Z$. Here $e_q(x)=e^{2\pi i x/q}$. The constant $c_Q$ satisfies $c_Q=1+O_N(Q^{-N})$ for any $N>0$. Furthermore,  we have $\partial_y^j h(x,y)\ll_N x^{-1-j}\min\left\{1,(|y|/x)^{-N}\right\}$ for all $y$ and $j\geq 0$
and $h(x,y)\neq 0$ only if $x\leq \max\{1,2|y|\} $.\\
\end{lemma}

The following lemma provides the key in proving Proposition \ref{prop:Kloos}. The main ingredient in the proof here is a very simple yet effective trick which has appeared in a work of Munshi \cite{Munshi15}, applied to the Lemma \ref{lem:DFI}.

\begin{lemma}
 \label{lem:klooster}
 Let $Q\geq 1$ and let $n$ be an integer. Then
$$
 \delta_0(n)=
 \sum_{q=1}^Q
~\starsum_{a=1}^q\int_z p_{q}(z)e((a/q+z)n)\,dz+O_N(Q^{-N}),
$$
where $p_q(z)$ is a smooth function satisfying
$$p_q(z)\ll_N (qQz)^{-N}, \textrm{ and }p_q(z)=1+O_N\left((1+Q^2|z|)^{2N+2} (q/Q)^{N}\right).$$
for any  $N\geq 0$. 
\end{lemma}
\begin{proof}
 Let $U:(-1/2,1/2)\rightarrow \RR$ be a non-negative compactly supported function satisfying $\int U(x) dx=1$ and $U(0)=1$. The starting point of this method is the following simple observation
 $$\delta_0(n)=\delta_0(n)U(n/Q^2).$$
Upon substituting \eqref{eq:DFI} for $\delta_0(n)$ on the right hand side of the above equation, we get
\begin{align}
\label{eq:DFI1}
 \delta_0(n)=\frac{c_Q}{Q^2}\sum_{q=1}^{Q}\;\starsum_{a=1}^qe_q\left(an\right)h\left(\frac{q}{Q},\frac{n} {Q^2}\right)U\left(\frac{n}{Q^2}\right)+O_N(Q^{-N}).
\end{align}
The truncated sum over $q$ is due to the fact that $U(n/Q^2)$ is non-zero only when $|n|<Q^2$, and that $h(x,y)$ is non-zero only if $x\leq \max\{1,2|y|\} $.
Next, we use the Fourier inversion formula to write
\begin{align*}
 h\left(x,y\right)U\left(y\right)=\int_\RR f(t)e(ty)dt,
\end{align*}
where
\begin{align*}
 f(t)=\int_{\RR}h(x,z)U(z)e(-tz)\, dz.
\end{align*}
Upon using repeated integration by parts, we get 
\begin{align*}
 |f(t)|&=|\int_{\RR}(-2\pi i t)^{-j}\partial_z^j(h(x,z)U(z))e(-tz)\, dz|\\
 &\ll_{j,N} |tx|^{-j}\int_{|z|<1}x^{-1}\min\left\{1,(|z|/x)^{-N}\right\}\, dz\\
 &\ll_j |tx|^{-j},
\end{align*}
for any $0<x\leq 1$. Moreover, using Lemma 4.1 of Browning and Vishe \cite{Browning_Vishe14}, which is essentially proved by repeated integration by parts together with the bounds on the derivatives of the function $h(x,z)$, we have
\begin{align*}
 f(t)=\int_{\RR}h(x,z)U(z)e(-tz)=1+O_{N}\Big((1+|t|)^{2N+2}|x|^{N}\Big).
\end{align*}
Substituting this back to \eqref{eq:DFI1}, we get 
\begin{align*}
 \delta_0(n)=\frac{1}{Q^2}\sum_{q=1}^{Q}\;\starsum_{a=1}^qe_q\left(an\right)\int_t f(t)e(tn/Q^2)dt+O_N(Q^{-N}).
\end{align*}
The theorem now ensues upon the change of variable $z=t/Q^2 $ and defining $p_q(z)=f(Q^2z)$.
\end{proof}
By using the decay properties of functions $p_q(z)$, one may easily derive Proposition \ref{prop:Kloos} from Lemma \ref{lem:klooster}.

To begin the circle method analysis, we consider a smooth weight function $\omega$ with a support in a very small region around a non-singular point $\x_0 \in X(\RR)$, a standard choice necessary to gain control over minor arc contribution. The existence of such a point is guaranteed by the assumption that $X$ possesses a non-singular ad\`elic point.

To this end, let $\x_0\in\RR^n$ be a point satisfying $F(\x_0)=0$ and $\nabla F(\x_0)\neq 0$, 
which we will fix 
from now on. Without loss of generality, we can assume that 
$\partial_{x_1}F(\x_0)\neq 0$. Let
$$\gamma(\x):=\begin{cases}
              \prod_j e^{-1/(1-|x_j|)^2}, &\textrm{ if }|\x|<1,\\
              0 &\textrm{ otherwise }.
             \end{cases}
$$
We will use the weight function \begin{equation}
\omega(\x):=\prod_{j=1}^n\gamma(\rho^{-1}(\x-\x_0)),
\label{eq:W def}
\end{equation}
where $\rho \in (0,1]$ is a parameter to be chosen at a later stage.
With this choice of the smooth weight function, our main goal will be to establish the following asymptotic formula:
\begin{theorem}
\label{thm:thm 3}
Let $F$ be a quartic form defined over $\QQ$ satisfying $n-\sigma \geq 31$. Then
\[
N_\omega(F,P) = c_F P^{n-4} + O(P^{n-4-\delta}),
\]
for a positive constant $c_F$, provided $P \gg 1$ and $X_{\mathrm{ns}}(\AA_\QQ) \neq \emptyset$.
\end{theorem}
Here,
\[
\sigma = \dim \Sing(X),
\]
where $\Sing(X)$ denotes the singular locus of the projective variety $X\subset \PP^{n-1}$. As a consequence of Theorem \ref{thm:thm 3}, we have
\[
N(F,P) \geq N_\omega(F,P) \geq c_F P^{n-4}.
\]
This would conclude the proof of Theorem \ref{thm:main thm}. In the remaining section, we devise a strategy to establish Theorem \ref{thm:thm 3}, via key Proposition \ref{prop:minor}, stated later.

To establish Theorem \ref{thm:thm 3}, we begin by using Proposition \ref{prop:delta_counting} with
\begin{equation}\label{eq:phidef}
Q = P^{8/5 + \phi},
\end{equation}
where $0<\phi<1$ is a small parameter to be chosen at a later stage. This choice  of $Q$ is concurrent with the one in \cite{Browning-Heath-Brown09}. It arises from balancing various terms in the bounds coming from van der Corput differencing which are useful for ``large'' $q$'s. This choice is much less than that of Birch $Q=P^2$, since our bounds are significantly better than those of Birch for minor arcs corresponding to large $q$'s. They are supplemented by bounds from Weyl differencing which are necessary for minor arcs around 
small $q$'s.

\subsection{Major arcs considerations} The dominating contribution to the main term in \eqref{eq:kloosterman} is expected to occur from small values of $q$ and $z$. More 
explicitly, given $\Delta>0$, let $S_\fM$ denote
the contribution from the {\em major arcs} regime;
\begin{equation}
 \label{eq:major}
S_\fM=\sum_{1\leq q\leq P^\Delta} \int_{|z|\leq P^{-4+\dl}}p_q(z) S(q,z)\,dz.
\end{equation}
If we take the major arcs to be narrow enough, then we may replace the function $p_q(z)$ inside the integral by $1$, with an admissible error. For this, it will be enough to assume that
\begin{equation}
\label{eq:narrow}
\phi + \frac{\Delta}{2} \leq \frac{2}{5}.
\end{equation}
Indeed, then $|z| \leq P^{-4+\Delta}$ implies $|z| \leq Q^{-2}$ so that $p_q(z) = 1 + O_{N}((q/Q)^{-N})$ for any $N$ by \eqref{eq:pbound2}. Thus we get
\begin{equation}
\label{eq:major2}
S_\fM= S'_\fM  + O(P^{n-N/2}) = S'_\fM  + O(P^{n-5}), 
\end{equation}
say, where
\[
S'_\fM := \sum_{1\leq q\leq P^\Delta} \int_{|z|\leq P^{-4+\dl}}S(q,z)\,dz.
\]
Here and throughout the paper, we adopt the convention that when the quantity $P^{-N}$ appears in an estimate, that estimate is asserted for arbitrary positive integers $N$, and the implied constant is allowed to depend on $N$ without mention. We may use the results from \cite{Browning-Heath-Brown09} to control this contribution. 

To this end, we  recall the standard definition of the \emph{singular series}
\begin{equation}
\label{eq:singular_series}
\fS = \lim_{R \to \infty} \fS(R), \quad \text{where} \quad  \fS(R) = \sum_{q \leq R} \frac{1}{q^n} \starsum_{a=1}^q \sum_{\x \bmod q} e_q(aF(\x)),
\end{equation}
and the \emph{singular integral}
\begin{equation}
\label{eq:singular_integral}
\fI = \lim_{R \to \infty} \fI(R), \quad \text{where} \quad \fI(R) = \int_{-R}^R \int_{\RR^n} \omega(\x)e(zF(\x))\,d\x\,dz.
\end{equation}
The contribution from the main term in \eqref{eq:major2} has been studied in \cite[Lemma 23]{Browning-Heath-Brown09}, which we restate here.
\begin{lemma}
\label{lem:BHB23}
Let $n - \sigma \geq 26$. Suppose that $\fS$ is absolutely convergent, and satifies the estimate
\begin{equation}
\label{eq:BHB_10.2}
\fS(R) = \fS + O_{\psi}(R^{-\psi})
\end{equation}
for some $\psi > 0$. Then $\fI$ is absolutely convergent, and for any choice of $\Delta \in (0,1/5)$ there exists $\delta > 0$ such that
\[
S'_\fM = \fS \fI P^{n-4} + O_\psi(P^{n-4-\delta}).
\]
Furthermore, we have $\fI > 0$ provided that $\rho$ is chosen small enough.
\end{lemma}
 The absolute convergence and positivity of $\fS$ for $n - \sigma \geq 26$ is also established in \cite{Browning-Heath-Brown09}. Unfortunately, the power saving asymptotic formula \eqref{eq:BHB_10.2}, which is essential in employing Lemma \ref{lem:BHB23}, is established there only for $n$ in the range $n - \sigma \geq 42$. 
We shall improve upon this range by refining arguments used there and proving Lemma \ref{lem:singular_series} stated below.

Before proceeding with a statement and proof of Lemma \ref{lem:singular_series}, we need a standard counting result, which is established in Lemma \ref{lem:20} below. To state this, we introduce the following notation, which will be adopted in the rest of the 
paper. For any $q\in \NN$, 
and for any $i\in\NN$, let
\begin{equation}
\label{eq:factors}
b_i=\prod_{p^i\mid\mid q}p^i,\,\,\, q_i=\prod_{\substack{p^e\mid\mid q\\e\geq 
i}}p^e.
\end{equation}
Thus we have for example
\[
q = b_1 b_2 q_3,
\]
where the factor $q_3$ is the cube-full part and $b_1$ the square-free part of $q$. 
\begin{lemma}
\label{lem:20}
For any positive real numbers $R_1,\dotsc,R_{\ell}$ we have the bound
\[
\sum_{\substack{b_1 \sim R_1, \dotsc,b_{\ell-1} \sim R_{\ell-1}\\
 q_\ell \sim R_\ell}} 1 \ll \prod_{i=1}^{\ell} R_i^{1/i}.
\]
\end{lemma}
The proof is standard and similar to that of \cite[Lemma 20]{Browning-Heath-Brown09}, and we omit it here. Equipped with this result, we are now ready to state and prove Lemma \ref{lem:singular_series}.

\begin{lemma}
\label{lem:singular_series}
If $n-\sigma \geq 26$, then the estimate 
\[
\fS(R) = \fS + O_{\psi}(R^{-\psi})
\]
holds for any $\psi \in (0,1/24)$. 
\end{lemma} 

\begin{proof}
The proof will follow along similar lines to that of \cite[Theorem 2]{Browning-Heath-Brown09}. The key idea here is a more refined treatment of the exponential sums arising in \eqref{eq:singular_series}, when $q$ is free of any $24$-th power, as established by bounds in \cite[Lemmas 7 and 25]{Browning-Heath-Brown09}. Fix a natural number $\ell \geq 3$ and write $q = b_1b_2\dotsb b_{\ell-1}q_\ell$. Using \cite[Lemmas 7 and 25]{Browning-Heath-Brown09} along with bound \cite[(6.12)]{Browning-Heath-Brown09}, one then has
\begin{align*}
 \left|\starsum_{a=1}^q \sum_{\x \bmod q} e_q(aF(\x))\right| &\ll q^{1+\ve} (b_1b_2)^{(n+\sigma+1)/2} b_3^{(2n +\sigma+1)/3} \dotsb b_{\ell-1}^{((\ell-2)n+\sigma+1)/(\ell-1)} q_\ell^{(23n + \sigma+1)/24} \\
&= \frac{q^{n+1+\ve}}{(b_1b_2)^{m/2}b_3^{m/3}\dotsb b_{\ell-1}^{m/(\ell-1)} q_\ell^{m/24}},
\end{align*}
with $m:= n-\sigma-1$. It follows that
\begin{align*}
|\fS - \fS(R)| &\leq \sum_{q \geq R} q^{-n} |S_q| \ll \sum_{b_1\dotsb b_{\ell-1}q_\ell \geq R} (b_1b_2)^{-(m/2-1-\ve)} b_3^{-(m/3-1-\ve)} \dotsb q_\ell^{-(m/24-1-\ve)}.
\end{align*}
We put $\ell = 24$ and note that for any $3 \leq k \leq 24$ one has
\[
-(\tfrac{m}{k}- 1) \leq -\tfrac{1}{k}
\]
as soon as $m \geq 25$. This gives
\begin{align*}
|\fS - \fS(R)| &\ll R^{-1/24+2\ve} \sum_{b_1\dotsb b_{23} q_{24} \geq R} (b_1b_2)^{-1-\ve} b_3^{-1/3-\ve} \dotsb b_{23}^{-1/23 -\ve} q_{24}^{-1/24-\ve} \\
&\ll R^{-1/24+2\ve} \sum_{b_1,\dotsb,b_{23},q_{24} = 1}^{\infty} (b_1b_2)^{-1-\ve} b_3^{-1/3-\ve} \dotsb b_{23}^{-1/23 -\ve} q_{24}^{-1/24-\ve}.
\end{align*}
Using a dyadic decomposition along with Lemma \ref{lem:20}, one now concludes  that the sum on the right hand side is convergent, so that the right hand side is $O_\ve(R^{-1/24 + \ve})$, as required. This proves Lemma \ref{lem:singular_series}. 
\end{proof}
Combining Lemma \ref{lem:BHB23} with \eqref{eq:major2}, we thus have
\begin{equation}
\label{eq:major3}
S_\fM = c_F P^{n-4} + O(P^{n-4-\delta}),
\end{equation}
where $c_F$ depends on the parameter $\rho$, and $c_F > 0$ if $\rho$ is chosen to be small enough, assuming that $X_{\mathrm{ns}}(\AA_\QQ) \neq \emptyset$. This settles our analysis of the contribution from the major arcs.

\subsection{Minor arcs contribution} The main part of the paper will be devoted to showing that the remaining ranges for $q$, $z$ give a negligible contribution to \eqref{eq:kloosterman}. More specifically, we consider the minor arcs contribution to Proposition \ref{prop:delta_counting}:
\begin{equation}
\label{eq:minor}
\begin{split}
S_\fm &=\sum_{1\leq q\leq P^\dl}\int_{P^{-4+\dl} \leq |z| \leq (qQ)^{-1+\theta}}
|p_q(z) S(q,z)|dz \\
&+\sum_{P^{\Delta}\leq q\leq Q}\int_{|z| \leq (qQ)^{-1+\theta}}
|p_q(z) S(q,z)|dz.
\end{split}
\end{equation}
It should be noted that the choice of minor arcs depend on the parameter $\theta>0$, which arises in \eqref{eq:kloosterman}, as well as the choice of $\Delta>0$. We will sometimes write $S_\fm = S_{\fm,\theta,\Delta}$ to emphasize this dependence. Our main bound for the minor arc contribution will take the following form.
\begin{proposition}
\label{prop:minor}
Let $F$ be a quartic form satisfying $n-\sigma \geq 31$.
Then there is a choice of $\rho \in (0,1]$ and, for any $0<\Delta< 1/5$, there exists $\phi_0>0$, such that for any $0<\phi<\phi_0$, one has 
\[
S_{\fm,\theta,\Delta} = O_{n,\Delta,\theta,\phi,\|F\|}(P^{n-4-\delta}).
\]
for some $\delta = \delta(\Delta,n,\phi) >0$, provided that $\theta \ll_{n,\Delta} 1$. 
Here 
\begin{equation}
\|F\|=\textrm{the maximum modulus of the coefficients of }F .
\end{equation}
\end{proposition}
Here and in the following propositions, $\rho$ is as in \eqref{eq:W def}, and $\phi$ is as in \eqref{eq:phidef} and \eqref{eq:narrow}. Fixing $\Delta$ and taking $\theta$ small enough that the conclusion of Proposition \ref{prop:minor} is valid, 
we put together the bounds in \eqref{eq:kloosterman}, \eqref{eq:major3} and Proposition \ref{prop:minor} to obtain Theorem \ref{thm:thm 3}. Thus, from now on, it is enough to concentrate on proving Proposition \ref{prop:minor}.

\subsection{A more general minor arcs bound}
\label{sec:minor_general}

The arguments used to bound the minor arcs contribution will not depend on the fact that $F$ is homogeneous. We shall thus deduce the bound in Proposition \ref{prop:minor} in a less restrictive setting. For a general polynomial $f \in \ZZ[x_1,\dotsc,x_n]$, not necessarily homogeneous, we introduce the alternative ``height function"
\[
\Vert f \Vert_P := \Vert P^{-\deg(f)} f(Px_1,\dotsc,Px_n) \Vert.
\]
Suppose now that $F \in \ZZ[x_1,\dotsc,x_n]$ is a quartic polynomial, not necessarily homogeneous. Let $F_0$ be its leading form, defining a quartic hypersurface $X_0 \subseteq \PP^{n-1}_\QQ$. Rather than fixing the weight function in \eqref{eq:W def}, we shall obtain a uniform estimate for an entire class $\cW_n$ of weight functions.
Given  a positive real number $c$ and a sequence $(c_j)_{j=0}^\infty$ of positive real numbers, we let $\cW_n = \cW_n(\underline{\c})$, where $\underline{\c}= (c,(c_j)_j)$ for short, be the set of infinitely differentiable functions $W:\RR^n \to \RR_{\geq 0}$ with support inside $[-c,c]^n$ that satisfy
\[
\max \left\{\left|\frac{\partial^{j_1+\dotsb+j_n}}{\partial x_1^{j_1} \dotsb \partial x_n^{j_n}}W(\x) \right|\mid \x \in \RR^n, j_1+\dotsb+j_n = j\right\} \leq c_j
\]
for all $j \geq 0$. In the sequel, we shall often suppress the dependence on $\underline{\c}$ in our estimates. 

We shall assume that there is a constant $M>0$ such that the following properties hold:
\begin{gather}
\label{eq:M_1}
\min_{\x \in P \supp(W)} \left|\partial_{x_1} F(\x)\right| \geq M P^3.\\
\label{eq:M_2}
\max_i \max_{\x_1,\x_2 \in P\supp(W)} |\partial_{x_i} F(\x_2) - \partial_{x_i} F(\x_1)| \leq \frac{M}{8^n \sqrt{n!}} P^3.
\end{gather}

\begin{proposition}
\label{prop:minor_inhom}
Let $F$ be a quartic polynomial and suppose that $n-\dim \Sing(X_0) \geq 31$, where $X_0 \subseteq \PP^{n-1}_\QQ$ is the hypersurface defined by the leading quartic form $F_0$. Let $W \in \cW_n$ and assume further that \eqref{eq:M_1}--\eqref{eq:M_2} hold.
Then, for any $0 < \Delta < 1/5$, there exists $\phi_0>0$ such that for any $0<\phi\leq \phi_0$,
\[
S_\fm = O_{n,\Delta,\theta, \phi,\|F\|_P,\underline{\c}}(P^{n-4-\psi})
\]
for some $\psi = \psi(n,\Delta,\phi)> 0$, provided that $\theta \ll_{n,\Delta,\phi} 1$. \end{proposition}

Let us verify that Proposition \ref{prop:minor} follows from Proposition \ref{prop:minor_inhom}. In the former, $F$ is assumed to be homogeneous, which implies that $\Vert F \Vert_P = \Vert F \Vert$. Thus it only remains to verify the conditions \eqref{eq:M_1}--\eqref{eq:M_2}. By assumption, we have $M_0=|\partial_{x_1} F(\x_0)| > 0$. By choosing  $\rho$ sufficiently small, we may ensure that none of the $\partial_{x_i} F$ vary more than $M_0/(2\cdot 8^n \sqrt{n!})$ on $\supp(\omega)$. Since $F$ is now a homogenenous polynomial by assumption, so are its derivatives $\partial_{x_i} F$. Therefore, for arbitrary $\y_1,\y_2 \in \supp(\omega)$ we have
\[
|\partial_{x_i} F(P\y_1) - \partial_{x_i} F(P\y_2)| = P^3 |\partial_{x_i} F(\y_1) - \partial_{x_i} F(\y_2)| \leq \frac{M_0}{2\cdot 8^n \sqrt{n!}} P^3.
\]
Furthermore, we have
\begin{align*}
|\partial_{x_1} F(P\y_1)| &\geq |\partial_{x_1} F(P\x_0)| - P^3|\partial_{x_1} F(\x_0) - \partial_{x_1} F(\y_1)| \\
&\geq M_0 P^3 - \frac{M_0}{2\cdot 8^n \sqrt{n!}} P^3    
\geq \frac{M_0}{2} P^3,
\end{align*}
so we have verified \eqref{eq:M_1}--\eqref{eq:M_2} with $M = M_0/2$, as required. 

We shall prove Proposition \ref{prop:minor_inhom} by induction on $\dim \Sing(X_0)$, the base step being the following  result for the non-singular case.

\begin{proposition}
\label{prop:minor_smooth}
Suppose that $n\geq 30$ and that the leading quartic form $F_0$ is non-singular. Let $W \in \cW_n$ and assume that
\eqref{eq:M_1}--\eqref{eq:M_2} hold. Then, for any $0 < \Delta < 1/5$, there exists $\phi_0>0$ such that for any $0<\phi\leq \phi_0$,
\[
S_\fm = O_{n,\Delta,\phi,\|F\|_P,\underline{\c}}(P^{n-4-\psi})
\]
for some $\psi = \psi(n,\Delta,\phi)> 0$, provided that $\theta \ll_{n,\Delta,\phi} 1$.
\end{proposition}

Our efforts in the Sections \ref{sec:vdC} through \ref{sec:vdC2} will culminate in the verification of Proposition \ref{prop:minor_smooth} in Section \ref{sec:minor}. The inductive argument leading to Proposition \ref{prop:minor_inhom} will be postponed to Section \ref{sec:slicing}, as it is similar to another slicing argument that we shall need to apply in the non-singular case as well.

In studying the minor arcs contribution, it will be convenient to split the integrals in \eqref{eq:minor} into 
suitable dyadic intervals. This allows us to consider the integrals
\begin{align}
\label{eq:Iqtdef}
\cI(q,t)= \int_{t \leq |z| \leq 2t} |S(q,z)|dz,
\end{align}
where $1\leq q\leq Q$, $0\leq t\leq (qQ)^{-1+\theta}$. Note that we can trivially establish $|\cI(q,t)|\ll qtP^n$. However, with some more work akin to Proposition \ref{prop:trivial}, using cancellations due to an average over $a$ in the definition \eqref{Su} of $S(q,z)$, for a large enough value of $P$, it is also possible to establish the bound $b_1^{1/2}q_2tP^n$, where $q=b_1q_2$, with $b_1$ being the square-free part of $q$. For any fixed choice of $t,R$ appearing in the minor arcs \eqref{eq:minor}, our final goal is to establish the bound
\begin{equation}\label{eq:I bound}
\sum_{R\leq q\leq 2R}\cI(q,t)\ll P^{n-4-\ve},
\end{equation}
which will be finally achieved in Section \ref{sec:minor} using various estimates we derive in the following sections.

\section{Van der Corput differencing}
\label{sec:vdC}
Van der Corput differencing provides a key tool in our analysis of bounding \eqref{eq:Iqtdef}. We will use both pointwise van der Corput differencing method in the spirit of \cite{Browning-Heath-Brown09} and  
averaged van der 
Corput differencing employed by Heath-Brown \cite{Heath-Brown07} and Hanselmann \cite{Hanselmann}, the latter being more advantageous when $t$ is not too small. A key advantage in using averaged van der Corput differencing is in the introduction of smoothed Gaussian averages in \eqref{eq:Mq}, which readily helps us establish Lemma \ref{lem:negligible}. 

We begin with a brief survey of van der Corput differencing. 
Let $f$ be a function on 
$\RR^n$ supported in the set $|\x|\ll P$ and let $\cH$ be a subset of 
$\ZZ^n\cap 
\{|\x|\ll P\}$. The starting point of van der Corput 
differencing is the following identity:
\begin{align*} 
\#\cH\sum_{\x\in\ZZ^n}f(\x)=\sum_{\h\in\cH}\sum_{\x\in\ZZ^n}f(\x+\h)=\sum_{
\x\in\ZZ^n}\sum_{\h\in\cH}f(\x+\h).
\end{align*}
A quick application of the Cauchy-Schwarz inequality implies
\begin{align}\label{eq:van der start}
\#\cH^2\left|\sum_{\x\in\ZZ^n}f(\x)\right|^2\ll 
P^n\sum_{\h}N(\h)\sum_{\y\in\ZZ^n}f(\y+\h)\overline{f(\y)},
\end{align}
where $$N(\h):=\#\{\h_1,\h_2\in\cH: \h=\h_1-\h_2\}, $$
and the sum over $\h$ is over $\h\in\ZZ^n$ such that $N(\h)\neq 0$. \eqref{eq:van der start} will denote a starting point for van der Corput differencing process for us. We shall apply it with $f(\x)$ being replaced by suitable exponential sums at various places in this section.

Our main goal is to utilise this process to estimate the sum
\[
S(q,z) = \starsum_{a=1}^q \sum_{\x\in \ZZ^n}W(\x/P)e((a/q+z)
F(\x))
\]
introduced above, where $W \in \cW_n$ and where $F$ is a quartic polynomial, not necessarily homogeneous.

\subsection{Pointwise van der Corput differencing} 
In the notation of \eqref{eq:factors}, we put $q = b_1q_2$, where $b_1$ is square-free and $q_2$ square-full. We plan to benefit from an extra averaging over the sum over $a\bmod{q}$ occurring in $S(q,z)$, achieved for a fixed value of $z$. However, this saving is essential only from the square-free part of $q$. To this end, we shall sum trivially over the square-full part of $q_2$ of $q$:
\begin{align}
\label{eq:sum_c}
 |S(q,z)|
&\leq \starsum_{a=1}^{q_2} |S_a(q,z)|,
\end{align}
where
\begin{equation}\label{eq:sum_c2}
S_a(q,z):= \starsum_{s=1}^{b_1} \sum_{\x\in 
\ZZ^n}W(\x/P)e((s/b_1+a/q_2+z)F(\x)).
\end{equation}
We now apply van der Corput differencing \eqref{eq:van der start} to the 
function 
$$f(\x)=\starsum_{s=1}^{b_1} W(\x/P)e((s/b_1+a/q_2+z)F(\x)), $$ to bound $|S_a(q,z)|$ 
for any fixed $a$:
\begin{align}
|S_a(q,z)|^2&\ll 
\#\cH^{-2}P^n\sum_{\h}N(\h)\sum_{\y\in\ZZ^n}f(\y+\h)\overline{f(\y)}
= \#\cH^{-2}P^n\sum_{\h}N(\h)\cT_{a,\h}(q,z)
\label{eq:vdc}
,\end{align}
where
\begin{multline}
\label{eq:Tuhdef}
\cT_{a,\h}(q,z):= \sum_{\x \in \Z^n} 
W_\h(\x/P)\starsum_{s_1,s_2=1}^{b_1} 
e_{b_1}\big(s_1 F_\h(\x) 
+(s_1-s_2)F(\x)\big) 
e_{q_2}\big(a F_\h(\x)\big)e(zF_\h(\x)), 
\end{multline}
with
$$W_\h(\y)=W(\y+P^{-1}\h)W(\y) \quad\text{ and }\quad F_\h(\x)=F(\x+\h)-F(\x). 
$$
Note that we may trivially bound $|\cT_{a,\h}(q,z)|\ll P^nb^2$.
We now let $H = H(q,z)$ be an integer with $ 1 \leq H \leq P$ and put
\begin{equation}
\cH:=\{\h\in\NN^n \mid 1\leq h_1,...,h_n\leq H\}.\label{eq:H_pointwise} 
\end{equation}
Then the bounds \eqref{eq:sum_c} and \eqref{eq:vdc}, along with the trivial bound $N(\h)\ll \#\cH$, imply the 
following result, which we refer to as the \emph{pointwise van der Corput 
bound} with respect to $z$.
\begin{lemma}
\label{lem:pointwise}
We have the bound
\[
|S(q,z)| \ll P^{n/2} H^{-n/2} \starsum_{a=1}^{q_2} \left(\sum_{|\h| \leq H} 
\left|\cT_{a,\h}(q,z)\right| \right)^{1/2}.
\]
\end{lemma} 
\begin{rem}A key advantage of step \eqref{eq:sum_c} is that after an application of Poisson summation, an overall factor of size $O(b_1^{1/2})$ can be saved via considerations of the resulting quartic exponential sums modulo the square-free part $b_1$ of $q$.  In the square-full case however, our sums coincide with the cubic sums in \cite{Browning-Heath-Brown09}, allowing us to reuse those bounds. See Lemmas \ref{lem:poisson} and \ref{lem:mult} in the following section for a more explicit version. This comment also applies to our bounds in Section \ref{Sec:Averaged} below. This significantly simplifies our work. Unfortunately, this treatment will not be enough to achieve $n=29$ variables, and some early calculations show that a small extra saving in the sum over $a$ in \eqref{eq:sum_c} will be vital to save one more variable.
\end{rem}
\subsection{Averaged van der Corput differencing}
\label{Sec:Averaged}
We next incorporate an averaging over $z$ to estimate the integrals $\cI(q,t)$ defined in \eqref{eq:Iqtdef}. 
At this stage, we shall impose the conditions \eqref{eq:M_1}--\eqref{eq:M_2} occuring in Proposition \ref{prop:minor_inhom}. Choose an arbitrary point $\x \in P\supp(W)$. By \eqref{eq:M_1}, $$|\partial_{x_1} F(\x)| \geq MP^3.$$ 
By \eqref{eq:sum_c}, we have
\begin{align*}
\cI(q,t) &\ll 
\starsum_{a=1}^{q_2}\int_{t\leq |z|\leq 2t} |S_a(q,z)|dz.
\end{align*}
Let $H = H(q,t)$ be an integer with $1 \leq H \leq P$. We may choose a set $T$ of cardinality $O(1+ tHP^3)$ such that
\[
[t,2t]\cup [-2t,-t] \subseteq \bigcup_{\tau \in T} \left[\tau-(HP^3)^{-1},\tau+(HP^3)^{-1}\right]. 
\]
Then we can write
\[
\cI(q,t) \ll 
\starsum_{a=1}^{q_2} \sum_{\tau \in T} \int_{|\tau-z|<(HP^3)^{-1}}|S_a(q,z)|dz. 
\]
An application of the Cauchy-Schwarz inequality next implies
\begin{align*}
\int_{|\tau-z|<(HP^3)^{-1}}|S_a(q,z)|dz&\ll(HP^3)^{-1/2}\left(\int_{|\tau-z|<(HP^3)^{-1}}|S_a(q,z)|^2\,
dz\right)^{1/2}.
\end{align*}
Inserting this into the above bounds, we may write
\begin{align*}
\cI(q,t) &\ll (HP^3)^{-1/2}
\starsum_{a=1}^{q_2} \sum_{\tau \in T} \cM_a(q,\tau)^{1/2},
\end{align*}
where
\begin{equation}
\label{eq:Mq}
\cM_a(q,\tau):=\int_{|\tau-z|<(HP^3)^{-1}}|S_a(q,z)|^2 dz \ll \int_{-\infty}^\infty\exp(-H^2P^6(\tau-z)^2)|S_a(q,z)|^2\,dz.
\end{equation}

We now employ van der Corput differencing \eqref{eq:vdc} with a modified 
differencing set. Let $1\geq c_1 > 0$ be a constant to be determined later, and put 
\begin{equation}
\cH_1:=\{\h\in\NN^n \mid 1\leq h_1\leq c_1 P,1\leq h_2,...,h_n\leq 
H\}.
\label{eq:H_average} 
\end{equation} 
to obtain
\begin{align}
\cM_a(q,\tau)&\ll \#\cH_1^{-1}P^n\sum_{\h \in \cH_1} 
\int_{-\infty}^\infty \exp(-H^2P^6(\tau-z)^2)\cT_{a,\h}(q,z)\, dz
\label{eq:vdc2}
,\end{align}
with $\cT_{a,\h}$ as in \eqref{eq:Tuhdef}. We have again used the trivial bound $N(\h)\ll \#\cH_1$ here. We now set $\cL:=\log P$ and use the lower bound $|\partial_{x_1}(F(\x))|\geq MP^3$ for all $\x\in P\supp(W)$, to obtain that the 
contribution in 
the above sum from the terms corresponding to $\h$ satisfying $|h_1|\gg H\cL$ is 
negligible, as derived in the following lemma.

\begin{lemma}
\label{lem:negligible}
Suppose that \eqref{eq:M_1} holds. Then there is a constant $c>0$ such that
for any $\h \in \cH_1$ such that $|h_1| \geq cH\cL$, 
we have
$$ \int_{-\infty}^\infty \exp(-H^2P^6(\tau-z)^2)\cT_{a,\h}(q,z)\, dz\ll_N P^{-N}. $$
\end{lemma}

\begin{proof}
We follow the strategy in \cite[Section 4]{Heath-Brown07} and \cite[Section 5]{Hanselmann}. The parameter $a$ is considered fixed for now.
Using \eqref{eq:Tuhdef} we have 
\begin{equation}
\label{eq:Ihx}
\int_{-\infty}^\infty \exp(-H^2P^6(\tau-z)^2)\cT_{a,\h}(q,z)\, dz = \sum_{\x \in \Z^n} 
W_\h(\x/P) I(\h,\x),
\end{equation}
where
\begin{multline*}
I(\h,\x) = \starsum_{s_1,s_2=1}^{b_1} 
e_{b_1}\big(s_1 F_\h(\x) 
+(s_1-s_2)F(\x)\big) e_{q_2}\big(a F_\h(\x)\big) \\
\times \int_{-\infty}^\infty \exp(-H^2P^6(\tau-z)^2)e(zF_\h(\x)) \, dz.
\end{multline*}
Calculating the inner integral explicitly, we may write
\begin{multline*}
I(\h,\x) = \frac{\sqrt{\pi}}{HP^3} \exp\left(-\frac{\pi^2}{H^2P^6}|F_\h(\x)|^2\right) e(-\tau F_\h(\x)) \\
\times \starsum_{s_1,s_2=1}^{b_1} 
e_{b}\big(s_1 F_\h(\x) 
+(s_1-s_2)F(\x)\big) e_{q_2}\big(a F_\h(\x)\big).
\end{multline*}
We now claim that
$c$ can be chosen in such a way that if $\h \in \cH_1$ satisfies $|h_1| \geq cH\cL$, then we have
\begin{equation}
\label{eq:Fhbound}
|F_\h(\x)| \gg HP^3\cL
\end{equation}
for any $\x$ occurring in \eqref{eq:Ihx} with $W_\h(\x/P) \neq 0$. If this is true, then for such values of $\h$ and $\x$ we have
\[
|I(\h,\x)| \ll \frac{q^2}{HP^3} \exp(-c'(\log P)^2) 
\]
for some constant $c'>0$, so the quantity in \eqref{eq:Ihx} is
\[
\ll P^{n+1}  \exp(-c'(\log P)^2) \ll P^{-N},
\]
as required.

It remains to verify the above claim. For this we write
\[
F_\h(\x) = \frac{\partial F}{\partial x_1}(\x)\cdot h_1 + O(HP^3) + O(h_1^2P^2),
\]
where the implied constants depend on $\Vert F \Vert_P$ only (but not on $\Vert F \Vert$).
By \eqref{eq:M_1}, we have $|\partial_{x_1} F(\x)| \geq MP^3$, yielding
\[
|F_\h(\x)| \geq |h_1|P^2\left(
MP - O(|h_1|)\right) - O(HP^3).
\]
We may thus choose $c_1$ in the definition \eqref{eq:H_average} so that
\[
|F_\h(\x)| \geq \frac{|h_1|MP^3}{2} - O(HP^3),
\]
say, for relevant values of $\x$. This clearly gives the desired bound \eqref{eq:Fhbound} as soon as $|h_1| \gg H\cL$.
\end{proof}

Using Lemma \ref{lem:negligible}, and noting that the contribution from the range $|\tau - z| \geq \cL (HP^3)^{-1}$ is also $O(P^{-N})$, we get the following bound for \eqref{eq:Mq}:
\begin{align*}
\cM_a(q,\tau) &\ll P^{-N} + \frac{P^{n-1}}{H^{n-1}} \int_{|\tau-z|< \cL(HP^3)^{-1}} \left(\sum_{|\h|\leq H\cL} \cT_{a,\h}(q,z)\right) dz \\
&\ll  P^{-N} + \frac{\cL}{HP^3} \frac{P^{n-1}}{H^{n-1}} \max_{|\tau-z|<\cL(HP^3)^{-1}} \sum_{|\h|\leq H\cL} |\cT_{a,\h}(q,z)|.
\end{align*}
Thus we have proved the following result. 
\begin{lemma}
\label{lem:hanselmann}
For any $1 \leq H \leq P$, we have
\begin{equation*}
\cI(q,t) \ll P^{-N} + P^\ve \left(t + \frac{1}{HP^3}\right) \left(\frac{P}{H}\right)^{(n-1)/2}\starsum_{a=1}^{q_2} {\max_{z}}^{(1)} \left\{\sum_{|\h|\leq HP^\ve} |\cT_{a,\h}(q,z)|\right\}^{1/2},
\end{equation*}
where $\max^{(1)}$ is taken over $t \leq |z| \leq \max\left\{2t,t+ \frac{1}{HP^{3-\ve}}\right\}$. 
\end{lemma}

We note here that as a standard feature in applications of van der Corput differencing, $H$ will eventually be chosen in order to obtain the bound \eqref{eq:I bound}. A minimum value of $H$ is necessary to make the contribution from $\h=\0$ work. However, choosing larger values of $H$ increases the contribution from {\em generic} values of $\h$'s. Lemma \ref{lem:hanselmann} will be crucial in dealing with the case when $q,t$ are both large.

 We shall also state a corresponding result that is derived in the vein of the ``classical'' circle method used in \cite{Browning-Heath-Brown09}. Here, a typical minor arcs contribution takes the shape
\[
\cI(q,t) \leq \int_{t \leq |z| \leq 2t} \starsum_{a=1}^q|S(a/q+z)|\,dz.
\]   
An easy modification of the arguments above give us the following estimate, which in essence is contained in Hanselmann's treatment \cite[(5.11)]{Hanselmann}:
\begin{lemma}
\label{lem:hanselmann_classic}
For any $1 \leq H \leq P$, we have
\begin{equation*}
\cI(q,t) \ll P^{-N} + P^\ve \left(t + \frac{1}{HP^3}\right) \left(\frac{P}{H}\right)^{(n-1)/2}\starsum_{a=1}^q{\max_{z}}^{(1)} \left\{\sum_{|\h|\leq HP^\ve} |\cS_{\h}(a/q +z)|\right\}^{1/2},
\end{equation*}
where $\max^{(1)}$ is defined as above and 
\begin{equation}
\label{eq:S_h}
\cS_\h(\alpha) := \sum_{\x \in \ZZ^n} w(\x/P)e(\alpha F_\h(\x)).
\end{equation}
\end{lemma}
A main difference here from Lemma \ref{lem:hanselmann} is that the sum \eqref{eq:S_h} is a purely cubic exponential sum only containing the form $F_\h$, as opposed to a mixture of cubic and quartic exponential sums \eqref{eq:Tuhdef}, appearing in Lemma \ref{lem:hanselmann}. In Section \ref{sec:weyl}, we will apply Weyl bound directly to estimate $\cS_\h(\alpha)$, via \cite[Lemma 3.3]{Browning_Prendiville}. This bound will turn out to be important for us. It will be used to supplement bounds from Lemma \ref{lem:hanselmann} and those from pure Weyl differencing.

\section{Evaluation of quartic exponential sums; initial considerations}
\label{sec:expsums_intro}

Our calculations above led us to the consideration of the exponential sums
$\cT_{a,\h}(q,z)$.
Here, we shall investigate a more general version of that sum. Let $f,g \in \Z[x_1,\dotsc,x_n]$ be polynomials and suppose that
\begin{equation}
\label{eq:degrees}
\deg(f) = d_1, \quad \deg(g) = d_2, \quad \text{where} \quad d_1 \geq d_2 \geq 2,
\end{equation}
and that
\[
\max\{\Vert f \Vert_P, 
\Vert g \Vert_P
\} 
\leq H 
\]
for some $H \geq 1$. Given a weight function $w\in \cW_n$ and a fixed integer $a$ relatively prime to $q$,
we define
\begin{multline}
\label{eq:Tuzdef}
\cT(q,z) = \cT_{n}(a,q,z;f,g,w,P) \\
:= \sum_{\x \in \Z^n} 
w(\x/P)\starsum_{s_1,s_2=1}^{b_1} 
e_{b_1}\big(s_1 g(\x) 
+(s_1-s_2)f(\x)\big)
e_{q_2}\big(a g(\x)\big)e(zg(\x)).
\end{multline}
For the result in this paper, we would only need to this result with $d_1=4,d_2=3$. However, we consider general degrees $d_1$ and $d_2$ since these bounds will be useful in bounding the contribution from the square-free part of $q$ in subsequent applications to higher degree forms. We can also obtain them here without too much extra work. From Section \ref{sec:cubefull} onwards however, we will specialise to the case $d_1=4,d_2=3$.
\subsection{Hyperplane intersection: main lemma} Here, we will deviate slightly from our need for bounding \eqref{eq:Tuzdef}, to consider a  question of handling singularities of an arbitrary finite set of varieties, via hyperplane intersections. This is achieved by Lemma \ref{lem:subspace} below. In this paper, it will only be necessary to simultaneously bound singular locus of a system of a cubic and a quartic form and of their complete intersection. However, a general result can be obtained with not much more extra work. This lemma will be instrumental in establishing an important bound obtained in Proposition \ref{prop:sum_h}. Moreover, the extra condition that $\e_1$ can be chosen in a certain way, mentioned at the end of Lemma \ref{lem:subspace} will be useful in Section \ref{sec:slicing}.

To this end, we introduce some notation, which will be used throughout this paper. We use the symbol $v$ to denote a place of $\QQ$, that is, $v = \infty$ or $v = p$ for a prime $p$. 
Given the forms $F_1,\dotsc,F_m \in \ZZ[x_1,\dotsc,x_n]$, we then denote by $V_v(F_1,\dotsc,F_m)$ the closed subvariety of 
$\PP^{n-1}_{\FF_v}$ defined by the image of the ideal $\langle F_1,\dotsc,F_m \rangle$ in $\FF_v[x_1,\dotsc,x_n]$, and 
let
\begin{align*}
s_v(F_1,\dotsc,F_m) &:= \dim (\Sing(V_v(F_1,\dotsc,F_m))), \\
\delta_v(F_1,\dotsc,F_m) &:= \dim(V_v(F_1,\dotsc,F_m)). 
\end{align*}

We will consider the aim of simultaneously reducing the dimension of the singular locus of a system of forms $F_1,...,F_m$, their complete intersection, and the complete intersection of any sub-collection $\{F_{i}:i\in I\}$ for any subset $I\subseteq \{1,...,m\}$. Given a collection of primes $\Pi$, Lemma \ref{lem:subspace} presents us with a {\em nice} lattice basis, which in turn give us an entire chain of subspaces which successively achieve this aim.

\begin{lemma}
\label{lem:subspace}
Let $F_1,\dotsc,F_m$ be forms of degree $d_1,\dotsc,d_m$ defining a complete intersection $X = V(F_1,\dotsc,F_m) \subset \PP^{n-1}_{\QQ}$ of dimension $n-1-m$. Let $\Pi$ be a collection of primes, with $\#\Pi = r \geq 0$, and write $\Pi_a := \{p \in \Pi \mid p > a\}$ for each $a \in \NN$. 
There is a constant $c = c(n,d_1,\dotsc,d_m)$ and a collection of primitive linearly independent integer vectors
\[
\e_1,\dotsc,\e_{n} \in \ZZ^n
\]
satisfying the following property for any integer $0 \leq \eta \leq n-1$, any subset $\varnothing \neq I \subseteq \{1,\dotsc,m\}$ and any $v \in \{\infty\} \cup \Pi_{cr}$:
\begin{itemize}
\item[$(\star)_{I,\eta,v}$]
Provided that the closed subscheme $X_I \subset \PP^{n-1}_\ZZ$ defined by $F_i = 0$ for all $i \in I$ satisfies $\dim (X_I)_v = n-1-|I|$, the subspace $\Lambda_\eta \subset \PP^{n-1}_{\FF_v}$ spanned by the images of $\e_1,\dotsc,\e_{n-\eta}$ is such that
\[
\dim (X_I \cap \Lambda_\eta)_v  = \max\{-1,\dim (X_I)_v -\eta\} 
\]
and
\[
\dim \Sing\big((X_I \cap \Lambda_\eta)_v \big) = \max\{-1,\dim \Sing\big((X_I)_v \big)-\eta\}. 
\]
\end{itemize}
Moreover, the basis vectors $\e_i$ may be chosen so that 
\begin{equation}
\label{eq:basis1}
L/2 \leq |\e_i| \leq L
\end{equation}
for all $i=1,\dotsc,n$ and
\begin{equation}
\label{eq:basis2}
L^n \ll \det(\e_1,\dotsc,\e_n) \ll L^n
\end{equation}
for a constant $L =O_{n,d_1,\dotsc,d_m}(r+1)$.

In addition, given any vector $\f\in \RR^n$, vector $\e_1$ in the basis above can be chosen such that it makes an angle of at most $\pi/3$ with $\f$.
\end{lemma}

\begin{proof}
First we give a short outline of the proof. We shall represent the collection of vectors $\e_i$ by the matrix $E$ having the $\e_i$ as columns. There are now two main points to our argument. First we argue that each condition $(\star)_{I,\eta,v}$ is implied by the non-vanishing of some polynomial, over $\FF_v$, in the entries $e_{i,j}$. On the other hand, a positive proportion of all matrices $E$ correspond to bases with the desired properties \eqref{eq:basis1}--\eqref{eq:basis2}. 

We begin by noting that the case $r=0$ follows from the case $r=1$, so we may assume that $r\geq 1$, so that $\Pi$ is non-empty. For fixed $I$, $\eta$ and $v$, it is a consequence of Bertini's theorem, as observed by Ghorpade and Lachaud \cite[Prop. 1.3]{Ghorpade}, that the set $\cU_{\eta,I}$ of subspaces $\Lambda$ satisfying the condition $(\star)_{I,\eta,v}$ contains a Zariski-open subset of the Grassmannian $\GG_{n-\eta} := \GG(n-1-\eta,n-1)$. (Note that for those places $v$ where the forms $F_i$, $i \in I$, do not intersect completely, the condition is void.) Recall that $\GG_{n-\eta}$ is a closed subvariety of $\PP^{N}_{\FF_v}$ for a certain $N = N(n,\eta)$ by virtue of the Pl\"ucker embedding. By an argument very similar to the proof of \cite[Lemma 2.8]{Marmon08}, one sees that one may in fact find a hypersurface $\cZ_{\eta,I}\subset \PP^N$ of degree $O_{n,d_1,\dotsc,d_r}(1)$ independent of $v$, such that $\cU_{\eta,I}$ contains the complement of $\cZ_{\eta,I}$ in $\GG_{n-\eta}$. Consequently, the complement of the hypersurface $\cZ_\eta = \bigcup_I \cZ_{\eta,I}$ is contained in $\cU_{\eta} = \bigcap_I \cU_{\eta,I}$. The hypersurface $\cZ_\eta$ embeds as a hypersurface in $\GG_{n-\eta} \times \GG_\eta$ under the natural map $\GG_{n-\eta} \hookrightarrow \GG_{n-\eta} \times \GG_\eta$. 

Define $\cM(B)$, for any $B >0$, to be the set of matrices $E = (e_{i,j}) \in M_{n\times n}(\RR)$ such that each column vector $\e_j = (e_{1,j},\dotsc,e_{n,j})$ satisfies 
$B/2\leq \Vert \e_j \Vert \leq B$. (Here, we use $\Vert \cdot \Vert$ to denote the usual Euclidean norm on $\RR^n$.) For any $\psi \in (0,\pi/2]$, let $\cM_\psi(B)$ be the set of $E \in \cM(B)$ such that each column vector $\e_j$ makes an angle of at least $\psi$ with the hyperplane spanned by the remaining vectors $\e_{k}$, $k\neq j$. It is clear that the set $\cM_\psi(B)$ has a well-defined and positive volume, that
\[
\vol(\cM(B)) = B^{n^2} \vol(\cM) \qquad \text{and} \qquad \vol(\cM_\psi(B)) = B^{n^2} \vol(\cM_\psi),
\]
where $\cM :=\cM(1)$ and $\cM_\psi := \cM_\psi(1)$, and that
\[
\vol(\cM_\psi) \geq \vol(\cM) - n \vol(\cN_\psi),
\] 
say, where $\cN_\psi$ is the set of $E \in \cM$ where the vector $\e_n$ is either the zero vector or makes an angle of less than $\psi$ with the hyperplane spanned by the first $n-1$ column vectors. Now one has
\[
\vol(\cN_\psi) \ll_n \psi \vol(\cM),
\] 
so choosing $\psi$ small enough, we may conclude that there exists a constant $c_1 = c_1(n)$ such that 
\[
\vol(\cM_\psi) \geq c_1 \vol(\cM).
\]
If we now put $M(B) = \cM_\psi(B) \cap M_{n\times n}(\ZZ)$, then it follows from the Lipschitz principle \cite{Davenport_Lipschitz} that 
\begin{equation}
\label{eq:lipschitz}
\#M(B) = B^{n^2} \vol(\cM_\psi) + O_n(B^{n^2-1}).
\end{equation}
Letting $M_\ast(B)$ denote the set of matrices $E \in M(B)$ for which $\gcd(e_{1,j},\dotsc,e_{n,j}) = 1$ for $j = 1,\dotsc,n$, one may deduce from \eqref{eq:lipschitz}, using standard M\"obius inversion arguments, that
\[
\#M_\ast(B) = B^{n^2} \frac{\vol(\cM_\psi)}{\zeta(n)^n} + O_n(B^{n^2-1}).
\]
We conclude that there is a constant  $C_1 = C_1(n)$ such that 
\[
\#M_\ast(B) \geq C_1 B^{n^2}.
\]
Moreover, it is evident that each matrix in $M_\ast(B)$ satisfies \eqref{eq:basis1} and \eqref{eq:basis2}.

Let $\cW$ be the set of matrices $E \in M_{n\times n}(\ZZ)$ with entries in $[-B,B]$ for which the column vectors fail the condition $(\star)_{I,\eta,v}$ for some choice of $I$, $\eta$ and $v \in \{\infty\} \cup \Pi_a$, where $a$ is yet to be chosen. Our aim will be to show that $\cW$ has cardinality strictly less than $C_1 B^{n^2}$, implying the existence of 
a basis with the desired property.

Fix a place $v$, and let us identify $n \times n$-matrices with points in $\AA^{n^2}$. Let $S$ be the closed subset of $\AA^{n^2}$ defined by $\det(E) = 0$. For each $\eta$, an appropriate Pl\"ucker map $\Phi_\eta: \AA^{n^2}\setminus S \to \GG_{n-\eta} \times \GG_{\eta}$ maps $\AA^{n^2}\setminus S$ onto an affine open subset of $\GG_{n-\eta} \times \GG_{\eta}$, and we may define $W = W_v$ to be the union of the closures in $\AA^{n^2}$ of the inverse images $\Phi_\eta^{-1}(\cZ_\eta)$ of the subvarieties $\cZ_\eta$ introduced above. Then a matrix $E$ lies in $\cW$ only if either $E \in W_\infty$ or $[E]_p \in W_p(\FF_p)$ for some $p \in \Pi_{a}$, where $[E]_p$ denotes the matrix where each entry is the reduction (mod $p$) of the corresponding entry of $E$.
By \cite[Lemma 4]{Browning-Heath-Brown09}, the number of matrices $E$ such that $|e_{i,j}| \leq B$ for all $i,j$, and such that either $E \in W_\infty$ or $[E]_p \in W_p(\FF_p)$ for a prime $p$ is
\[
\leq C_2( B^{n^2} p^{-1} + B^{n^2-1})
\]
for some constant $C_2  = C_2(n,d_1,\dotsc,d_m)$. The number of matrices $E$ satisfying this condition for at least one prime $p \in \Pi_a$ is thus
\[
\leq C_2 \left( B^n \sum_{p \in \Pi_a} p^{-1} + r B^{n-1} \right) \leq \frac{C_2 r}{a} B^{n^2} + C_2 rB^{n^2-1}.
\]
Choosing $c = 4C_2/C_1$ and $a = cr = 4C_2 r/C_1$, this is
\[
\leq \frac{1}{2} C_1 B^{n^2} 
\]    
as soon as $B \leq a$, which is what we wanted to prove. We have thus established the assertion in the first part of the lemma, where we may take $L \approx cr$.

Now to prove the last condition, we now want to choose the vectors $\e_1,\dotsc,\e_n$ so that the vector $\e_1$ makes an angle of at most $\pi/3$, say, to the fixed vector $\f$.
 Recall the sets $\cM(B),\cM_\psi(B)$ considered before, and define $\cM^{\theta}(B)$, $\cM_{\psi}^{\theta}(B)$ to be the set of matrices $E$ in $\cM(B)$ or $\cM_\psi(B)$, respectively, such that the first column vector $\e_1$ makes an angle of at most $\theta$ with the vector $\f$. Let $\cM^{\theta} = \cM^{\theta}(1)$ and $\cM_{\psi}^{\theta} = \cM_{\psi}^{\theta}(1)$. Since $\cM_\psi$ is invariant under rotations, one sees that for any $0 < \theta \leq \pi/2$, the set $\cM_{\psi}^{\theta}$ has a well-defined and positive volume and that
\[
\vol(\cM_{\psi}^{\theta}) \geq \vol(\cM^\theta) - n \vol(\cN_\psi).
\]
One may now calculate that $\vol(\cM^\theta) \gg_n \theta^2 \vol(\cM)$, so upon replacing the constant $c_1(n)$ in the proof of Lemma \ref{lem:subspace} by a smaller yet positive constant, one may replace $\cM_{\psi}(B)$ by $\cM_{\psi}^{\pi/3}(B)$ in the definition of the sets $M(B)$ and $M_\ast(B)$, and still conclude that there exists a matrix $E \in M_\ast(B)$ with the desired properties, for $B=L$ large enough. 

\end{proof}

Let us compare Lemma \ref{lem:subspace} with \cite[Lemma 5]{Browning-Heath-Brown09}, which specialises to the case of one form $G$. An argument akin to Bertini's theorem was used in \cite[Lemma 5]{Browning-Heath-Brown09} to find a primitive integer vector $\m$ such that intersecting with the hyperplane $\m.\x = 0$ lowers the dimension of the singular locus of $V_v(G)$ for $v\in \Pi\cup \infty$ . In addition, at each stage, $\m$ could be chosen in a {\em nice} way. Lemma \ref{lem:subspace} can be seen a a generalisation of \cite[Lemma 5]{Browning-Heath-Brown09}. Being able to deal with  an arbitrary finite collection of forms of various degrees will be crucial in dealing with systems of forms of higher degree.

We end this analysis by explaining how Lemma \ref{lem:subspace} will be used to bound the sums $\cT(q,z)$. For $f,g$ as defined at the beginning of this section, above we let
\[
F(x_0,\dotsc,x_n) := x_0^{\deg(f)}\, f(\tfrac{x_1}{x_0},\dotsc,\tfrac{x_n}{x_0})
\quad\text{and}\quad
G(x_0,\dotsc,x_n) := x_0^{\deg(g)}\, g(\tfrac{x_1}{x_0},\dotsc,\tfrac{x_n}{x_0})
\]
be their homogenisations and 
$$F_0(x_1,\dotsc,x_n) := F(0,x_1,\dotsc,x_n), \quad G_0(x_1,\dotsc,x_n) := G(0,x_1,\dotsc,x_n)$$ the leading forms of $f$ and $g$. Now, if $n \geq 2$, we define
\begin{equation}
\label{eq:s'def}
s_v' = s_v'(f,g) := 
\max \left\{s_v(F_0),s_v(G_0),s_v(F_0,G_0)
\right\}, 
\end{equation}
provided that
\[
\delta_v(F_0) = \delta_v(G_0) = \delta_v(F_0,G_0)+1 = n-2,
\]
and
\(
s_v' = n-1
\)
otherwise.
Note that in this definition, $V_v(G_0)$ and $V_v(F_0,G_0)$ are considered as subvarieties of $\PP^{n-1}$. Let us give an example of how Lemma \ref{lem:subspace} will be used: The exponential sum $\cT(q,z)$ can be most efficiently estimated in the non-singular case, that is, when $s'_\infty = -1$. Therefore, when $s'_\infty \neq -1$, we will employ Lemma \ref{lem:subspace}.

\subsection{Bounding $\cT(q,z)$.} The starting point for our investigation of $\cT(q,z)$ is an application of the Poisson summation formula, a standard technique. The proof follows from a minor modification of \cite[Lemma 8]{Browning-Heath-Brown09} and we will skip it here.
\begin{lemma}
\label{lem:poisson}
We have
\[
\cT(q,z) = q^{-n} \sum_{\v \in \Z^n} S(q,\v) I(z,q^{-1}\v),
\]
where
\begin{equation*}
S(q,\v) :=
\starsum_{s_1,s_2 = 1}^{b_1} \sum_{\a \bmod q} e_{b_1}\big(s_1 g(\a) 
+(s_1-s_2)f(\a)\big) 
e_{q_2}(ag(\a)) e_q(\v.\a)
\end{equation*}
and
\[
I(z,\v) = \int w(\x/P)\e(z g(\x)-\v.\x )d\x.
\]
\end{lemma}

A key feature in our approach is that, when applying Lemma \ref{lem:poisson} in the case where $f$ is a quartic polynomial and $g$ a cubic polynomial, only the cubic polynomial $g$ occurs in the exponential integral, allowing us to use \cite[Lemma 9]{Browning-Heath-Brown09}. Upon an application of \cite[Lemma 9]{Browning-Heath-Brown09} we get the following estimate.
\begin{proposition}
\label{prop:poisson+intbound}
Suppose that $d_1 = 4$ and $d_2=3$. Then we have the estimate
\begin{equation*}
\cT(q,z) \ll P^{-N} + q^{-n}P^n \max_{\v_0} \sum_{|\v - \v_0| \leq P^\ve V} |S(q,\v)|
\end{equation*}
where \begin{equation}
V=qP^{-1}\max\{1,(HP^3|z|)^{1/2}\}.
\label{eq:V def}
\end{equation}
\end{proposition}


The following observation from \cite[Lemma 10]{Browning-Heath-Brown09} will be useful for establishing multiplicativity relations for our exponential sums.

\begin{lemma}
\label{lem:mult}
Let $r,s$ be integers with $(r,s)=1$, and let $\bar{r},\bar{s}$ be integers satisfying $r\bar{r}+ s\bar{s} = 1$. Then, for any polynomial $h \in \Z[x_1,\dotsc,x_n]$, and any $\x,\y \in \Z^n$, we have
\[
h(r\bar{r}\x + s\bar{s}\y) \equiv r\bar{r} h(\x) + s\bar{s} h(\y) \pmod{rs}.
\]
Moreover, for any rational function $R(\x) = \frac{h_1(\x)}{h_2(\x)}$, where $h_1,h_2 \in \Z[x_1,\dotsc,x_n]$, we have 
\[
R(r\bar{r}\x + s\bar{s}\y) \equiv r\bar{r} R(\x) + s\bar{s} R(\y) \pmod{rs},
\]
provided that $h_2(\x),h_2(\y)$ and $h_2(r\bar{r}\x + s\bar{s}\y)$ are invertible in $\Z/rs\Z$.
\end{lemma}

\begin{proof}
The above statements are consequences of the fact that the map
\[
\phi: \Z/rs\Z \times \Z/rs\Z \to \Z/rs\Z, \quad (x,y) \mapsto r\bar r x+ s\bar s y
\]
is a ring homomorphism, which is easily verifiable.
\end{proof}

Now we may prove the following multiplicativity property for the exponential sums: 

\begin{lemma}
Let $\bar{b_1},\bar{q_2}$ be integers such that 
\(
b_1\bar{b_1} + q_2 \bar{q_2} = 1.
\)
Then we have
\begin{equation}
\label{eq:Sumult}
 S(q,\v)=T(b_1,\overline{q_2}\v)T^*(q_2,
\overline{b_1} \v ),
\end{equation}
where 
\begin{equation}
\begin{split}
T(q,\v) = T(q,\v;f,g)=\sum_{\x\bmod {q}}\:\:\starsum_{s_1,s_2=1}^qe_{q}(
s_1g(\x)+(s_1-s_2)f(\x)+\v.\x)
\label{Tudef}
\end{split}
\end{equation}
and
\begin{equation}
T^*(q,\v)= T^*_a(q,\v;g) := \sum_{\x\bmod q} e_{q}\big(a g(\x)+\v.\x\big)\label{Tadef}
\end{equation}
for any $q \in \NN$.
Moreover, if $r,s$ are coprime integers and 
$\bar{r},\bar{s}$ are integers such that $r\bar{r}+s\bar{s}=1$, 
then
\begin{equation}
 T(rs,\v;f,g)=T(r,\bar{s}\v;\bar s f,\bar s g)T(s,\bar{r}\v;\bar r f,\bar r g)\label{Tumult}
\end{equation}
and
\begin{align}
\label{eq:Tamult}
T^*_a(rs,\v)=T^*_{\bar s a}(r,\bar s\v)T^*_{\bar r a}(s,\bar r \v).
\end{align}
\end{lemma}

\begin{proof}
Substituting $\a = q_2\overline{q_2}\x_1 + b_1\overline{b_1}\x_2$ in the definition of $S(q,\v)$, where the vectors $\x_1$ and $\x_2$ run through all residue classes $\bmod{b_1}$ and $\bmod{q_2}$, respectively, Lemma \ref{lem:mult} gives
\begin{align*}
s_1 g(\a) +(s_1-s_2)f(\a) &\equiv s_1 g(\x_1) 
+(s_1-s_2)f(\x_1) \pmod{b_1},\\
ag(\a) &\equiv ag(\x_2) \pmod{q_2}, \textrm{ and }e_q(\v.\a) = e_{b_1}(\bar{q_2}\v.\x_1)e_{q_2}(\bar{b_1}\v.\x_2).
\end{align*}
This gives us the relation \eqref{eq:Sumult}.

To prove \eqref{Tumult}, we set
\begin{equation}
\label{eq:subst}
(s_1,s_2,\x) = s\bar{s}(s_1',s_2',\x') + r\bar{r}(s_1'',s_2'',\x'')
\end{equation}
in \eqref{Tudef}, where the $s_i'$ run through $(\Z/r\Z)^*$, the $s_i''$ run through $(\Z/s\Z)^*$, and $\x',\x''$ run through $(\Z/r\Z)^n$ and $(\Z/s\Z)^n$, respectively. Now, if we put
\[
U_\v(s_1,s_2,\x) = s_1g(\x)+(s_1-s_2)f(\x)+\v.\x,
\]
then Lemma \ref{lem:mult} gives
\[
e_q(U_\v(s_1,s_2,\x)) = e_r(\bar s U_\v(s_1',s_2',\x)) e_s(\bar r U_\v(s_1'',s_2'',\x)),
\]
which establishes \eqref{Tumult}. 

The multiplicativity relation \eqref{eq:Tamult}, finally, is precisely the one given in \cite[Lemma 10]{Browning-Heath-Brown09}.
\end{proof}

We end this section by introducing a quantity that will appear in our estimates for the exponential sums $\cT(q,z)$. 
For $q = b_1q_2$, 
define
\begin{equation}
\label{eq:D_def}
\cD(q) = \cD_{f,g}(q) = \prod_{i=1}^n \prod_{\substack{p \vert b_1 \\s_p'(f,g) = i-1}} p^{i/2} \prod_{\substack{p \vert q_2 \\s_p'(f,g) = i-1}} p^i,
\end{equation}
provided that $n \geq 2$. If $n=1$, we instead define $\cD(q)$ to be the product of all primes $p$ such that $G_0$ vanishes identically (mod $p$), that is
\[
\cD(q) = \prod_{p \mid (q,\content(G_0))} p. 
\]
Here,
\begin{equation}\label{def:cont}
\content(G_0)=\gcd \textrm{of all the coefficients of }G_0.
\end{equation} 
\section{Exponential sums to square-free moduli}
\label{sec:cubefree}

In this section, we shall provide bounds for the exponential sums $T(b_1,\v)$ defined in \eqref{Tudef} for square-free integers $b_1$. We would like to emphasize here that after a sufficient number of hyperplane intersections, we may end up with having to estimate exponential sums involving any number of variables $\leq n$. Since we did not want to introduce an extra notation for the number of variables defining the polynomials $f,g$, in this section, we will derive bounds which are valid for all $n$ including $n=1,2$. 

Here, we consider a more general version of the sum $T(b_1,\v)$ where $e_{b_1}(\cdot)$ is replaced by an arbitrary primitive additive character $\psi : \ZZ/b_1\ZZ \to \CC$; put
\[
T(\psi,\v) := \sum_{\x\bmod{b_1}}\:\:\starsum_{s_1,s_2=1}^{b_1}\psi\big(
s_1g(\x)+(s_1-s_2)f(\x)+\v.\x\big).
\]
This is crucial, since after using our multiplicativity relation \eqref{Tumult}, we end up with exponential sums involving different characters, and we need a uniform way for bounding them. We shall sometimes attach the subscript $b_1$ to the character $\psi$ for clarity. 

We should also note here that all the implied constants in this section are allowed to depend only on $n$, $d_1$ and $d_2$. 

Our main bound for the exponential sums to prime moduli is contained in the following result. 

\begin{lemma}
\label{lem:prime}
Let $n \geq 2$, $d_1>d_2$ and assume that $s_\infty' = -1$. 
There exists a non-zero homogeneous polynomial $\Phi= \Phi_{f,g} \in \ZZ[v_1,\dotsc,v_n]$
such that the estimate
\[
|T(\psi_p,\v)| \ll_{n,d_1,d_2} p^{(n+3+s_p')/2}\big(p,\Phi(\v))^{1/2}
\]
holds for any prime $p$ such that the leading form $F_0$ does not vanish identically modulo $p$ (i.e. $\delta_p(F_0)=n-2$), any non-trivial additive character $\psi_p$ on $\FF_p$ and any $\v \in \ZZ^n$. The polynomial may be chosen so that 
\begin{enumerate}
\item
\label{enum:nonvanishing}
coefficients of $\Phi$ are co-prime, i.e., $\content(\Phi) = 1$; 
\item
we have the bounds
\begin{equation*}
\deg \Phi \ll_{n,d_1,d_2} 1, \quad \log \Vert \Phi \Vert \ll_{n,d_1,d_2} \log \Vert f \Vert + \log \Vert g \Vert.
\end{equation*}
\end{enumerate}
\end{lemma}

Lemma \ref{lem:prime} states, in other words, that we may find a polynomial $\Phi$ such that the optimal bound $|T(\psi_p,\v)| \ll p^{(n+3+s_p')/2}$ is attained for all $\v$ such that $p \nmid \Phi(\v)$.
Before giving the proof of Lemma \ref{lem:prime}, we state and prove a result that will supply the polynomial $\Phi$ in the statement.
For any $\v \in \ZZ^n$, we let $\delta(\v)$ be the dimension of the singular locus of the intersection of $V_p(F_0,G_0)$ with the hyperplane $L_\v$ in $\PP^{n-1}_{\FF_p}$, i.e. $\delta(\v) = s_p(F_0,G_0,L_\v)$. It turns out that the quantity $\delta(\v)$ will govern the strength of our bound for $T(\psi_p,\v)$.

\begin{proposition}
\label{prop:dualform}
There exists a non-zero homogeneous polynomial $\Phi = \Phi_{F_0,G_0} \in \ZZ[v_1,\dotsc,v_n]$ with the following properties:
\begin{itemize}
\item
we have
\(
\delta(\v) \leq s_p'(f,g)
\)
as soon as $p \nmid \Phi(\v)$;
\item
$\content(\Phi) = 1$; 
\item
we have the bounds
\begin{equation*}
\deg \Phi \ll_{n,d_1,d_2} 1, \quad \log \Vert \Phi \Vert \ll_{n,d_1,d_2} \log \Vert f \Vert + \log \Vert g \Vert.
\end{equation*}
\end{itemize}
\end{proposition} 

\begin{proof}
A necessary condition for the inequality
\[
s_v(F_0,G_0,L_\v) > s_v(F_0,G_0)
\]
to hold is that $L_\v$ be tangent to the variety $V_v(F_0,G_0)$ at one of its non-singular points, or in other words, that $L_\v$ belong to the dual variety $V_v(F_0,G_0)^*$ in $(\PP^{n-1}_{\FF_v})^\vee$. We shall obtain universal equations defining this dual variety.

We parametrise homogeneous polynomials $F_i$ of degree $d_i$ by coefficient vectors $\c^{(i)}$, viewed as points in projective spaces $\PP^{N_i}$. 
By Chevalley's constructibility theorem, the set
\[
\cW = \left\{(F_1,F_2,\v) \mid \ L_\v \text{ is tangent to } V(F_1,F_2) \text{ at a non-singular point}\right\}
\]
is a constructible subset of $\PP^{N_1}_\ZZ \times \PP^{N_2}_\ZZ \times (\PP^{n-1}_\ZZ)^\vee$. Thus we may write
\[
\cW = \bigcup_{i=1}^k \cU_i \cap \cS_i,
\]
where each $\cU_i$ is open and each $\cS_i$ is closed. Clearly we may assume that the $\cU_i$ are non-empty. We may now fix, once and for all, a collection of non-zero homogeneous polynomials
\[
\Psi_1,\dotsc,\Psi_R 
\]
in the multigraded ring $\ZZ[\c^{(1)},\c^{(2)},\v]$ such that generators for the vanishing ideals of all the closed sets $\cS_i$ may be found among the $\Psi_j$.

For a fixed tuple $(F_1,F_2) =(F_0,G_0) \in \ZZ[\x]^2$, denote by $U_i$ and $S_i$ the fiber over $(F_1,F_2)$ of $\cU_i$ and $\cS_i$, respectively, and put
\[
W = \bigcup_{i=1}^k U_i \cap S_i, \qquad Z = \bigcup_{i=1}^k S_i.
\] 
By discarding some of the indices if necessary, we may again assume that the $U_i$ are all non-empty. As subsets of $(\PP^{n-1}_\ZZ)^\vee$, $W$ is constructible and $Z$ is closed. Furthermore, generators for the vanishing ideals $I(S_i)$ of the sets $S_i$ in $\ZZ[\v]$ may be found among the specialisations $\Theta_j$ of the polynomials $\Psi_j$ above at $(F_1,F_2)$. 

By definition, the dual variety $V^*_v \subset (\PP^{n-1}_{\FF_v})^\vee$ of the variety $V_v = V_v(F_1,F_2)$ is the Zariski closure of $W_{\FF_v} = W \otimes \Spec{\FF_v}$ in $(\PP^{n-1}_{\FF_v})^\vee$. If $V_v$ is irreducible, then $V_v^*$ is irreducible of dimension at most $n-2$. In general, if $V_v$ has a decomposition into irreducible components $V_v = C_1 \cup \dotsb \cup C_m$, then it is easy to see that $V_v^* = C_1^* \cup \dotsb \cup C_m^*$.

For each $1 \leq j \leq R$, put
\[
\tilde{\Theta}_j = \frac{\Theta_j}{\content(\Theta_j)} \in \ZZ[\v].
\]
Then we claim that
\[
\Phi:= \prod_{j=1}^R \tilde{\Theta}_{k}
\] 
is a polynomial satisfying the desired properties.
To see this, let $p$ be any prime. If $V_p$ is irreducible, then for some closed subscheme $S_{i}$ we have $V_p^* \subseteq (S_i)_{\FF_p} \neq \PP^{n-1}_{\FF_p}$. Thus there is a polynomial $\Theta_{j}$, with $(p,\content(\Theta_j)) = 1$, such that $\delta(\v) \leq s'_p(F_1,F_2)$ as soon as $p \nmid \Theta_j(\v)$. But then the polynomial $\tilde{\Theta}_j$ 
satifies the same property. On the other hand, if $V_p$ is not irreducible, then by definition we have $s'_p(F_1,F_2) \geq n-4$, so we trivially have $\delta(\v) \leq s'_p(F_1,F_2)$ for all $\v$. 
\end{proof}
\begin{rem}
We would like to note that Proposition \ref{prop:dualform} holds for an arbitrary choice of $d_1$ and $d_2$, i.e., it does not require the assumption $d_1>d_2$, as required by a future application. In fact, the differing degrees of $f$ and $g$ are only used in bounding the term $\Sigma_2$ in the proof of Lemma \ref{lem:prime} below.
\end{rem}

\begin{proof}[Proof of Lemma \ref{lem:prime}]
We may write
\begin{align*}
T(\psi_p,\v) &= \Sigma_1 - \Sigma_2 - \Sigma_3 + \Sigma_4, 
\end{align*}
where
\begin{align*}
\Sigma_1&=\sum_{a,b=1}^p \sum_{\x\bmod {p}}\psi(
ag(\x)+b f(\x)+\v.\x), & \Sigma_3 &= \sum_{b=1}^p \sum_{\x\bmod{p}}  \psi(bf(\x) + \v.\x) \\ 
\Sigma_2 &= \sum_{a=1}^p \sum_{\x\bmod{p}}  \psi(a(g(\x) + f(\x)) + \v.\x), & 
\Sigma_4 &= \sum_{\x\bmod{p}}  \psi(\v.\x).
\end{align*}
Let us first treat the case where $p \mid \v$. By a theorem of Hooley \cite{Hooley} (see \cite[Lemma 3.2]{Marmon08} for its affine reformulation), we then have
\begin{align*}
\Sigma_1&=p^2 \#\{\x \bmod{p} \mid f(\x) \equiv g(\x) \equiv 0 \bmod{p}\} = p^n + O\big(p^{(n+4+s_p')/2}\big).
\end{align*}
By the same argument we have
\[
\Sigma_2 = p\left(p^{n-1} + O\big(p^{(n+1+s_p')/2}\big)\right) = p^n + O\big(p^{(n+3+s_p')/2}) = \Sigma_3, 
\]
since $F_0(\x)$ is the leading form of $f(\x) + g(\x)$. Since $\Sigma_4 = p^n$, we conclude that
\[
T(\psi_p,\v) \ll  p^{(n+4+s_p')/2} = p^{(n+3+s_p')/2}(p,\v)^{1/2}.
\]
This agrees with our claim, since the polynomial $\Phi$ is homogeneous. 

Now we turn to the case where $p \nmid \v$. 
We observe that
\begin{equation}
\label{eq:Sigma_1}
\Sigma_1 = p^2 \sum_{\substack{\x \bmod{p} \\ p \mid g(\x),\ p \mid f(\x)}} \psi(\v.\x).
\end{equation}
Suppose first that $n \geq 3$. If $F_0$ and $G_0$ intersect properly (mod $p$), i.e. if the corresponding projective varieties intersect properly (mod $p$),
then the exponential sum over the variety defined by the equations $g=f=0$ in $\AA^n_{\FF_p}$ may be treated by means of a result of Katz \cite{Katz}. Indeed, by \cite[Thm. 4]{Katz}, we have
\[
\sum_{\substack{\x \bmod{p} \\ p \mid g(\x),\ p \mid f(\x)}} \psi(\v.\x) \ll p^{(n-1+\delta(\v))/2},
\]
where $\delta(\v) = s_p(F_0,G_0,L_\v)$. By a result of Zak and Fulton and Lazarsfeld, as explained in \cite[p.~897]{Katz_SE}, we have $\delta(\v) \leq s_p'+1$. Furthermore, by Proposition \ref{prop:dualform}, we have $\delta(\v) \leq s_p'$ as soon as $p \nmid \Phi(\v)$.
Thus we get
\[
\Sigma_1 \ll p^{(n+3+s_p')/2} (p,\Phi(\v))^{1/2} 
\]
if $F_0$ and $G_0$ intersect properly (mod $p$). 
On the other hand, if $F_0$ and $G_0$ do not intersect properly (mod $p$), 
then we have $s_p' = n-1$ by definition.
Then we certainly still have the estimate
\[
|\Sigma_1| \leq p^2 \sum_{\substack{\x \bmod{p} \\ p \mid g(\x),\ p \mid f(\x)}} 1 \ll p^{n+1} = p^{(n+3+s_p')/2},
\]
by our assumption that $F_0$ does not vanish entirely (mod $p$).

Furthermore, we may use \cite[Lemma 7]{Browning-Heath-Brown09} to obtain
\[
\left| \Sigma_3 \right| \leq \sum_{b\bmod{p}} \left|\sum_{\x\bmod{p}}\psi(bf(\x) + \v.\x)\right| \ll p^{(n+3+s_p')/2},
\]
and the sum $\Sigma_2$ satisfies the same bound, since the polynomial $g(\x) + f(\x)$ has $F_0(\x)$ as its leading form. Thus, we see that $\Sigma_2$ and $\Sigma_3$ both give negligible contributions to $T_{0,g}(p,\v)$. (We could equally well have proved this by the arguments used to estimate $\Sigma_1$ above.) The term $\Sigma_4$ vanishes in this case.

Finally, we treat the case where $n=2$. Here, we need not discern whether $p \mid \v$ or not. The bounds for $\Sigma_2$ and $\Sigma_3$ from the previous case remain valid for any $n \geq 1$ and any $\v$. Also, we trivially have $|\Sigma_4| \leq p^2 \leq p^{(n+3+s_p')/2}$ for $n=2$. To estimate $\Sigma_1$, note that we have $s_p' = 1$ if the binary forms $F_0$ and $G_0$ have a common projective zero, and $s_p' = -1$ otherwise. 
In the former case, we have
\[
|\Sigma_1| \leq p^3 + \sum_{a=1}^{p-1}\sum_{b=1}^p \left| \sum_{\x \bmod p} \psi(af(\x) + bg(\x) + \v.\x)\right| \ll p^3 = p^{(n+3+s_p')/2}
\]
by \cite[Lemma 7]{Browning-Heath-Brown09}, since each polynomial $af(\x) + bg(\x) + \v.\x$ has a multiple of $F_0$ as its leading form. In the latter case, there are only $O(1)$ solutions to $f(\x) \equiv g(\x) \equiv 0 \bmod p$, so 
\[
\Sigma_1 = p^2 \sum_{\substack{\x \bmod p \\f(\x) \equiv g(\x) \equiv 0 \bmod p}} \psi(\v.\x) \ll p^2 = p^{(n+3+s_p')/2}. 
\]
In both cases, we used the fact that $F_0$ was supposed not to vanish entirely (mod $p$). This concludes the proof of Lemma \ref{lem:prime}. We observe that we could have disposed of the factor $(p,\Phi(\v))^{1/2}$ in the case $n=2$, but this stronger bound is not needed in the sequel.
\end{proof}

In the case $n=1$, the quantity $s_v'$ is not meaningful. Instead, we have the following alternative bounds. Here $\resultant(f,g)$ denotes the usual resultant of two univariate polynomials.

\begin{lemma}
\label{lem:p_n=1}
Let $n=1$. Then we have
\[
|T(\psi_p,v)| \ll p(p,\resultant(f,g))
\]
for any $v \in \Z$, provided that $p$ does not divide the leading coefficient of $f$.
\end{lemma}

\begin{proof}
We write
\[
T(\psi,v) = \Sigma_1 - \Sigma_2 - \Sigma_3 + \Sigma_4,
\]
as in the proof of Lemma \ref{lem:prime}. Here we have $|\Sigma_4| \leq p$, and $|\Sigma_2|$ and $|\Sigma_3|$ are bounded from above by $p$ times the number of zeroes of $f+g$ and $f$, respectively, in $\ZZ/p\ZZ$. Under our assumption, we therefore get $\Sigma_i \ll p$ for $i =2,3,4$. Finally,
\[
|\Sigma_1| = \left|\sum_{a,b=1}^p \sum_{x=1}^p \psi(af(x)+bg(x)+vx)\right| \leq p^2 \# \{x \in \ZZ/p\ZZ \mid f(x) \equiv g(x) \equiv 0 \pmod p\},
\]
so $\Sigma_1$ vanishes unless $f$ and $g$ have a common zero (mod $p$), in which case $\Sigma_4 \ll p^2$.
\end{proof}

Now let $q = b_1$ be an arbitrary squarefree integer. The estimates above assumed that $F_0$ does  not vanish identically (mod $p$). To say that this should hold for all primes $p \mid b_1$ amounts to the condition 
\[
(b_1,\content(F_0)) = 1,
\]
where $\content(F_0)$ is defined by \eqref{def:cont}. Now we observe that (an easy extension of) the multiplicativity property \eqref{Tumult} may be reformulated to state that if $\psi_{b_1}$ is a primitive additive character modulo $b_1 = rs$, where $(r,s)=1$, then there exist primitive additive characters $\psi_r$ modulo $r$ and $\psi_s$ modulo $s$ such that
\[
T(\psi_{b_1},\v;f,g) = T(\psi_r,\v;f,g) T(\psi_s,\v;f,g).
\] 
Decomposing $b_1$ into prime factors and multiplying together the bounds obtained for each factor by Lemma \ref{lem:prime} or \ref{lem:p_n=1}, we obtain a bound where the implied constant $C = C(n,d_1,d_2)$, say, is replaced by a factor which is at most $C^{\omega(b_1)} \ll b_1^\ve$. We thus arrive at the following results. 

\begin{lemma}
\label{lem:squarefree}
Let $n \geq 2$ and let $b_1$ be a square-free integer. For the polynomial $\Phi$ from Lemma \ref{lem:prime}, the estimate
\[
|T(\psi_{b_1},\v)| \ll b_1^{(n+2)/2 + \ve}\cD(b_1)(b_1,\Phi(\v))^{1/2}
\]
holds for any square-free number $b_1$ such that $(b_1,\content(F_0)) = 1$, any primitive additive character $\psi_{b_1}$ modulo $b_1$ and any $\v \in \ZZ^n$. 
\end{lemma}

\begin{lemma}
\label{lem:squarefree_n=1}
Let $n=1$. Then the estimate
\[
|T(\psi_{b_1},v)| \ll b_1^{3/2} \big(b_1,\resultant(f,g)\big)^{1/2} 
\]
holds for any square-free $b_1$ such that $(b_1,\content(F_0)) = 1$, any primitive additive character $\psi_{b_1}$ modulo $b_1$ and any $v \in \ZZ$.
\end{lemma}

\section{Exponential sums to square-full moduli}
\label{sec:cubefull}
Let $g \in \ZZ[x_1,\dotsc,x_n]$ be a cubic polynomial. Let $q_2$ be an arbitrary square-full integer. The exponential sums $T^*(q_2,\v)$ in \eqref{Tadef} that we consider coincide precisely with those investigated in \cite{Browning-Heath-Brown09}. We again begin by writing $q_2=b_2q_3$, where $q_3$ is cube-full part of $q_2$. We begin by estimating the exponential sum $T^*(b_2,\v)$ for an arbitrary integer $b_2$, which is purely a product of squares of distinct primes. The bound from \cite[Lemma 7]{Browning-Heath-Brown09} implies:

\begin{lemma}
\label{lem:cubefree}
Let $n\geq 1$. Then for any $\v \in \ZZ^n$, we have
\[
T^*(b_2,\v) \ll b_2^{n/2+\ve} \cD(b_2)
\]

\end{lemma}


Rest of this section is dedicated to estimating exponential sums modulo cube-full integers $q_3=c^2d$, where $d$ is square-free. The following bound is proven in \cite[Lemma 11]{Browning-Heath-Brown09}:  
\begin{lemma}
For any $r \in \ZZ$ with $(r,q_3) = 1$, we have the bound
 \label{lem:cubefull}
 \begin{equation}
 \label{eq:cubefull}
|T^*_{ra}(q_3,r\v)|\ll q_3^{n/2}\sum_{\substack{\s\bmod {c}\\c\mid a\nabla 
g(\s)+\v
}} 
M_d(\s)^{1/2},
 \end{equation}
 where
\[
M_d(\s) := \big\{\t \in (\Z/d\Z)^n \mid \nabla^2 g(\s)\t \equiv \0 \pmod d\big\}.
\]
\end{lemma}

We shall need bounds for two kinds of averages of the exponential sums $T^*(q_3,\v)$. First, for any $\v_0 \in \RR^n$ and any $V \geq 1$, we shall evaluate the sum
\begin{equation}
\label{eq:cubefull_average}
\sum_{|\v-\v_0|\leq V} T^*(c^2d,\v).
\end{equation}
In view of Lemma \ref{lem:cubefull}, we then need to estimate the quantity 
\begin{equation}
\label{eq:S(V)}
\sum_{|\v-\v_0|\leq V} \cP(c^2d,\v), \quad \text{where} \quad \cP(c^2d,\v) := \sum_{\substack{\s\bmod {c}\\c\mid a\nabla 
g(\s)+\v
}} 
M_d(\s)^{1/2}.
\end{equation}
We shall assume that
the leading form $G_0$ is non-singular, in other words that
\begin{equation}
\label{eq:g_nonsing}
s_\infty(G_0) = -1
\end{equation}
in the notation of Section \ref{sec:expsums_intro}. Furthermore, we shall assume that
\begin{equation}
\label{eq:g_height}
\Vert g \Vert_P \leq H
\end{equation}
for some $H \geq 1$.
In \cite{Browning-Heath-Brown09}, two alternative bounds for the quantity in \eqref{eq:S(V)} are presented. In the present situation we shall only need the latter of these, given by Lemma 16. Together with the discussion concluding \cite[\S 5]{Browning-Heath-Brown09}, this implies the following estimate. 

\begin{lemma}
\label{lem:S(V)_BHB}
If the cubic polynomial $g \in \Z[x_1,\dotsc,x_n]$ satisfies \eqref{eq:g_nonsing} and \eqref{eq:g_height}, then we have the bound
\[
\sum_{|\v-\v_0|\leq V} \cP(q_3,\v) \ll q_3^\ve \cD(d)\big(V^n+(q_3H)^{n/3}\big),
\]
where
$\cD(\cdot)$ is the quantity defined in \eqref{eq:D_def}. 
\end{lemma}

We also need to consider versions of the sum \eqref{eq:cubefull_average} where the vectors $\v$ are restricted to the ones satisfying the equation $\Phi(\v) = 0$ for a certain polynomial $\Phi \in \ZZ[v_1,\dotsc,v_n]$, or a congruence $\Phi(\v) \equiv 0 \bmod m$ for some integer $m$. 

\begin{lemma}
\label{lem:S^*(V)}
Let $\Phi \in \ZZ[v_1,\dotsc,v_n]$, where $n \geq 1$, be a non-zero polynomial. Then we have the bound
\begin{equation}
\label{eq:sparse_trivial}
\sum_{\substack{|\v-\v_0|\leq V\\ \Phi(\v) = 0}} \cP(q_3,\v) \ll_{n,\deg(\Phi)} q_3^\ve c^n \cD(d)^{1/2} \left(1 + \frac{V}{c} \right)^{n-1}.
\end{equation}
Furthermore, for any squarefree integer $m$ with $(m,c) = 1 = (m,\content(\Phi))$, we have
\begin{equation}
\label{eq:sparse_congruence}
\sum_{\substack{|\v-\v_0|\leq V\\ \Phi(\v) \equiv 0\bmod m}} \cP(q_3,\v) \ll_{n,\deg(\Phi)} (mq)^\ve c^n \cD(d)^{1/2} \left(1 + \left(\frac{V}{c}\right)^{n-1} +  \left(\frac{V}{c}\right)^n m^{-1}\right).
\end{equation}
\end{lemma}

\begin{proof}
We may write
\[
\sum_{\substack{|\v-\v_0|\leq V\\ \Phi(\v) = 0}} \cP(q_3,\v) \leq \sum_{\s \bmod c} M_d(\s)^{1/2} \max_{\r \bmod c} U_\r(V),
\]
where
\[
U_\r(V) := \#\{\v \in \ZZ^n \mid |\v-\v_0| \leq V,\ \Phi(\v) = 0,\ \v \equiv \r \bmod c\}.
\]
Analogously, we have
\[
\sum_{\substack{|\v-\v_0|\leq V\\ \Phi(\v) \equiv 0\bmod m}} \cP(q_3,\v) \leq \sum_{\s \bmod c} M_d(\s)^{1/2} \max_{\r \bmod c} U_{\r,m}(V),
\]
where
\[
U_{\r,m}(V) := \#\{\v \in \ZZ^n \mid |\v-\v_0| \leq V,\ \Phi(\v) \equiv 0 \bmod m,\ \v \equiv \r \bmod c\}.
\]
If $c \geq V$, then we obviously have $U_\r(V) \leq 1$ and $U_{\r,m}(V) \leq 1$. If $c < V$, then we may write
\[
U_{\r}(V) \leq \#\left\{\u \in \ZZ^n ; |\u|\leq \frac{2V}{c}, \Psi(\u) = 0 \right\}
\]
and
\[
U_{\r,m}(V) \leq \#\left\{\u \in \ZZ^n ; |\u|\leq \frac{2V}{c}, \Psi(\u) \equiv 0 \bmod m \right\},
\]
where $\Psi(\u) = \Phi(\v_1 + c\u)$ for some $\v_1 \in \ZZ^n$. The polynomial $\Psi \in \ZZ[x_1,\dotsc,x_n]$ is not the zero polynomial, and by our assumption that $(m,c)= 1 = (m,\content(\Phi))$, its image in $\FF_p[x_1,\dotsc,x_n]$ is also non-vanishing for all $p \mid m$. It then follows from \cite[Lemma 4]{Browning-Heath-Brown09} that
\[
U_\r(V) \ll \left(\frac{V}{c}\right)^{n-1} \quad \text{and} \quad U_{\r,m}(V) \ll m^\ve \left(\left(\frac{V}{c}\right)^{n-1} + \left(\frac{V}{c}\right)^{n}m^{-1}\right).
\]
In general, we therefore have the bounds
\[
U_\r(V) \ll 1 + \left(\frac{V}{c}\right)^{n-1} \quad \text{and} \quad U_{\r,m}(V) \ll m^\ve \left(1 + \left(\frac{V}{c}\right)^{n-1} + \left(\frac{V}{c}\right)^{n}m^{-1}\right).
\]

By \cite[Lemma 14]{Browning-Heath-Brown09}
we have 
\[
\sum_{\s \bmod c} M_d(\s) \ll \left(\frac{c}{d}\right)^n \sum_{\s \bmod d} M_d(\s) \ll q_3^\ve c^n \cD(d).
\]
An application of the Cauchy-Schwarz inequality gives
\[
\sum_{\s \bmod c} M_d(\s)^{1/2} \ll c^{n/2} \left\{\sum_{\s \bmod c} M_d(\s)\right\}^{1/2} \ll q_3^\ve c^n \cD(d)^{1/2}.
\]
We conclude that
\[
\sum_{\substack{|\v-\v_0|\leq V\\ \Phi(\v) = 0}} \cP(q,\v) \ll \left(1 + \frac{V}{c}\right)^{n-1} \sum_{\s \bmod c} M_d(\s)^{1/2} \ll q^\ve c^n \cD(d)^{1/2} \left(1 + \frac{V}{c} \right)^{n-1}
\]
and similarly
\[
\sum_{\substack{|\v-\v_0|\leq V\\ \Phi(\v) \equiv 0\bmod m}} \cP(q_3,\v) \ll (mq_3)^\ve c^n \cD(d)^{1/2} \left(1 + \left(\frac{V}{c}\right)^{n-1} +  \left(\frac{V}{c}\right)^n m^{-1}\right),
\]
as claimed.

\end{proof}

Note that if $n \geq 2$ and the polynomial $\Phi$ is absolutely irreducible of degree at least $2$, then using a bound by Serre \cite[Chapter 13]{Serre97}, the exponent in the bound \eqref{eq:sparse_trivial} can be improved to $n-3/2$. However the current bound suffices to establish Theorem \ref{thm:main thm}.


\section{Evaluation of quartic exponential sums; further considerations}
\label{sec:expsums_final}

We shall use the bounds from the previous sections to evaluate the exponential sum $\cT(q,z)$. From now on, we again restrict the degrees of the polynomials to be $\deg(f) = 4$ and $\deg(g) = 3$ (unless $g$ vanishes entirely). Thus, the implied constants in our estimates depend only on $n$ and $\ve$. In the same vein as Section \ref{sec:cubefree}, we will obtain bounds for all $n$ including $n=1,2$, since $f$ and $g$ will correspond to the forms obtained from applying various stages of hyperplane intersections to $F$ and $F_\h$ respectively. 
We write
\[
q = bq_3 = b_1 b_2 q_3 \quad \text{and} \quad q_3 = c^2d
\]
as above. Here $b_1$ is the square-free and $q_3$ is the cube-full part of $q$, as per our notation \eqref{eq:factors}. Moreover, $d$ is square-free as chosen in Section \ref{sec:cubefull}.
We shall assume that
\begin{equation}
\label{eq:F_nonzero}
(b_1,\content(F_0)) = 1
\end{equation}
and furthermore that 
\begin{equation}
\label{eq:g_height'}
\Vert g \Vert_P \leq H \leq P^A
\end{equation}
for some $A>0$. To begin with, let us in addition assume that
\begin{equation}
\label{eq:nonsing}
s'_\infty = -1.
\end{equation}
Recall the bound in Proposition \ref{prop:poisson+intbound} for $\cT(q,z)$ and the factorisation \eqref{eq:Sumult} for the sums $S(q,\v)$ occurring there. Assume that $n \geq 2$. For the first factor in \eqref{eq:Sumult} we note that
\begin{align*}
T(b_1,\bar{q_2}\v) &= 
\sum_{\x\bmod {b_1}}\:\:\starsum_{s_1,s_2=1}^{b_1}e_{b_1}(
s_1g(\x)+(s_1-s_2)f(\x)+\bar{q_2}\v.\x) \\
&= \sum_{\x\bmod {b_1}}\:\:\starsum_{s_1,s_2=1}^{b_1}e_{b_1}\Big(\bar{q_2}\big(
s_1'g(\x)+(s_1'-s_2')f(\x)+\v.\x\big)\Big)\\
&= T(\psi',\v),
\end{align*}
where we have put $s_i' = q_2s_i$ for $i=1,2$ in the second step, and where $\psi'(\cdot) = e_{b_1}(\bar{q_2}\,\cdot)$ is a primitive additive character modulo $b_1$.
Thus we have the bound
\[
|T(b_1,\bar{q_2}\v)| \ll b_1^{(n+2)/2 + \ve} \cD(b_1) (b_1,\Phi(\v))^{1/2},
\]
by Lemma \ref{lem:squarefree}, where the polynomial $\Phi = \Phi_{f,g}$ satisfies the properties listed there. We split the second factor in \eqref{eq:Sumult} further into a cubefree and a cubefull part, thus writing 
\[
T^*_a(q_2,
\overline{b_1} \v ) = T^*_{\bar{q_3}a}(b_2,\bar{q_3}\bar{b_1}\v)T^*_{\bar{b_2}a}(q_3,\bar{b_2}\bar{b_1}\v).
\]
For the first factor, Lemma \ref{lem:cubefree} applies to give
\[
T^*_{\bar{q_3}a}(b_2,\bar{q_3}\bar{b_1}\v) \ll b_2^{n/2+ \ve} \cD(b_2).
\]

Inserting these bounds into Proposition \ref{prop:poisson+intbound}, and assuming that $q \leq P^2$, say, we get
\begin{multline}
\label{eq:poisson+cubefree}
\cT(q,z) \ll_N P^{-N} + 
q^{-n}P^{n+\ve} b_1 b^{n/2} \cD(b) \max_{\v_0} \sum_{|\v - \v_0| \leq P^\ve V} |T^*_{\bar{b_2}a}(q_3,\bar{b_2}\bar{b_1}\v)| (b_1,\Phi(\v))^{1/2}.
\end{multline}
Using Lemma \ref{lem:cubefull}, we then get
\begin{equation}
\label{eq:legoland}
\cT(q,z) \ll b_1 P^{n+\ve} q^{-n/2} \cD(b) \max_{\v_0} \sum_{|\v - \v_0| \leq P^\ve V} \cP(q_3,\bar{b_1}\v) (b_1,\Phi(\v))^{1/2}. 
\end{equation}
We estimate the sum over $\v$ using Lemma \ref{lem:S^*(V)}, writing
\begin{align*}
\sum_{|\v - \v_0| \leq P^\ve V} &\cP(q_3,\bar{b_1}\v) (b_1,\Phi(\v))^{1/2} \leq \sum_{m \mid b_1} m^{1/2} \sum_{\substack{|\v - \v_0| \leq P^\ve V\\\Phi(\v) \equiv 0 \bmod m}} \cP(q_3,\bar{b_1}\v) \\
&\ll \cD(d) P^\ve \sum_{m \mid b_1} m^{1/2} (c^n + cV^{n-1} + V^n m^{-1})\ll \cD(d) P^\ve \Big(b_1^{1/2}(c^n + c V^{n-1}) + V^n\Big).
\end{align*}
Inserting this bound into \eqref{eq:legoland}, we obtain
\[
\cT(q,z) \ll b_1 P^{n+\ve} q^{-n/2} \cD(q) \Big(b_1^{1/2}(c^n + c V^{n-1}) + V^n\Big).
\]
Alternatively, we may use the trivial bound $(b_1,\Phi(\v))^{1/2} \leq b_1^{1/2}$ and apply Lemma \ref{lem:S(V)_BHB} to obtain
\begin{align*}
\sum_{|\v - \v_0| \leq P^\ve V} \cP(q_3,\bar{b_1}\v) (b_1,\Phi(\v))^{1/2} &\leq b_1^{1/2} \sum_{|\v - \v_0| \leq P^\ve V} \cP(q_3,\bar{b_1}\v)
\ll b_1P^\ve \cD(d)(V+(q_3H)^{1/3})^n,
\end{align*}
which in turn yields
\[
\cT(q,z) \ll b_1^{3/2} P^{n+\ve} q^{-n/2} \cD(q) \Big(V^n + (q_3H)^{n/3}\Big).
\]
Observing that $\min\{c^n,V^n\} \leq cV^{n-1}$, we conclude as follows.

\begin{proposition}
\label{prop:1}
Suppose that $n \geq 2$. 
Then, under the conditions \eqref{eq:F_nonzero}--\eqref{eq:nonsing}, we have the bound
\begin{equation}
\label{eq:prop1}
\cT(q,z) \ll b_1 q^{-n/2} \cD(q) P^{n+\ve} \Big(V^n  +  b_1^{1/2}\big(cV^{n-1} + 
(q_3H)^{n/3}\big)\Big), 
\end{equation}
where
$V=\frac qP\max\{1,(HP^3|z|)^{1/2}\}$ as in \eqref{eq:V def}.
\end{proposition}

In the case $n=1$, we repeat the arguments above, replacing Lemma \ref{lem:squarefree} by \ref{lem:squarefree_n=1}, to obtain the following bound.

\begin{proposition}
\label{prop:1_n=1}
If $n=1$ and \eqref{eq:F_nonzero}--\eqref{eq:nonsing} hold, then we have
\[
\cT(q,z) \ll b_1 q^{-1/2} \cD(q_2) P^{1+\ve} (b_1,\resultant(f,g))^{1/2}(V+(q_3H)^{1/3}).
\]
\end{proposition}

We shall also derive a 'trivial' bound, which is useful also in the case when $g$ vanishes identically. In that particular case, the condition \eqref{eq:nonsing} is automatically violated. Instead, we only impose a non-singularity condition on $f$.

\begin{proposition}
\label{prop:trivial}
Suppose that $P\geq q^{1/2+1/n}$, that \eqref{eq:F_nonzero} holds, and in addition that $s_p(F_0) = -1$ for all primes $p \mid b_1$. Then we have the bound
\[
\cT(q,z) \ll b_1 P^{n+\ve}.
\]
\end{proposition}  

Proposition \ref{prop:trivial} will be a consequence of the following observation. 

\begin{lemma}
\label{lem:india}
Let $f \in \ZZ[x_1,\dotsc,x_n]$ be a polynomial of degree $d \geq 2$ and let $w \in \cW_n$. Let $m$ be a squarefree number such that the leading form $F_0$ of $f$ is non-singular (mod $p$) for every prime $p \mid m$. Then we have the bound 
\[
\sum_{\substack{\x \in \ZZ^n \\ m \mid f(\x)}} w(\x/P) \ll P^{n+\ve} m^{-1}
\] 
whenever
$P \geq m^{1/2+1/n}$.
\end{lemma} 

\begin{proof}
By Poisson summation we have
\begin{align*}
\sum_{\substack{\x \in \ZZ^n \\ m \mid f(\x)}} w(\x/P) &= \sum_{\substack{\z \bmod m\\ m \vert f(\z)}} \sum_{\y \in \ZZ^n} w(P^{-1}(\z + m\y)) = \frac{P^n}{m^n} \sum_{\v \in \ZZ^n} \hat{w}(\tfrac{P}{m} \v) \Sigma_m(\v),
\end{align*}
where
\[
\Sigma_m(\v) = \sum_{\substack{\x \bmod m \\ m \vert f(\x)}} e_m(\v.\x).
\]
By Lemma \ref{lem:mult} we have the multiplicativity relation
\[
\Sigma_{m_1m_2}(\v) = \Sigma_{m_1}(\bar{m}_2\v) \Sigma_{m_2}(\bar{m}_1 \v)
\]
if $(m_1,m_2) = 1$, where $\bar{m}_1,\bar{m}_2 \in \ZZ$ are such that $m \bar{m}_1 + m_2\bar{m}_2 = 1$. Suppose that $p\mid m$ is a prime. If $p \nmid \v$, then we have
\[
\Sigma_p(\v) = p^{-1} \sum_{a=1}^{p-1} \sum_{\x \bmod{p}} e_p(af(\x) + \v.\x) \ll p^{n/2}
\] 
by \cite[Lemma 7]{Browning-Heath-Brown09}.
In the opposite case we still have the bound $\Sigma_p(\v) \ll p^{n-1}$. In particular it follows that one has
\begin{equation}
\label{eq:sigma_trivial}
\Sigma_m(\v) \ll m^{n-1 + \ve}
\end{equation}
for all $\v \in \ZZ^n$ and
\begin{equation}
\label{eq:sigma_squareroot}
\Sigma_m(\v) \ll m^{n/2 + \ve}
\end{equation}
if $\gcd(m,\v) := \gcd(m,v_1,\dotsc,v_n) = 1$. The asserted bound follows if we can prove that
\begin{equation}
\label{eq:emporia}
\sum_{\v \in \ZZ^n} \hat{w}(\tfrac{P}{m} \v) |\Sigma_m(\v)| \ll m^{n-1+\ve}.
\end{equation}
Using repeated integration by parts, one may show that
\[
\hat{w}(\t) \ll_k (1+|\t|)^{-k}
\]
for any $k \geq 1$, from which it follows that
\[
\sum_{\v \in \ZZ^n} \hat{w}(\tfrac{P}{m} \v) \ll \left(1+\frac{m}{P}\right)^n. 
\]
In case $m \leq P$, the bound \eqref{eq:emporia} therefore follows directly from \eqref{eq:sigma_trivial}. We may thus suppose from now on that $m \geq P$. In this case we prove \eqref{eq:emporia} by induction on the number $\omega(m)$ of prime divisors of $m$. The case where $\omega(m) = 0$, that is, $m=1$, is trivial. For the induction step, we write 
\[
\sum_{\v \in \ZZ^n} \hat{w}(\tfrac{P}{m} \v) |\Sigma_m(\v)|  = \Sigma_1 + \Sigma_2,
\]
where $\Sigma_1$ denotes the sum ranging over $\v \in \ZZ^n$ satisfying $(m,\v) = 1$, and $\Sigma_2$ denotes the sum over those $\v$ for which there exists a divisor $m_1 > 1$ of $m$ such that $m_1 | \v$. There then exists one such divisor $m_1$, for which $\Sigma_2 \ll m^\ve \Sigma_2'$, say, where $\Sigma_2'$ is the sum over $\v$ such that $m_1 | \v$. Writing $m = m_1 m_2$ and $\v = m_1 \w$, we have
\[
\Sigma_2'  = \sum_{\substack{\v \in \ZZ^n\\ m_1 | \v}}  \hat{w}(\tfrac{P}{m} \v) |\Sigma_{m_1}(\bar{m}_2\v)||\Sigma_{m_2}(\bar{m}_1\v)|  \ll m_1^{n-1+\ve} \sum_{\w \in \ZZ^n} \hat{w}(\tfrac{P}{m_2} \w)|\Sigma_{m_2}(\w)| \ll m^{n-1+\ve},
\]
by \eqref{eq:sigma_trivial} and the induction hypothesis.

Furthermore, by \eqref{eq:sigma_squareroot} one has
\[
\Sigma_1 \ll m^{n/2+\ve} \sum_{\v \in \ZZ^n} \hat{w}(\tfrac{P}{m} \v) \ll m^{n/2+\ve} \frac{m^n}{P^n} \leq m^{n-1+\ve}, 
\]
so we have proved the bound \eqref{eq:emporia}. 
\end{proof}

\begin{proof}[Proof of Proposition \ref{prop:trivial}]
We may write
\[
|\cT(q,z)| \leq \sum_{\x \in \ZZ^n} w(\x/P) \left|Z\big(b_1,f(\x)+g(\x),f(\x)\big)\right|,
\textrm{ where }
Z(r,x,y) = \starsum_{s_1,s_2 = 1}^{r} e_{r}\big(s_1 x - s_2 y\big).
\]
It is easy to see that $Z(r,x,y)$ is a multiplicative function of $r$ for fixed $x$ and $y$, and that for prime $p$ one has $Z(p,x,y) = 1$ if $x$ and $y$ are both non-zero (mod $p$), $Z(p,x,y) = 1-p$ if precisely one of $x$ and $y$ vanish (mod $p$) and $Z(p,x,y) = (p-1)^2$ if $x \equiv y \equiv 0 \bmod p$. It follows that
$
|Z(r,x,y)| \leq (r,x)(r,y).
$ 
Thus we have
\[
|\cT(q,z)| \leq \sum_{\x \in \ZZ^n} w(\x/P) (b_1,f(\x)+g(\x))(b_1,f(\x)) \leq b_1 \sum_{\x \in \ZZ^n} w(\x/P) (b_1,f(\x)).
\]
But
\[
\sum_{\x \in \ZZ^n} w(\x/P) (b_1,f(\x)) \leq \sum_{m \mid b_1} m \sum_{\substack{\x \in \ZZ^n \\ m \mid f(\x)}} w(\x/P) \ll P^n d(b_1) \ll P^{n+\ve}
\]
by Lemma \ref{lem:india}.
\end{proof}


\section{Finalisation of van der Corput differencing}
\label{sec:vdC2}

In this section, we use the results from Section \ref{sec:expsums_final} to estimate the quantity
\begin{equation}
\label{eq:sum_h}
\sum_{|\h| \leq H} \cT_{a,\h}(q,z)
\end{equation}
occurring in Section \ref{sec:vdC}. $q$ will be a fixed arbitrary number throughout this section. Our final aim is to prove the minor arcs bound in Proposition \ref{prop:minor_smooth}, so we recall that $F$ is a quartic polynomial whose leading form $F_0$ is non-singular. From now on, the implied constants in our estimates are allowed to depend on the height $\Vert F \Vert_P$ and the quantity $M$ introduced in \eqref{eq:M_1}--\eqref{eq:M_2}.
We shall prove the following result:

\begin{proposition}
\label{prop:sum_h}
Provided that $q \leq P^{2-4/(n+2)}$, we have the bound 
\begin{equation}
\label{eq:maundy}
\sum_{|\h| \leq H} \cT_{a,\h}(q,z)
\ll b_1 P^{n+\ve} \left\{
1 + \sum_{\eta=0}^{n-1} \cY_\eta 
\right\},
\end{equation}
where
\[
\cY_\eta := \frac{H^{n-\eta}}{q^{(n-\eta)/2}} \cX_\eta
\]
for any $0 \leq \eta \leq n-1$, with
\begin{equation*}
\cX_\eta := V^{n-\eta}  
+ b_1^{1/2}\left(q_3^{1/2} V^{n-\eta - 1}
+ (q_3H)^{(n-\eta)/3}\right)
\end{equation*}
for $0 \leq \eta \leq n-2$, and
\[
\cX_{n-1} := V + q_3H^{1/3}.
\]
Here $V=qP^{-1}\max\{1,(HP^3|z|)^{1/2}\}$, as per \eqref{eq:V def}.

\end{proposition}
The terms $\cY_\eta$ appearing on the equation \eqref{eq:maundy} roughly correspond to $\h$'s for which the maximum dimension of the singular locus for a system of projective varieties corresponding to $F$, the leading form of $F_\h$ or their complete intersection is $\eta-1$. Given the nature of the bounds presented, the terms $1, \cY_0,\cY_{n-1} $ and $\cY_{n-2}$ will be critical to our analysis. In this paper, the main contribution is presented by the terms $1$ and $\cY_0$. Moreover, the term $b_1^{1/2}q_3^{1/2}V^{n-1}$ turns out to give the dominating contribution in $\cX_0$, which is partly why we are unable to establish the result for $n=29$.

Given $q = b_1q_2$ and $\h \in \ZZ^n$, we write 
\[
b_1 = b_{1,0}(\h) \dotsb b_{1,n}(\h) \quad \text{and} \quad q_2 = q_{2,0}(\h) \dotsb q_{2,n}(\h),
\]
where 
\[
 b_{1,i}(\h) = \prod_{\substack{p \Vert b_1\\ s_p'(\h) = i-1}} p \quad \text{and} \quad
q_{2,i}(\h) = \prod_{\substack{p^e\mid q_2\\ s_p'(\h) = i-1}} p^e.
\]
Here $s_p'(\h)=s_p'(F,F_\h)$. Firstly, since $q$ is fixed, there are then only $O(P^\ve)$ possible different choices for $2(n+1)$-tuples
\[
(b_{1,0}(\h),\dotsc,b_{1,n}(\h),q_{2,0}(\h),\dotsc,q_{2,n}(\h)) \in \NN^{n+1} \times \NN^{n+1}.
\]
Thus there is one such tuple $(b_{1,0},\dotsc,b_{1,n},q_{2,0},\dotsc,q_{2,n})$, for which
\begin{align*}
\sum_{|\h| \leq H} \cT_{a,\h}(q,z) &\ll P^\ve \sum_{\h \in \cH'} \cT_{a,\h}(q,z)
 = P^\ve \sum_{s=-1}^{n-1} \sum_{\h \in \cH'_s} \cT_{a,\h}(q,z),
\end{align*}
say, where 
\[
\cH' := \{ \h \in \ZZ^n \mid |\h| \leq H, b_{1,i}(\h) = b_{1,i} \text{ and } q_{2,i}(\h) = q_{2,i}  \text{ for all } i = 0,\dotsc,n\}
\]
 is a subset of $\cH=\{|\h|\leq H\} $ and
\[
\cH'_s := \{ \h \in \cH' \mid s_\infty'(\h) = s\}
\]
for each $s=-1,\dotsc,n-1$. Furthermore, put
\[
\tilde{q}_{2,i} = \prod_{p\mid q_{2,i}}p.
\]
for any $i$. We note that since $s'_p(\h) \geq s'_\infty(\h)$, we have $\cH'_s = \varnothing$ unless $b_{1,i} = q_{2,i} = 1$ for $i \leq s$. To estimate the cardinality of the set $\cH'_s$, we need the following result.

\begin{lemma}
\label{lem:salberger}
The set $V_{v,i} = \{\h \in \AA^n_{\FF_v} \mid s'_v(F,F_\h) \geq i-1\}$ is a closed subvariety in $\AA^n_{\FF_v}$ of degree $O(1)$, and there is a constant $C$, with $\log C \ll_n \log \Vert F_0 \Vert$, such that
\[
\dim(V_{v,i}) \leq n -i
\]
as soon as $v=\infty$ or $v=p > C$. More precisely, $V_{v,i}$ is the affine cone over a closed subvariety of $\PP^{n-1}_{\FF_v}$ of dimension $n-i-1$.
\end{lemma}
 
\begin{rem}
If $F$ is a form of degree $d_1$, then the homogeneous part of $F_\h$ of degree $d_1-1$ is the form
\[
F^\h(\x) := \h.\nabla F(\x).
\]
For any $i$, the set $V_{v,i}$ under consideration may be written as a union of the three sets $V^{(j)}_{v,i}$, say, for $j=1,2,3$, defined by the conditions
\[
s_v(F) \geq i-1, \quad s_v(F^\h) \geq i-1, \quad \text{and} \quad s_v(F,F^\h) \geq i-1, 
\]
respectively. It is clear that $V^{(1)}_{v,i} = \AA^n$ or $\emptyset$, and 
the latter holds as soon as $i \geq 1$ and $v \gg_F 1$. The fact that $V^{(2)}_{v,i}$ satisfies the conclusions in the lemma was proven in \cite{Browning-Heath-Brown09}. Thus it suffices to prove the assertion with $V_{v,i}$ replaced by $V^{(3)}_{v,i}$. 
Such a statement was first proven by Salberger in an unpublished note \cite{Salberger-Integral_points}. The same arguments were used to prove a similar result in a paper by the first author \cite[Lemma 2.2]{Marmon10}, the proof of which may be used with only minor  modifications to obtain a proof of Lemma \ref{lem:salberger}. We omit the proof here, but it is worth pointing out that it uses a version of Bertini's theorem that is only valid in characteristic zero, hence producing a condition on $p$ that was not present in the corresponding result in \cite{Browning-Heath-Brown09}.  
\end{rem}

To estimate $\#\cH_s'$, we follow the argument in \cite[\S 7]{Browning-Heath-Brown09}, with a difference that the absolute constant $c$ in \cite{Browning-Heath-Brown09} is replaced by a constant $C =O_{\Vert F_0 \Vert}(1)$. Since 
\[
\cH_s' \subset \left\{ \h \in V_{\infty,s+1}\cap [-H,H]^n \mid [\h]_p \in V_{p,i} \text{ for all } p \mid b_{1,i}q_{2,i} \right\},
\] 
we may use \cite[Lemma 4]{Browning-Heath-Brown09} to show that
\begin{equation}
\label{eq:cardH_s}
\#\cH_s' \ll q^\ve\max_{s+1 \leq \eta \leq n} \frac{H^{n-\eta}}{\displaystyle\prod_{i = \eta+1}^n (b_{1,i} \tilde{q}_{2,i})^{i-\eta}}.
\end{equation}

\begin{proof}[Proof of Proposition \ref{prop:sum_h}]
Let $-1\leq s \leq n-1$ be fixed and put
\[
\cU_s = \sum_{\h \in \cH_s} \cT_{a,\h}(q,z).
\]
Each exponential sum $\cT_{a,\h}(q,z)$ is an instance
\[
\cT(q,z) = \cT_n(a,q,z;F,F_\h,W_\h,P)
\]
of the general sum introduced in Section \ref{sec:expsums_intro}. We may assume that the condition \eqref{eq:F_nonzero} holds. Indeed, if $b_1' = (b_1,\content(F_0))$, then $b_1' \leq \Vert F_0 \Vert \leq \Vert F \Vert_P$, so all exponential sums $T(b',\v)$ satisfy a trivial upper bound $O_{\Vert F \Vert_P}(1)$. By our convention on the implied constant, it suffices to bound $\cT_{a,\h}(q/b_1',z)$. Moreover, one sees that $\Vert F_\h \Vert_P \ll H$, so that the condition \eqref{eq:g_height'} is also verified, at least upon replacing $H$ with $CH$ for some $C \ll 1$.

Choose $\eta$ such that  the maximum in \eqref{eq:cardH_s} is attained. Let us first treat the case where $\eta \leq n-2$. In this case, for each $\h \in \cH_s'$, we apply Lemma \ref{lem:subspace}, with $\Pi$ being the set of primes $p \mid q$, so that $r = \omega(q)$, and with $\{F_1,F_2\} = \{F_0,F_0^\h\}$. We obtain a lattice $\Lambda_\eta$ of rank $n-\eta$ and a basis $\e_1,\dotsc,\e_{n-\eta}$ for $\Lambda_\eta$, with the property that for any $\t \in \ZZ^{n}$, the polynomials
\[
\tilde f_\t(y_1,\dotsc,y_{n-\eta}) := F(\t + \textstyle\sum y_i \e_i) \quad \text{and} \quad
\tilde g_\t(y_1,\dotsc,y_{n-\eta}) := F_\h(\t + \textstyle\sum y_i \e_i)
\]
satisfy
\begin{equation}
\label{eq:iterated_bertini}
s'_v(\tilde{f}_\t,\tilde{g}_\t) = \max \{-1,s'_v(F,F_\h)-\eta\}
\end{equation}
for all $v \in \{\infty\} \cup \Pi_{cr}$. Indeed, denote by $\tilde{F}_0$ the leading form of $\tilde{f}_\0$, which is simultaneously the leading form of $\tilde{f}_\t$ for all $\t$. Similarly, let $\tilde{G}_0$ be the common leading form of the $\tilde{g}_\t$. Then we see that for each $v \in \{\infty\} \cup \Pi_{cr}$, we have
\[
V_v(\tilde{F}_0) \cong V_v(F_0) \cap \Lambda_\eta,\quad V_v(\tilde{G}_0) \cong V_v(G_0) \cap \Lambda_\eta
\quad \text{and} \quad V_v(\tilde{F}_0,\tilde{G}_0) \cong V_v(F_0,G_0) \cap \Lambda_\eta,
\]
so \eqref{eq:iterated_bertini} is precisely the condition that is asserted in Lemma \ref{lem:subspace}. It also follows that $\deg{\tilde{F}_0} = 4$ and $\deg(\tilde{G}_0) = 3$, and that
\[
(b_1,\content(\tilde{F}_0)) = 1.
\] 

As there are at most $O(q^\ve) = O(P^\ve)$ choices for a basis satisfying \eqref{eq:basis1}, there is in fact one such choice for which we may write
\[
\cU_s \ll P^\ve \dashsum_{\h \in \cH_s'} |\cT_{a,\h}(q,z)|,
\] 
where the superscript $'$ denotes that the sum is taken over those vectors $\h$ occurring in the original sum for which the condition \eqref{eq:iterated_bertini} holds for the designated basis $\e_1,\dotsc,\e_n$. 
For such an element $\h$, we may partition the sum over $\x \in \ZZ^n$ defining $\cT_{a,\h}(q,z)$ into cosets $\Lambda_{\eta} + \t$ of the lattice $\Lambda_{\eta}$, where $\t$ runs over some subset $T_\eta \subset (\ZZ\cap [-P,P])^n$. We claim that the set $T_\eta$ may be chosen of cardinality $O(P^{\eta})$. To see this, 
consider a general linear combination 
\[
\x = \sum_{i=1}^n y_i \e_i.
\]
Denoting by $\pi_i$, for each $1 \leq i \leq n$, the projection onto the orthogonal complement of the subspace spanned by the vectors $\e_j$, $j \neq i$, one then has 
\[
\Vert \x \Vert \geq \Vert \pi_i \x \Vert = |y_i| \Vert \pi_i \e_i \Vert = |y_i| \frac{|\det(\Lambda)|}{|\det(\Lambda_i)|}, 
\]
where $\Lambda \subseteq \ZZ^n$ denotes the full-dimensional lattice spanned by $\e_1,\dotsc,\e_n$ and $\Lambda_i$ the lattice spanned by all $\e_j$, $j\neq i$. Now it follows from the conditions \eqref{eq:basis1} and \eqref{eq:basis2} that 
\begin{equation}
\label{eq:annalinda}
|y_i| \ll \frac{\Vert \x \Vert}{L}.
\end{equation}
From this we conclude that we may choose $T_\eta$ to be the set of vectors of the form 
$\sum_{i=n-\eta + 1}^{n} \lambda_i \e_i$, 
with $\lambda_i \ll P$.
 
Defining new weight functions
\[
\tilde{w}_{\t}(y_1,\dotsc,y_{n-\eta}) := W_\h(P^{-1}\t + L^{-1} \sum y_i\e_i),
\]
we then have
\begin{equation}
\label{eq:iterated_slicing}
\cT_{a,\h}(q,z) \leq  \sum_{\t \in T_\eta} \cT_{a,\h,\t}(q,z), \textrm{ where }\cT_{a,\h,\t}(q,z) = \cT_{n-\eta}(q,z;\tilde{f}_\t,\tilde{g}_\t,\tilde{w}_\t,P/L)
\end{equation} 
in the notation of \eqref{eq:Tuzdef}. The verification that $\tilde{w}_\t \in \cW_{n-\eta}$ for $\t \ll P$ follows as in the proof of \cite[Prop. 2]{Browning-Heath-Brown09}. Indeed, the deciding property of the basis $\e_1,\dotsc,\e_{n-\eta}$ is the bound \eqref{eq:annalinda} derived above.
One also sees, following \cite{Browning-Heath-Brown09}, that 
\[
 \Vert \tilde{g}_\t \Vert_{P/L} \ll L^3 \Vert F_\h \Vert_P \ll P^\ve H \Vert F \Vert_P \ll P^\ve H.
\]
Thus the conditions \eqref{eq:F_nonzero} and \eqref{eq:g_height'} are verified, at least upon replacing $H$ by $\tilde{H}$ for some $\tilde{H} \ll P^\ve H$. By \eqref{eq:iterated_bertini} and our assumption that $\eta \geq s+1$, we also have \eqref{eq:nonsing}, so that $\cT_{a,\h,\t}(q,z)$ indeed qualifies as an instance of the exponential sum treated in Proposition \ref{prop:1}, with dimension $n$ replaced by $n-\eta$. At this point it is worth noting that this is the main reason behind obtaining bounds for exponential sums defined over an arbitrary number of variables including $1$ and $2$ in Sections \ref{sec:cubefree} through \ref{sec:expsums_final}.
Using \eqref{eq:cardH_s}, we may now write 
\begin{equation}
\label{eq:brysselkex}
\cU_s \ll  \frac{P^{\eta + \ve} H^{n-\eta}}{\displaystyle\prod_{i = \eta+1}^n (b_{1,i}\tilde{q}_{2,i})^{i-\eta}} \cdot \dashmax_{\h \in \cH_s'} \ \max_{\t \in T_\eta} \cT_{a,\h,\t}(q,z).
\end{equation}
By \eqref{eq:iterated_bertini} it follows, again arguing as in the proof of \cite[Prop. 2]{Browning-Heath-Brown09}, that for each $\t$ occurring in \eqref{eq:brysselkex} one has
\begin{align*}
\cD_{\tilde{f}_\t,\tilde{g}_\t}(q) &= \prod_{i=1}^{n-\eta} \prod_{\substack{p \vert b_1 \\s_p'(\tilde{f}_\t,\tilde{g}_\t) = i-1}} p^{i/2} \prod_{\substack{p \vert q_2 \\s_p'(\tilde{f}_\t,\tilde{g}_\t) = i-1}} p^{i} \ll 
q^\ve \prod_{i=\eta +1}^{n} \prod_{\substack{p \vert b_1 \\s_p'(F,F_\h) = i-1}} p^{(i-\eta)/2} \prod_{\substack{p \vert q_2 \\s_p'(F,F_\h) = i-1}} p^{i-\eta}  \\
&\ll q^\ve \prod_{i=\eta+1}^{n} (b_{1,i}^{1/2} \tilde{q}_{2,i})^{i-\eta}.
\end{align*}
Inserting the bound from Proposition \ref{prop:1} into \eqref{eq:brysselkex} thus yields
\[
\cU_s \ll b_1 P^{n+\ve} \frac{H^{n-\eta}}{q^{(n-\eta)/2}} \cX_{\eta},
\]
where $\cX_\eta$ is the quantity defined in the statement of the proposition.

It remains to deal with the possibility that $\eta \in \{n-1,n\}$. Suppose first that $\eta = n-1$. Since $s \leq \eta-1 = n-2$, the variety defined by $F = F_\h = 0$ in $\AA^{n}_\QQ$ has dimension $n-2$ for all $\h$ occurring in the sum. We apply Lemma \ref{lem:subspace} again, although the non-singularity condition in the lemma is automatic in this case, since the resulting varieties all have dimension at most zero. Using similar arguments as in the previous case, we may then write
\begin{equation}
\label{eq:U_s_n-1}
\cU_s \ll  \frac{P^{\ve} H^{n-1}}{b_{1,n}\tilde{q}_{2,n}} \cdot \dashmax_{\h \in \cH_s} \ \sum_{\t \in T_{n-1}} \cT_{a,\h,\t}(q,z).
\end{equation}
For each $\t$ occurring in the sum, we may apply Proposition \ref{prop:1_n=1} to obtain
\[
\cT_{a,\h,\t}(q,z) \ll b_1 q^{-1/2} \cD(q_2) P^{1+\ve} (b_1,\resultant(\tilde{f}_\t,\tilde{g}_\t))^{1/2}(V+(q_3H)^{1/3}). 
\]
In order to capitalise the average over $\t$, we use the explicit description
\[
T_{n-1} = \left\{\t = \sum_{i=2}^{n} z_i \e_i;\ z_i \in \ZZ, |z_i| \ll P\right\}.
\]
The resultant $\resultant(\tilde{f}_\t,\tilde{g}_\t)$ is then a polynomial $\cR(\z)$ in the variables $z_2,\dotsc,z_{n}$ with integer coefficients. The fact that $F$ and $F^\h$ intersect completely implies that $\cR(\t)$ does not vanish identically in $\t$. Similarly, the congruence $\cR(\t) \equiv 0 \bmod p$ defines a hypersurface in $\AA^{n-1}_{\FF_p}$ unless either $p \mid b_{1,n}$ or $p \ll P^\ve$.
Thus we write
\begin{align*}
\sum_{\t \in T_{n-1}} \left(b_1,\resultant(\tilde{f}_\t,\tilde{g}_\t)\right)^{1/2} 
&\ll P^\ve b_{1,n}^{1/2} \sum_{\substack{\z \in \ZZ^{n-1}\\ \z \ll P^{1+\ve}}} \left(b_1^*,\cR(\z)\right)^{1/2}\\
&= P^\ve b_{1,n}^{1/2} \sum_{\ell \mid b_1^*} \ell^{1/2} \#\{\z \in \ZZ^{n-1} \mid \z \ll P^{1+\ve},\ \ell \mid \cR(\z)\},
\end{align*}
where we have put $b_1^* = b_{1,n}^{-1}b_1 = b_{1,0}\dotsb b_{1,n-1}$.
By \cite[Lemma 4]{Browning-Heath-Brown09} we have
\[
\#\{\z \in \ZZ^{n-1} \mid \z \ll P^{1+\ve},\ \ell \mid \cR(\z)\} \ll P^\ve(P^{n-2} + P^{n-1}\ell^{-1}),
\]
which in turn gives the bound
\[
\sum_{\t \in T_{n-1}} \left(b_1,\resultant(\tilde{f}_\t,\tilde{g}_\t)\right)^{1/2} \ll b_1^{1/2} P^{n-2+\ve} + b_{1,n}^{1/2} P^{n-1+\ve}.
\]
Inserting this into \eqref{eq:U_s_n-1}, and arguing as above that $\cD_{\tilde{f}_\t,\tilde{g}_\t}(q) \ll q^\ve \tilde{q}_{2,n}$, gives
\begin{align*}
\cU_s &\ll b_1 q^{-1/2} \frac{H}{ b_{1,n}\tilde{q}_{2,n}}P^{n+\ve} \left(b_1^{1/2} \tilde{q}_{2,n} \frac{V+(q_3H)^{1/3}}{P} +  b_{1,n}^{1/2} \tilde{q}_{2,n}(V+(q_3H)^{1/3})\right)
\\
&\ll b_1 \frac{H}{q^{1/2}} P^{n+\ve}(V+(q_3H)^{1/3})\left(1 + \frac{q^{1/2}}{P}\right) \ll b_1 \frac{H}{q^{1/2}} P^{n+\ve}(V+(q_3H)^{1/3})
\end{align*}
as claimed.

Finally, suppose that $\eta = n$. Invoking the hypothesis that $q \leq P^{2-4/(n+2)}$, we may use Proposition \ref{prop:trivial} for arbitrary $\h$, yielding the bound
\[
\cU_s \ll b_1 P^{n+\ve}.
\]
Summing the contribution from all possible values of $s$ and $\eta$, we arrive at the bound stated in the lemma.
\end{proof}

\section{Weyl differencing}
\label{sec:weyl} 
 
We now describe an approach parallel to our main argument. To bound the exponential sum in \eqref{eq:S_h}, we shall then use a result of Browning and Prendiville-\cite[Lemma 3.3]{Browning_Prendiville}, with $d=3$. This can be summarised as applying three consecutive Weyl differencing steps to the cubic exponential sum $\cS_\h(\alpha)$. Whenever $\h \neq \0$,  this produces the bound
\begin{equation}
\label{eq:BP}
\cS_\h(a/q+z) \ll P^{n+\ve}\Big(P^{-2} + q|z|H + qP^{-3} + (q|z|P^3)^{-1}\Big)^{n/8}.
\end{equation}
(It follows from the nonsingularity of $F$ that the cubic part of $F_\h$ does not vanish, for example as a special case of Lemma \ref{lem:salberger}, so it is indeed a cubic polynomial.) Inserting this bound into Lemma \ref{lem:hanselmann_classic} and adding the contribution from $\h = \0$, yields
\begin{multline}
\label{eq:fotsvamp}
\cI(q,t) \ll P^{n-1/2 + \ve}q(t+\tfrac{1}{HP^3})H^{-(n-1)/2}\\
\times\left\{1 + H^{n/2}\Big(P^{-n/8} + (qtH)^{n/16} + (qP^{-3})^{n/16} + (qtP^3)^{-n/16}\Big)\right\}.
\end{multline}
 
Alternatively, we may apply Weyl differencing four times to the original quartic exponential sum. We then obtain a bound
\[
S(a/q+z) \ll P^{n+\ve} \Big(q|z| + q^{-1}|z|^{-1}P^{-4}\Big)^{n/24}, 
\]  
or, if one will,
\begin{equation}
\label{eq:ariana}
\cI(q,t) \ll P^{n+\ve}qt\Big(qt + q^{-1}t^{-1}P^{-4}\Big)^{n/24}. 
\end{equation}

\section{Bounds for the minor arc contribution, nonsingular case}
\label{sec:minor}

In this section we shall combine several different approaches to achieve the bound $$S_\fm = O(P^{n-4-\psi})$$ in Proposition \ref{prop:minor_smooth}. First we observe that when estimating the integrals comprising $S_\fm$, we may replace the integrand $|p_q(z)S(q,z)|$ by $|S(q,z)|$. Indeed, by \eqref{eq:pbound1}, this happens at the cost of introducing a dependence on the parameter $\theta$ in the implied constant. 
We divide the ranges for $q$ and its factors $b_1$, $b_2$, $q_3$ into $O(Q^\ve)$ dyadic intervals 
\begin{equation}
\label{eq:dyadic}
R < q \leq 2R,\quad  R_1 < b_1 \leq 2R_1,\quad R_2 < b_2 \leq 2R_{2},\quad R_3 < q_3 \leq 2R_3,
\end{equation}
(Note that our notation differs from that in \cite[(9.4)]{Browning-Heath-Brown09}).
We then put
\begin{equation}
K(t,R,\R) := \sum_{q, \eqref{eq:dyadic}} \int_{t \leq |z| \leq 2t} |S(q,z)| \,dz = \sum_{q, \eqref{eq:dyadic}} \cI(q,t)
\end{equation}
for any $t$,
where $\R := (R_1,R_2,R_3)$. The quantity $K(t,R,\R)$ is relevant for the minor arcs estimate only if 
\begin{equation}
\label{eq:Rt_condition}
R \leq Q, \quad R_1 R_2 R_3\asymp R, \quad t \leq (RQ)^{-1+\theta},
\end{equation}
where $A\asymp B$ stands for the usual notation that $A\ll B\ll A$.
In addition, either
\begin{equation}
\label{eq:minor_arcs_condition}
t \geq P^{-4+\Delta} \quad \text{or} \quad R \geq P^{\Delta}.
\end{equation}
 Clearly we then have
\begin{equation}
\label{eq:minor_dyadic}
S_\fm \ll P^\ve \max_{t,R,\R \text{ satisfying } \eqref{eq:Rt_condition} + \eqref{eq:minor_arcs_condition}} K(t,R,\R) + O(P^{n-5}),
\end{equation}
where the error term arises from very small values of $t$, say $|t| \leq P^{-10}$. 

What follows in this section is a delicate comparison of various estimates that we have obtained so far to establish the bound 
$$K(t,R,\R)\ll P^{n-4-\ve},$$
for every choice $t,R,\R$ relevant to the minor arcs contribution. We will start by comparing three main bounds. The first bound is obtained by an application of van der Corput differencing followed by Weyl differencing in \eqref{eq:fotsvamp}, which gives rise to Proposition \ref{prop:v1w3} below. This bound is sufficient, unless at least one of \eqref{eq:vw1}, \eqref{eq:vw2} or \eqref{eq:vw3} is violated, which will be assumed henceforth. When $t$ is not very small, we further compare the Weyl differencing bound in Proposition \ref{prop:w4} (obtained from \eqref{eq:ariana}) and our main bound in Proposition \ref{prop:v} (a consequence of Proposition \ref{prop:sum_h}). This delicate optimisation is carried in Propositions \ref{prop:avdC} and \ref{prop:prop10}. 

Proposition \ref{prop:avdC} provides a satisfactory bound for all the contribution from all terms in \eqref{eq:maundy}, except from a part of bounds coming from terms $\cY_{n-2}$ and $\cY_0$. The comparison of these bounds with those from Proposition \ref{prop:w4} is rather challenging. To this end, we introduce variables $Z$ and $\alpha$ via \eqref{eq:opti}, and view the problem of comparing these contributions as a rational linear optimisation problem in $Z$ and $\alpha$. We are thus able to calculate an explicit minimum value for each of these, using Mathematica. Finally, when $t$ is very small, our averaged van der Corput bound in Proposition \ref{prop:v} is rather wasteful, and we instead need to use an alternate point-wise van der Corput bound in Lemma \ref{lem:pointwise} and compare it with Proposition \ref{prop:w4} to achieve a satisfactory bound.

Let us put $B_1 = R_1R_2$ for short, and introduce the quantity
\[
B_2 := B_1 R_3^{1/3} = R_1 R_2 R_3^{1/3}.
\]
The significance of $B_2$ is explained by the bound
\begin{equation}
\label{eq:lem20}
\sum_{q, \eqref{eq:dyadic}} 1 \ll R_1 R_2^{1/2} R_3^{1/3} \leq B_2,
\end{equation}
following from Lemma \ref{lem:20}. In other words, $B_2$ denotes a suitable upper bound for the number of $q$'s satisfying \eqref{eq:dyadic}. We seemingly give up a factor of $R_2^{1/2}$ here. However, this is firstly due to the fact that this saving is not necessary for us, and secondly, since this simplifies our linear optimisation process later by allowing us to work with only two variables $Z$ and $\alpha$ (see \eqref{eq:opti}).
We next put 
\(
T := Rt
\)
and note that \eqref{eq:Rt_condition} implies
\[
T \leq Q^{-1+2\theta} \leq P^{-8/5-\phi/2}
\]
if only $\phi \geq 8\theta$, say, which we may assume from now on. 

We are now ready to record the bounds for $K(t,R,\R)$ obtained by the Weyl differencing procedure in Section \ref{sec:weyl}. The first alternative (one van der Corput differencing step followed by three Weyl differencing steps) provides the following bound:

\begin{proposition}
\label{prop:v1w3}
We have the bound
\[
K(t,R,\R) \ll P^{n-4-\phi/4}
\]
provided that 
$n \geq 21$ and 
the following three conditions are satisfied:
\begin{gather}
\label{eq:vw1}
B_2 \leq P^{\frac{4n}{45}-\frac{179}{90}},\\
\label{eq:vw2}
T \geq B_2^{\frac{16}{n-17}}P^{-\frac{3n-59}{n-17}+4\phi},\\
\label{eq:vw3}
T \geq B_2^{-2/(n+1)} R^{1-2/(n+1)} P^{-3-1/(n+1)}.
\end{gather}
\end{proposition}

\begin{proof}
By \eqref{eq:fotsvamp} we get
\begin{equation}
\label{eq:lamisil}
\begin{split}
K(t,R,\R) \leq \sum_{q, \eqref{eq:dyadic}} &\cI(q,t)
\ll B_2 R (t+\tfrac{1}{HP^3})P^{n-1/2 + \ve}H^{-(n-1)/2}\\
&\hphantom{\ll} \times \left\{1 + H^{n/2}\Big(P^{-n/8} + (RtH)^{n/16} + (RP^{-3})^{n/16} + (RtP^3)^{-n/16}\Big)\right\}.
\end{split}
\end{equation}
We put
\[
H := B_2^{2/(n-1)}T^{2/(n-1)}P^{7/(n-1) + c\phi}=(B_2TP^{7/2} )^{2/(n-1)}P^{c\phi},
\]
for some suitable constant $c$. The upper bound $T\ll P^{-8/5-\phi/2} $ implies that $1 \leq H \leq P$ as required, for $n$ in the required range. Furthermore, the lower bound \eqref{eq:vw3} implies that we always have $t \geq (HP^3)^{-1}$, allowing us to simplify \eqref{eq:lamisil} to
\begin{equation}
\label{eq:arboga}
P^{-n+4}K(t,R,\R) \ll \frac{B_2 T P^{7/2 + \ve}}{ H^{(n-1)/2}}\left\{1 + H^{n/2}\Big(P^{-n/8} + (TH)^{n/16} +(TP^3)^{-n/16}\Big)\right\}.
\end{equation}
Our choice of $H$ implies that the contribution from the term '1' inside the brackets is 
\[
\ll P^{-\phi/2},
\]
say, provided only that $c \geq 2/(n-1)$ and $\phi \geq 2\ve$. 

Note that the hypothesis \eqref{eq:vw1} implies that
\begin{equation}
\label{eq:luskam}
\tR TP^{7/2} \leq P^{4(n-1)/45 - \phi/2}.
\end{equation}
The contribution from the term $P^{-n/8}$ is then
\[
\ll (\tR TP^{7/2})^{n/(n-1)} P^{-n/8 + c\phi/2 + \ve} \ll P^{-13n/360 + (c-1)\phi/2 + \ve},
\]
which is clearly admissible for any $n$, provided $c \leq 1$ and $\ve$ is chosen small enough.

The hypothesis \eqref{eq:vw1} also ensures that the term $(TH)^{n/16}$ gives an admissible contribution. Indeed, by \eqref{eq:luskam}, this contribution is 
\begin{align*}
&\ll B_2TP^{7/2+\ve}H^{(n+8)/16}T^{n/16} \ll (B_2TP^{7/2})^{9n/(8(n-1))} P^{-n/10 - n\phi(1/2-c)/16+c\phi/2 + \ve}\\
 &\ll P^{- n\phi(1/2-c)/16 + \ve} \ll P^{-\phi/4}
\end{align*}
provided that $n\geq 16, c \leq 1/8$ and $\phi \geq 8\ve$.

Finally, under the hypothesis \eqref{eq:vw2}, the contribution from the term $(TP^3)^{-n/16}$ is
\begin{align*}
&\ll (B_2TP^{7/2})^{n/(n-1)} (TP^3)^{-n/16} P^{\ve+c\phi/2}\\
&\ll B_2^{n/(n-1)}P^{-n(3n-59)/(16(n-1)) }T^{-n(n-17)/(16(n-1)) }P^{\ve+c\phi/2} \ll P^{-\phi+\ve+c\phi/2}\ll P^{-\phi/4},
\end{align*}provided that $n\geq 21 $, $c \leq 1$ and $\phi \geq 4\ve$. It is thus possible to choose $c$ and $\ve$ such that all these bounds are satisfied.
\end{proof} 

\begin{rem}
In the case $n = 30$, the conditions in Proposition \ref{prop:v1w3} read
\[
\tR \leq P^{61/90}, \quad T \geq \tR^{16/13} P^{-31/13+4\phi} \quad \text{and} \quad T \geq \tR^{-2/31}R^{29/31}P^{-94/31}.
\] 
\end{rem}

The approach featuring four consecutive Weyl differencing steps produces the following bound, where $\delta > 0$ is a parameter to be chosen at a later stage.

\begin{proposition}
\label{prop:w4}
Suppose that 
\begin{equation}
\label{eq:weyl_condition}
T \geq \min\{B_2^{24/(n-24)} P^{-4 + \delta},P^{-2}\}.
\end{equation}
Then we have the estimate
\[
K(t,R,\R) \ll B_2 T^{1+n/24}  P^{n+\ve} + P^{n-4-\delta(n/24-1)+\ve}.
\]
\end{proposition}

\begin{proof}
By \eqref{eq:ariana} we get
\begin{equation*}
K(t,R,\R) \leq \sum_{q, \eqref{eq:dyadic}} \cI(q,t) \\
\ll B_2 T P^{n+\ve}\Big(T + T^{-1}P^{-4}\Big)^{n/24}.
\end{equation*}
Let us assume that $T \geq B_2^{24/(n-24)} P^{-4 + \delta}$. Then we have 
\begin{align*}
K(t,R, \R) &\ll B_2 T^{1+n/24} P^{n+\ve} + B_2 T^{1-n/24}P^{n-n/6+\ve} \\
&\ll B_2 T^{1+n/24} P^{n+\ve} + P^{n-4-(n-24)\delta/24 + \ve}.
\end{align*}
Alternatively, if $T \geq P^{-2}$, say, then 
$
T^{-1}P^{-4} \leq T, 
$
so we get
\[
K(t,R,\R) \ll B_2 T^{1+n/24}  P^{n+\ve}.
\]
\end{proof}

Next we give the estimate for $K(t,R,\R)$ coming from our main approach, where only one van der Corput differencing step was carried out. To this end we define the quantity
\[
B_3 := B_1^{1/2} R_3^{1/3},
\]
assuming the role of $B_2$ in the previous two results. The saving here by of a factor of $R_1^{1/2}$ over $B_2$ is a major achievement of our van der Corput differencing method devised in Section \ref{sec:vdC}.

\begin{proposition}
\label{prop:v}
Provided that $R \leq P^{2-4/(n+2)}$, one has
\begin{equation*}
K(t,R,\R) \ll B_3 R(t + \tfrac{1}{HP^3}) P^{n-1/2 + \ve} H^{-(n-1)/2} \left(1 + \cY_0 + \dotsb +  \cY_{n-1} \right)^{1/2},
\end{equation*}
where 
\[
\cY_\eta := \frac{H^{n-\eta}}{R^{(n-\eta)/2}} \cX_\eta
\]
for any $0 \leq \eta \leq n-1$, with
\begin{equation*}
\cX_\eta := V^{n-\eta}  
+ R_1^{1/2}\left(R_3^{1/2} V^{n-\eta - 1}
+ (R_3H)^{(n-\eta)/3}\right)
\end{equation*}
for $0 \leq \eta \leq n-2$, and
\[
\cX_{n-1} := V + (R_3H)^{1/3}, \textrm{ where }
V = \frac{R}{P}\max\{1,\sqrt{tHP^3}\}.
\]
\end{proposition} 

Note that there is a slight abuse of notation in that we have reused the same letters for the quantities $V$, $\cX_\eta$ and $\cY_\eta$ although replacing $q,b_1,q_3$ and $|z|$ by $R,R_1,R_3$ and $t$, respectively, in their definition.

\begin{proof}
Suppose that $q$ satisfies \eqref{eq:dyadic}, and recall the bound for $\cI(q,t)$ from Lemma \ref{lem:hanselmann}. We employ Proposition \ref{prop:sum_h} for each $z$ in the range $t \leq |z| \leq \max\left\{2t,t+ \frac{1}{HP^{3-\ve}}\right\}$, yielding
\[
\cI(q,t) \ll R_1^{1/2} R_2 R_3\left(t + \frac{1}{HP^3}\right) P^{n-1/2+\ve} H^{-(n-1)/2}\left(1 + \cY_0 + \dotsb +  \cY_{n-1} \right)^{1/2}.
\] 
Indeed, the shape of $V$ and the presence of a factor $P^\ve$ in the bound allows us to replace $|z|$ by $t$ in the final estimate, although slightly larger values of $|z|$ are also considered. The asserted bound now follows from \eqref{eq:lem20}.
\end{proof}

We shall later choose the parameter $\phi$ to be fairly small, so we may certainly assume that $Q \leq P^{2-4/(n+2)}$, as required for Proposition \ref{prop:v}.
Let 
\[
K'(t,R,\R) := B_3 R(t + \tfrac{1}{HP^3})P^{n-1/2 + \ve} H^{-(n-1)/2},
\]
corresponding to the contribution to $K(t,R,\R)$ from the term '$1$' in the expression in brackets in Proposition \ref{prop:v}.
We put
\[
H = \max\{H_1,H_2\},
\]
where
\begin{equation}
\label{eq:H_bottom}
\begin{split}
H_1 &= B_3^{2/(n-1)} T^{2/(n-1)} P^{36/(5(n-1))},\\
H_2 &= B_3^{2/(n+1)} R^{2/(n+1)} P^{6/(5(n+1))},
\end{split}
\end{equation}
a choice that produces the bound 
\begin{equation}\label{eq:1 case}K'(t,R,\R) \ll P^{n-4-1/10+\ve}.
\end{equation}

$H_1$ and $H_2$ are chosen such that when $H_1\geq H_2$, then $HP^3t\geq 1$ and vice versa. This can be easily checked by setting $H_1=H_2$, which in turn implies that 
\[
T=Rt=B_3^{-2/(n+1)}R^{1-2/(n-1)}P^{-18/5+3/5-6/(5(n+1))}=R(H_2P^3)^{-1}.
\] 
Let us now treat the other terms. Our choice of $H$ guarantees that for the other terms to be admissible, it is enough to check that the the remaining terms satisfy the following bound: $$\cY_0,\cdots,\cY_{n-1}\ll P^{1/5-4\ve}.$$ 

For $0 \leq \eta \leq n-2$, we more precisely write
\[
\cX_\eta \ll \cX'_\eta + \cX''_\eta,
\]
where, for $0\leq \eta\leq n-2$,
\begin{align*}
\cX'_\eta& = {V}^{n-\eta} 
+ B_1^{1/2}R_3^{1/2} {V}^{n-\eta-1},\cX''_\eta = B_1^{1/2} (R_3H)^{(n-\eta)/3},
\end{align*} 
\[\cX'_{n-1}=V,\quad\quad \cX''_{n-1}=(R_3H)^{1/3}. \]
Note that $V \asymp R/P+H^{1/2}R^{1/2}P^{1/2}T^{1/2}=R^{1/2}V_0$, say, where \[V_0=R^{1/2}/P+H^{1/2}P^{1/2}T^{1/2}.\]

We may then in turn write
\[
\cY_\eta \ll \cY'_\eta + \cY''_\eta,
\]
where
\[
\cY'_\eta = \frac{H^{n-\eta}}{R^{(n-\eta)/2}} \cX'_\eta, \qquad \cY''_\eta = \frac{H^{n-\eta}}{R^{(n-\eta)/2}} \cX''_\eta.
\]
Upon inspecting the shape of these expressions, one easily sees that
\begin{align}
\label{eq:ysum}
1+\cY_0 + \dotsb + \cY_{n-1} &\ll 1+\cY'_{0} + \cY''_{0} +  \cY'_{n-2} +\cY''_{n-2}+\cY''_{n-1}.
\end{align}
 We begin by inspecting the contribution from the $\cY'_\eta$ terms. First we simplify the expression to obtain
\begin{align}
\label{eq:cYeta}
\cY'_{\eta} &\ll H^{n-\eta}V_0^{n-\eta -1}(V_0+1),
\end{align}
the key observation here being that $B_1R_3\ll R$.
For $\eta = 0$, this gives
\begin{align*}
\cY'_{0} &\ll H^{n}V_0^{n-1}\left(V_0+1\right). 
\end{align*}

Note that the expressions for $H$ and $V_0$ contain non-negative powers of $R$, $B_3$ and $T$. Therefore, the maximum value of these expressions is achieved when 
\begin{equation}\label{eq:condition 1}B_3\asymp R^{1/2}\asymp P^{4/5+\phi/2}\quad \textrm{ and }\quad T\asymp P^{-8/5-\phi/2}.\end{equation} Inserting these bounds, we get 
\begin{align}
\label{eq:H_1}H_1 &= P^{8/5(n-1)+\phi/(n-1)} P^{-16/5(n-1)-\phi/(n-1)} P^{36/(5(n-1))}=P^{28/5(n-1)},\\
\label{eq:H_2}H_2 &= P^{8/5(n+1)} P^{16/5(n+1)+2\phi/(n+1)} P^{6/5(n+1)}=P^{6/(n+1)+3\phi/(n+1)}.
\end{align}
Note that $H_2\geq H_1$, when $n\geq 29$. Thus, $V_0 \asymp R^{1/2}/P \asymp P^{\phi/2-1/5} $. For a small enough value of $\phi$, we may assume that $V_0\ll 1$.

Upon using this value of $V_0$, we get that
\begin{align*}
\cY_0' &\ll H_2^{n}V_0^{n-1}\ll P^{6n/(n+1)+3\phi-(n-1)/5+(n-1)\phi/2}\ll P^{1/5-6/31+n\phi},
\end{align*}
for $n\geq 30$, which is an admissible contribution, for a small enough value of $\phi$.

Next, we consider $\eta = n-2$, where we have 
\begin{align*}
\cY_{n-2}'\ll H^2V_0(V_0+1).
\end{align*}
Again, the right hand side is non-decreasing in $B_3,R, T$, and therefore the maximum value here is reached when \eqref{eq:condition 1} holds. Replacing the corresponding values of $H_2$ and $V_0$, we get that
\begin{align*}
\cY_{n-2}'\ll P^{12/(n+1)+6\phi/(n+1)}P^{-1/5+\phi/2}\ll P^{1/5-2/155+\phi}
\end{align*}
for $n \geq 30$, which is again an admissible contribution for a small enough value of $\phi$.
Finally, 
\begin{align*}\cY''_{n-1}&\ll H^{4/3}R^{-1/6}=\max\{R^{4/3(n-1)}P^{16/3(n-1)}, R^{4/(n+1)} P^{8/5(n+1)}\}R^{-1/6}\\
&\ll P^{16/3(n-1)}\ll P^{1/5-0.01},
\end{align*}
for $n\geq 30$. We summarise our findings in the following result.
\begin{proposition}
\label{prop:avdC}
Suppose that $n\geq 30$. There exist absolute constants $\phi_0,\ve_0 >0$, depending only on $n$, such that the estimate
\begin{equation*}
K(t,R,\R) \ll_\ve P^{n-4-\ve} 
+ P^{n-4-1/10+\ve}(\cY_0''+\cY_{n-2}'')^{1/2}
\end{equation*}
holds for any $\ve \leq \ve_0$, provided that $\phi \leq \phi_0$.
\end{proposition}

\begin{rem}
With some more work, appealing to the previously mentioned bound of Serre \cite[Chapter 13]{Serre97} to improve the exponent in (\ref{eq:sparse_trivial}), the conclusion of Proposition \ref{prop:avdC} could also have been obtained for $n=29$. However, satisfactorily handling contribution from the remaining terms $\cY_0''$, $\cY_{n-2}''$ in the case $n=29$ requires substantially new arguments. We defer the necessary refinements of our method to a forthcoming paper.
\end{rem}

It remains to treat the contributions from $\cY''_\eta$ for $\eta=0,n-2$.  We begin by writing 
\[
\cY''_\eta\ll \cY''_{\eta,1}+\cY''_{\eta,2},
\]
where $\cY''_{\eta,j}$ denotes the contribution obtained by replacing $H$ with $H_j$. Let $K''_{j,\eta}(t,R,\R)$, for $\eta \in \{0,n-2\}, j\in\{0,1\}$, denote the corresponding contribution to $K(t,R,\R)$ from the term $\cY''_{j,\eta}$. More explicitly,
\begin{align}
P^{-n+4}K_{\eta,1}''(t,R,\R)&\ll P^{-1/10+\ve}H_1^{2(n-\eta)/3}R_1^{1/4}R_3^{(n-\eta)/6}R^{-(n-\eta)/4}\nonumber\\
&= (B_3 T)^{\frac{4(n-\eta)}{3(n-1)}} P^{{\frac{24(n-\eta)}{5(n-1)}}-\frac{1}{10}+\ve}R_1^{1/4}R_3^{(n-\eta)/6}R^{-(n-\eta)/4},\label{eq:B1}\\
P^{-n+4}K_{\eta,2}''(t,R,\R)&\ll P^{-1/10+\ve}H_2^{2(n-\eta)/3}R_1^{1/4}R_3^{(n-\eta)/6}R^{-(n-\eta)/4}\nonumber\\
&=(B_3 R)^{\frac{4(n-\eta)}{3(n+1)}}  P^{\frac{4(n-\eta)}{5(n+1)}-\frac{1}{10}+\ve}R_1^{1/4}R_3^{(n-\eta)/6}R^{-(n-\eta)/4}.\label{eq:B2}
\end{align}
We will compare these contributions to the relevant contribution coming from Proposition \ref{prop:w4}, which we denote by
\[
K_{\operatorname{Weyl}}(t,R,\R) := \tR T^{1+n/24}P^{n+\ve}. 
\]
Thus we shall assume until further notice that 
\begin{equation}
\label{eq:T_notsmall}
T \geq \min\{B_2^{24/(n-24)} P^{-4 + \delta},P^{-2}\},
\end{equation}
so that Proposition \ref{prop:w4} applies. The remaining range where $T$ is very small will be treated at the end of this section by different means.

We  put 
\begin{equation}\label{eq:opti}
B_1 = P^\alpha, \quad R = P^Z,
\end{equation}
so that $R_3 \asymp P^{Z-\alpha}$ and
\[
B_2\asymp P^{ \frac{Z}{3} + \frac{2\alpha}{3}}, \qquad B_3\asymp P^{\frac{Z}{3} + \frac{\alpha}{6}}. 
\]
Our strategy will be to express our bounds for the quantities $\cY''_{\eta,j}$ as functions of $Z $ and $\alpha$, and check that the desired bounds hold throughout the region $\Omega$ defined by
\begin{equation}
\label{eq:Omega}
0\leq \alpha \leq Z \leq \frac{8}{5}.
\end{equation}
Note that the Weyl bound contribution roughly grows as $R$ increases, and the contributions in \eqref{eq:B1} and \eqref{eq:B2} increase as $R$ decreases. Thus, it is natural to expect the minimum value is obtained when $Z$ is in the intermediate range between $1$ and $8/5$. This is asserted by our findings in Appendix \ref{sec:Math}.

Recall that $H=H_2$ if and only if 
\begin{equation}
\label{eq:range1}
T\leq R(H_2P^3)^{-1} = B_3^{-2/(n+1)} R^{1-2/(n+1)} P^{-6/(5(n+1))-3}.
\end{equation}
Let us first examine this range. In the case $\eta=0$, \eqref{eq:B2} implies that 
\begin{align*}
P^{-n+4}K''_{0}(t,R,\R)P^{-n+4} &\asymp K''_{0,2}(t,R,\R) \ll B_3^{\frac{4n}{3(n+1)}}  P^{\frac{4n}{5(n+1)}-\frac{1}{10}+\ve}R_1^{1/4}R_3^{n/6}R^{\frac{4n}{3(n+1)}-\frac{n}{4}}\\
&\ll P^{h_2(Z,\alpha)+\ve},
\end{align*}
say, where
\begin{align*}
h_2(Z,\alpha)&=\frac{4n}{3(n+1)}\left(\frac{Z}{3}+\frac{\alpha}{6}\right)+\frac{n}{6}\left(Z-\alpha\right)+\frac{\alpha}{4}+\left(\frac{4n}{3(n+1)}-\frac{n}{4}\right)Z+\frac{4n}{5(n+1)}-\frac{1}{10}.
\end{align*}
One may check that $h_2(Z,\alpha)$ is decreasing as a function of $n$ for admissible values of $Z$ and $\alpha$. This feature will repeat in all the bounds that we will derive in this section. Therefore we may assume that $n=30$ and thus
\begin{align}\label{eq:h_2def}
h_2(Z,\alpha)=-\frac{145}{186}Z-\frac{1687}{372}\alpha+\frac{209}{310}.
\end{align} 
In this case we have
\begin{align*}
P^{-n+4}K_{\operatorname{Weyl}}(t,R,\R) &\ll B_2 \left(B_3^{-2/(n+1)} R^{1-2/(n+1)} P^{-6/(5(n+1))-3}\right)^{1+n/24}  P^{4+\ve}\\
&\ll B_2B_3^{-\frac{2}{n+1} \frac{24+n}{24}}R^{\frac{n-1}{n+1}\frac{24+n}{24}}P^{4-\frac{15n+21}{5(n+1)}\frac{24+n}{24} + \ve} \\
&\ll P^{w_2(Z,\alpha) + \ve},
\end{align*}
where
\begin{align*}
w_2(Z,\alpha)&=\frac{Z}{3}+\frac{2\alpha}{3}-\frac{24+n}{12(n+1)}\left(\frac{Z}{3}+\frac{\alpha}{6}\right)+\frac{(n-1)(24+n)}{24(n+1)}Z+4-\frac{(15n+21)(24+n)}{120(n+1)}.
\end{align*}
Again, this is decreasing in $n$, and for $n=30$ translates to 
\begin{align}\label{eq:w_2def}
w_2(Z,\alpha)=\frac{889}{372}Z+\frac{239}{372}\alpha-\frac{1759}{620}.
\end{align}
One now reaches a simple linear optimization problem of checking whether
\[
\max_{(Z,\alpha) \in \Omega} \min\{h_2(Z,\alpha),w_2(Z,\alpha)\} < 0.
\]
Indeed, it is then obvious that
\[
\min\{K''_{\eta}(t,R,\R),K_{\operatorname{Weyl}}(t,R,\R)\} \ll P^{n-4-\phi/4},
\]
say, provided that $\phi$ and $\ve$ are chosen small enough.
These quantities may be messy to compute by hand, but could be easily computed using a simple code in Mathematica/sage. We will attach these codes in the appendix, for the aid of the reader. Using this, it is easy to check that the required maximum is $< -0.1$, giving an admissible bound in the range \eqref{eq:range1}. 

In the case of $\eta=n-2$, \eqref{eq:B2} implies that
\begin{align*}
P^{-n+4}K''_{n-2}(t,R,\R)& \asymp P^{-n+4}K''_{n-2,2}(t,R,\R) \ll (B_3 R)^{\frac{8}{3(n+1)}}  P^{\frac{8}{5(n+1)}-\frac{1}{10}+\ve}B_1^{1/4}R_3^{1/3}R^{-1/2}\\
&\ll R^{\frac{4}{(n+1)}-\frac{1}{6}}  P^{\frac{8}{5(n+1)}-\frac{1}{10}+\ve} \ll P^{\frac{8}{5(n+1)}-\frac{1}{10}+\ve} \ll P^{8/155-1/10+\ve},
\end{align*}
which gives an admissible contribution. 
Here we have again used the fact that $B_1R_3 \asymp R$.

We now consider the complementary range
\begin{equation}
\label{eq:locoid}
T > B_3^{-2/(n+1)} R^{1-2/(n+1)} P^{-6/(5(n+1))-3}.
\end{equation}
where we always have $H = H_1$. Observe that we only need to study the range where Proposition \ref{prop:v1w3} fails. Let us split that range up into three parts as follows:
\begin{gather}
\label{eq:range2}
B_3^{-2/(n+1)} R^{1-2/(n+1)} P^{-3-6/(5(n+1))} < T \leq B_2^{\frac{16}{n-17}}P^{-\frac{3n-59}{n-17}+4\phi},\\
\label{eq:range3}
B_3^{-2/(n+1)} R^{1-2/(n+1)} P^{-3-6/(5(n+1))} < T \leq B_2^{-2/(n+1)} R^{1-2/(n+1)} P^{-3-1/(n+1)},\\
\label{eq:range4}
B_2 \geq P^{\frac{4n}{45}-\frac{179}{90}}.
\end{gather}

Suppose first that \eqref{eq:range2} applies. 
Then, inserting the upper bound for $T$, we have
\begin{align*}
&P^{-n+4}K''_{\eta}(t,R,\R) \asymp P^{-n+4}K''_{\eta,1}(t,R,\R)\\
&\ll B_3^{\frac{4(n-\eta)}{3(n-1)}} B_2^{\frac{64 (n-\eta)}{3(n-1)(n-17)}} P^{\frac{24(n-\eta)}{5(n-1)}-\frac{3n-59}{n-17}\frac{4(n-\eta)}{3(n-1)} + c_1\phi-1/10 + \ve} B_1^{1/4} R_3^{(n-\eta)/6} R^{-(n-\eta)/4} \\
& \ll P^{h_1(Z,\alpha,\eta) + c_1\phi + \ve}, 
\end{align*}
say. Here, and in the sequel, $c_i$'s will denote absolute constants whose precise value is immaterial to our arguments. These constants may have an implicit dependence on $\eta$ and $n$ only.
The cases of interest for us are $\eta=0$ and $n-2$. More explicitly, 
\begin{align*}
h_1(Z,\alpha, 0) &= \frac{4n}{3(n-1)}\left(\frac{Z}{3} + \frac{\alpha}{6}\right) +\frac{64n}{3(n-1)(n-17)}\left(\frac{Z}{3} + \frac{2\alpha}{3}\right)  + \frac{1}{4} \alpha + \frac{n}{6}(Z-\alpha) - \frac{n}{4} Z\\
&+ \frac{24n}{5(n-1)}-\frac{3n-59}{n-17}\frac{4n}{3(n-1)}-\frac{1}{10}.
\end{align*}
For $n=30$ this gives
\[
h_1(Z,\alpha,0) = 
-\frac{115}{78}Z - \frac{15329}{4524} \alpha + \frac{5943}{3770}.
\]
Similarly, 
\begin{align*}
h_1(Z,\alpha, n-2) &= \frac{8}{3(n-1)}\left(\frac{Z}{3} + \frac{\alpha}{6}\right) +\frac{128}{3(n-1)(n-17)}\left(\frac{Z}{3} + \frac{2\alpha}{3}\right)  + \frac{1}{4} \alpha + \frac{1}{3}(Z-\alpha) \\
&- \frac{1}{2} Z + \frac{48}{5(n-1)}-\frac{3n-59}{n-17}\frac{8}{3(n-1)}-\frac{1}{10},
\end{align*}
which for $n=30$ translates to 
\begin{align*}
h_1(Z,\alpha, n-2) &= -\frac{23}{234}Z+\frac{101}{13572}\alpha+\frac{133}{11310}.
\end{align*}

 If we instead feed the upper bound for $T$ into the bound in Proposition \ref{prop:w4}, we get
\begin{align*}
P^{-n+4}K_{\operatorname{Weyl}}(t,R,\R) &= B_2 T^{1+n/24}  P^{4+\ve}\ll B_2^{1+\frac{16}{n-17} \frac{24+n}{24}}P^{4-\frac{3n-59}{n-17}\frac{24+n}{24}+c_2\phi + \ve} \\
&\ll P^{w_1(Z,\alpha) + c_2\phi + \ve},
\end{align*}
where
\[
w_1(Z,\alpha) = \left(1+\frac{16}{n-17} \frac{24+n}{24}\right)\left(\frac{Z}{3} + \frac{2\alpha}{3}\right)+4-\frac{3n-59}{n-17}\frac{24+n}{24}.
\]
Again for $n=30$ one has 
\[
w_1(Z,\alpha) = \textstyle \frac{49}{39} Z + \frac{98}{39} \alpha - \frac{71}{52}.
\]
Using linear optimization again one checks that
\[
\max_{(Z,\alpha) \in \Omega} \min\{h_1(Z,\alpha,0),w_1(Z,\alpha)\} \leq -0.01,\quad \max_{(Z,\alpha) \in \Omega} \min\{h_1(Z,\alpha,n-2),w_1(Z,\alpha)\} \leq -0.02.
\]

Next we will treat the range \eqref{eq:range3}. Again using $H = H_1$ and the upper bound for $T$ in \eqref{eq:range3},
we have
\begin{align*}
&P^{-n+4}K''_{\eta}(t,R,\R) \asymp P^{-n+4}K''_{\eta,1}(t,R,\R)\\
&\ll B_3^{\frac{4(n-\eta)}{3(n-1)}} B_2^{-\frac{8(n-\eta)}{3(n+1)(n-1)}} R^{\frac{4(n-\eta)}{3(n+1)}-\frac{n-\eta}4} P^{-\frac{4(n-\eta)}{3(n-1)}\frac{3n+4}{n+1}+\frac{24(n-\eta)}{5(n-1)}-\frac{1}{10} + \ve } B_1^{1/4} R_3^{(n-\eta)/6}  \\
& \ll P^{h_3(Z,\alpha,\eta) + \ve}, 
\end{align*}
where
\begin{align*}
h_3(Z,\alpha,0) &= \frac{4n}{3(n-1)}\left(\frac{Z}{3} + \frac{\alpha}{6}\right) -\frac{8n}{3(n+1)(n-1)}\left(\frac{Z}{3} + \frac{2\alpha}{3}\right) + \left(\frac{4n}{3(n+1)}-\frac{n}{4}\right)Z \\
& + \frac{1}{4} \alpha + \frac{n}{6}(Z-\alpha) -\frac{4n}{3(n-1)}\frac{3n+4}{n+1}+\frac{24n}{5(n-1)}-\frac{1}{10}, \\ 
h_3(Z,\alpha,n-2) &= \frac{8}{3(n-1)}\left(\frac{Z}{3} + \frac{\alpha}{6}\right) -\frac{16}{3(n+1)(n-1)}\left(\frac{Z}{3} + \frac{2\alpha}{3}\right) + \left(\frac{8}{3(n+1)}-\frac{1}{2}\right)Z \\
& + \frac{1}{4} \alpha + \frac{1}{3}(Z-\alpha) -\frac{8}{3(n-1)}\frac{3n+4}{n+1}+\frac{48}{5(n-1)}-\frac{1}{10}.\end{align*}
Specialising to $n=30$, we get
\begin{align*}
h_3(Z,\alpha,0) = -\frac{145}{186} Z - \frac{49403}{10788}\alpha + \frac{6141}{8990},\,\,\,
h_3(Z,\alpha,n-2) = -\frac{29}{558} Z - \frac{2329}{32364} \alpha -\frac{1289}{26970}.
\end{align*}
The Weyl differencing bound in this range for $T$ yields
\begin{align*}
P^{-n+4} K_{\operatorname{Weyl}}(t,R,\R) &\ll B_2 \left(B_2^{-2/(n+1)} R^{1-2/(n+1)} P^{-1/(n+1)-3}\right)^{1+n/24}  P^{4+\ve} \\
&\ll P^{w_3(Z,\alpha)  + \ve},
\end{align*}
where
\begin{align*}
w_3(Z,\alpha)&=\left(\frac{Z}{3}+\frac{2\alpha}{3}\right)\left(1-\frac{2}{n+1}\left(\frac{24+n}{24}\right)\right)+\frac{(n-1)(24+n)}{(n+1)24}Z\\
&+4-\left(\frac{3n+4}{n+1}\right)\left(\frac{24+n}{24}\right).
\end{align*}
For $n=30$, this becomes
$
w_3(Z,\alpha) = \frac{889}{372} Z + \frac{53}{93} \alpha - \frac{175}{62}.$
Clearly, the exponent $$h_3(Z,\alpha,n-2) \leq -\frac{1289}{26970} $$ suffices. In the other case,  one calculates
\[
\max_{(Z,\alpha) \in \Omega} \min\{h_3(Z,\alpha,0),w_3(Z,\alpha)\} \leq -0.008,
\]
giving an admissible bound in the range \eqref{eq:range3}.

Finally, suppose that \eqref{eq:range4} applies. In this range, in some cases, it will be more convenient to refer to an earlier bound in Proposition \ref{prop:v} directly:
\begin{align*}
P^{-n+4}K''_{\eta}(t,R,\R)&\ll B_3 R(t + \tfrac{1}{HP^3}) P^{7/2+ \ve} H^{-(n-1)/2} (\cY_\eta'')^{1/2}\\
&\ll B_3 R(t + \tfrac{1}{HP^3})  P^{7/2+ \ve}B_1^{1/4}H^{(1-\eta)/2}(R_3H)^{(n-\eta)/6}R^{-(n-\eta)/4}.
\end{align*}  Again, we start with the case $\eta=0$ first. Comparing with the bounds \eqref{eq:H_1} and \eqref{eq:H_2}, one also always has $H \leq P^{1/5}$, so that
\[
 H^{1/2}(R_3H)^{n/6} R^{-n/4}\asymp H^{(n+3)/6} B_2^{-n/4} \leq P^{\frac{n+3}{30} + \frac{179n}{360}-\frac{n^2}{45}} \leq P^{-239/60}
\]
for $n \geq 30$. This yields
\begin{align*}
P^{-n+4}K''_{0}(t,R,\R) &\ll  B_3B_1^{1/4} (T+RP^{-3}) P^{7/2+\ve} H^{1/2}(R_3H)^{n/6} R^{-n/4}   \\
&\ll RP^{-7/5+\phi} P^{7/2-239/60 +\ve}\ll P^{-17/60+2\phi+\ve }  
\end{align*}
a bound which is clearly admissible.

For the contribution $K''_{n-2}(t,R,\R)$, we start by writing 
\begin{align*}
P^{-n+4}K''_{n-2}(t,R,\R)&\ll P^{-1/10}(\cY''_{n-2})^{1/2}\ll P^{-1/10+\ve}R^{-1/2}\bR_1^{1/4}R_3^{1/3}H^{4/3}.
\end{align*}
As a consequence,
\begin{align*}
P^{-n+4}K''_{n-2,1}(t,R,\R)&\ll P^{-1/10+\ve}R^{-1/2}\bR_1^{1/4}R_3^{1/3}(B_3 T)^{\frac{8}{3(n-1)}}P^{\frac{48}{5(n-1)}}.
\end{align*}
We insert the trivial upper bound $T \leq P^{-8/5}$ to obtain
\[
P^{-n+4}K''_{n-2,1}(t,R,\R) \ll B_3^{\frac{8}{3(n-1)}} P^{\frac{16}{3(n-1)}-\frac{1}{10} + \ve}  \bR_1^{1/4} R_3^{1/3} R^{-1/2}.
\]
By \eqref{eq:range4}, one now has
\begin{align*}
B_3^{\frac{8}{3(n-1)}} \bR_1^{1/4} R_3^{1/3} R^{-1/2} &\ll B_1^{\frac{4}{3(n-1)}-\frac{1}{4}} R_3^{\frac{8}{9(n-1)}-\frac{1}{6}} \ll B_2^{\frac{4}{3(n-1)}-\frac{1}{4}} \ll P^{(\frac{4}{3(n-1)}-\frac{1}{4})(\frac{4n}{45}-\frac{179}{90})}.
\end{align*}
implying that
\begin{align*}
P^{-n+4}K''_{n-2,1}(t,R,\R) &\ll P^{\frac{16}{3(n-1)}-\frac{1}{10}+(\frac{4}{3(n-1)}-\frac{1}{4})(\frac{4n}{45}-\frac{179}{90})+ \ve }\ll P^{-0.05+\ve},
\end{align*}
for $n\geq 30$.

We have now obtained a bound which is valid everywhere except for very small values of $T$.

\begin{proposition}
\label{prop:prop10}
Suppose that $n \geq 30$ and that 
\[
T \geq \min\{B_2^{24/(n-24)} P^{-4 + \delta},P^{-2}\}.
\]
Then we have the bound 
\[
K(t,R,\R) \ll P^{n-4} \Big(P^{-5\delta/24} + P^{-\phi/4}\Big)
\]
provided only  that 
$\phi \leq \phi_0$, where $\phi_0$ is an absolute constant.
\end{proposition}

In the remaining range, we shall employ the pointwise van der Corput differencing bound in Lemma \ref{lem:pointwise}. Using this result in place of Lemma \ref{lem:hanselmann} in the above argument, one obtains the bound
\begin{align*}
K(t,R,\R) &\ll B_3 T P^{n+\ve} H^{-n/2} \Big(1 + \cY_0 + \dotsb + \cY_{n-1}\Big)^{1/2}.
\end{align*} 

Let us put 
\[
H := 1 + B_2^{4/n}(B_3 T P^4)^{2/n}, %
\] 
yielding
\[
B_3 T P^{n+\ve} H^{-n/2} \ll P^{n-4+\ve} B_2^{-2}.
\]
Recall that we are assuming that 
$T \leq \min\{B_2^{24/(n-24)} P^{-4 + \delta},P^{-2}\}.$
Since $n \geq 30$, we have in particular $T \leq B_2^4 P^{-4+\delta}$.
On the other hand, by \eqref{eq:minor_arcs_condition} we have either 
\[
T \geq RP^{-4+\Delta} \qquad \text{or} \qquad R \geq P^{\Delta}.
\]
Since $R^{1/3} \ll B_2 \ll R$, we may conclude in both these cases that
\[
B_2 \gg P^{(\Delta-\delta)/3}.
\]

In the range in the contention,
\begin{equation}
\label{eq:Hbound}
H\leq B_2^{(4(n-12)/(n(n-24))}B_3^{2/n}P^{2\delta/n}.
\end{equation}
The inequality \eqref{eq:ysum} is still valid, and therefore, it is enough to only look at the corresponding terms. Therefore, we shall now estimate the different contributions $K'_\eta(t,R,\R)$, $K''_\eta(t,R,\R)$ for $\eta \in \{0,n-2\}$ and $K_{n-1}''(t,R,\R)$, named according to the contribution from the corresponding terms in \eqref{eq:ysum}.  We begin by investigating the term $K'_{0}(t,R,\R)$ first:
\begin{align*}
P^{-n+4} K'_\eta(t,R,\R) \ll B_3 T P^{4+\ve} R^{-(n-\eta)/4} V^{(n-\eta)/2} \Big(1 + R^{1/4}V^{-1/2}\Big). 
\end{align*}
We note that when inserting the chosen value of $H$, we have $B_3$, $B_2$, $R$ and $T$ occurring to non-negative exponents, as long as $\eta\leq n-2 $ so it suffices to estimate this contribution when $B_3^2 \asymp B_2 \asymp R \asymp P^{8/5+\phi}$ and $T \asymp P^{-2}$. In that case one has
\[
tHP^3 \asymp HPR^{-1} \asymp PR^{-1}(B_2^2 B_3 P^2)^{2/n} \asymp P^{1+4/n} R^{-1+5/n} \asymp P^{-3/5 + 12/n-(n-5)\phi/n} \ll 1
\]
as soon as $n > 20$. Thus we may assume that $V \asymp R/P$. Then we get
\begin{align*}
P^{-n+4} K'_0(t,R,\R) &\ll B_3 T P^{4+\ve} R^{n/4} P^{-n/2} \Big(1 + R^{-1/4}P^{1/2}\Big)\\
&\ll B_3 P^{2+\ve - n/10 + n\phi/4} \Big(1 + P^{1/10}\Big) \ll P^{-(n-29)/10+ (n+2)\phi/4 + \ve},
\end{align*}
so an admissible contribution as soon as $n \geq 30$. 

We next observe that in this range, we have $V\ll R/P<R^{1/2}$, which means that
\begin{align*}
P^{-n+4} K'_{n-2}(t,R,\R)&\ll B_2^{-2}P^{\ve} (HR^{-1/2}V + HR^{-1/4}V^{1/2})\ll  B_2^{-2}P^{\ve} HR^{-1/4}V^{1/2}\\
&\ll B_2^{4/n}(B_3 P^2)^{2/n}P^{\ve-1/2} B_2^{-2}R^{1/4}\\
&\ll P^{\ve-1/2+4/n}R^{5/n-7/4}\ll P^{\ve-1/2+4/n},
\end{align*}
clearly, for $n\geq 3$, rendering us with an appropriate contribution.

Similarly,
\begin{align*}
P^{-n+4}K_{n-1}''(t,R,\R)&\ll P^{\ve}B_2^{-2}H^{1/2}R^{-1/4} R_3^{1/6}H^{1/6}\ll P^{\ve}B_2^{-3/2}H^{2/3}R^{-1/4}\\
&\ll P^{\ve+4\delta/3n}B_2^{-3/2}R^{-1/4}B_2^{8(n-12)/(3n(n-24))}B_3^{4/(3n)}\\
&\ll P^{\ve+4\delta/3n}B_2^{8(n-12)/(3n(n-24))+4/(3n)-7/4}\\
&\ll P^{\ve+4\delta/75}B_2^{-31/100} \ll P^{\ve-(\Delta-2\delta)/10}, 
\end{align*}
for $n\geq 25$, giving a reasonable contribution as long as $\delta<\Delta/4$, say. We have not tried to optimise the exponents here.

Next, we have
\begin{align*}
P^{-n+4} K''_0(t,R,\R) &\ll P^{\ve}B_2^{-2} H^{2n/3} R^{-n/4} \bR_1^{1/4} R_3^{n/6} \\
&\ll P^{\ve+4\delta/3}B_2^{8(n-12)/(3(n-24))-2}B_3^{4/3}R^{-n/4} \bR_1^{1/4} R_3^{n/6} \\
&\ll P^{\ve+4\delta/3}\bR_1^{8(n-12)/(3(n-24))-13/12-n/4}R_3^{8(n-12)/(9(n-24))-n/12-2/9}\\
&\ll P^{\ve+4\delta/3}\bR_1^{-1/2}R_3^{-0.05}\ll P^{\ve+4\delta/3}B_2^{-0.15}\ll P^{\ve-\delta/6}
\end{align*}
for $n \geq 30$, provided that $\delta \leq \Delta/31$.

Finally, for $n\geq 30$ we have
\begin{align*}
P^{-n+4}K''_{n-2}(t,R,\R) &\ll B_2^{-2} P^\ve\frac{H}{R^{1/2}} \bR_1^{1/4} (R_3H)^{1/3} \\
&\ll B_2^{-2} P^\ve \bR_1^{1/4} R_3^{1/3} R^{-1/2} (B_3 B_2^{2(n-12)/(n-24)} P^\delta)^{8/(3n)}\\
&\ll B_2^{-2} P^\ve \bR_1^{1/4} R_3^{1/3} R^{-1/2} (B_3 B_2^6 P^\delta)^{4/45} \\
&= B_2^{-22/15} B_3^{4/45} \bR_1^{1/4} R_3^{1/3} R^{-1/2} P^{4\delta/45 + \ve}\ll B_2^{-22/45} P^{4\delta/45 + \ve} \\
&\ll P^{-22(\Delta-\delta)/135 + 4\delta/45 + \ve}\ll P^{-\delta + \ve},
\end{align*}
under the previous assumption on $\delta$. 
Choosing $\delta = \Delta/31$,
and $\ve$ small enough, we have now finally verified the assertion in Proposition \ref{prop:minor_smooth}, where we may take $\psi = \min\{\phi/4, \Delta/186\}$ for any $\phi \leq \phi_0$.

\section{Bounds for the minor arcs contribution, general case}
\label{sec:slicing}

In this section we shall give the inductive argument to deduce the Proposition \ref{prop:minor_inhom} from the Proposition \ref{prop:minor_smooth}. Assume that $\dim \Sing(X_0) \geq 0$.
We subject the exponential sum
\[
S(q,z) = S(q,z;F,W,P,n) = \starsum_{a=1}^q \sum_{\x\in \ZZ^n} W(\x/P)e((a/q+z)
F(\x))
\]
to a similar slicing argument to that used in the proof of Proposition \ref{prop:sum_h}. We may apply Lemma \ref{lem:subspace} for the single form $F_0$, with $\Pi = \varnothing$, and with  $\f = (f_1,\dotsc,f_n) = \nabla F(P\x_0)$, where $\x_0$ is the point in \eqref{eq:W def}, to obtain a linearly independent set of vectors $\e_1,\dotsc,\e_n$ with the properties listed there and such that the angle of $\e_1$ with $\f$ is at most $\pi/3$.
We then split the sum above into cosets of the lattice spanned by $\e_1,\dotsc,\e_{n-1}$, as in the proof of Proposition \ref{prop:sum_h}, to obtain
\begin{align*}
S(q,z) &= \sum_{\t \in T} S_\t(q,z) =  \sum_{\t \in T} S(q,z;F_\t,W_\t,P/L,n-1),
\end{align*}
say, where
\[
F_\t(\y) = F(\t + \sum y_i \e_i) \quad \text{and} \quad W_\t(\y) = W(P^{-1}\t + L^{-1} \sum y_i \e_i)
\]
for some $L=O(1)$, and where $T$ is a set of cardinality $O(P)$. The sum over $\t$ can be moved outside to obtain
\begin{equation}
\label{eq:minor_slice}
S_\fm = S_\fm(F,W,P,n) = \sum_{\t} S_{\fm,\t},
\end{equation}
where
\begin{align*}
S_{\fm,\t}= S_\m(F_\t,W_\t,P',n-1)&= 
\sum_{1\leq q\leq P^\dl}\int_{P^{-4+\dl} \leq |z| \leq (qQ)^{-1+\theta}} |p_q(z)||S(q,z;F_\t,W_\t,P',n-1)|\,dz \\
&+\sum_{P^{\Delta}\leq q\leq Q}\int_{|z| \leq (qQ)^{-1+\theta}}
|p_q(z)||S(q,z;F_\t,W_\t,P',n-1)|\,dz.
\end{align*}
Here we have put $P' := P/L$. We must then investigate whether the polynomials $F_\t$ and the weight functions $W_\t$ fulfil the hypotheses in Proposition \ref{prop:minor_inhom}, in order to invoke our induction hypothesis.

Much like in the proof of Proposition \ref{prop:sum_h}, one checks that $W_\t \in \cW_{n-1}$ and
\[
\Vert F_\t \Vert_{P'} \ll L^4 \Vert F \Vert_P \ll \Vert F \Vert_P. 
\] 
We shall now find $M'> 0$ such that the conditions \eqref{eq:M_1}--\eqref{eq:M_2} hold with $F,W,P$ replaced by $F_\t,W_\t,P'$ and $M$ replaced by $M'$. 

If $\t \in \ZZ^n$ is arbitrary and $\y \in \ZZ^n$ is such that $\x = \t + \sum_{i=1}^{n-1}y_i \e_i$ occurs in the exponential sums $S_\t(q,z)$, then we have
\[
\nabla F_\t(P\y) = \nabla F(P\x).E,
\]
where we denote by $E$ the $n \times (n-1)$-matrix with the basis vectors $\e_1,\dotsc,\e_{n-1}$ as columns.  By \eqref{eq:M_1} we then have
\[
|f_1| \geq M P^3.
\]
Now, let $\y_1,\y_2 \in P'\supp(W_\t)$ and put $\x_j = \t + \sum_i y_i \e_i \in P\supp(W)$ for $j = 1,2$. If $\f_j := \nabla F(\x_j)$ and $\g_j:= \nabla F_\t(\y_j)$ then we have $\g_j = \f_j.E$ for $j = 1,2$. Thus we see that
\begin{align*}
|g_{1,1}| &= |\f_1.\e_1| \geq |\f.\e_1| - |(\f_1-\f).\e_1|.
\end{align*} 
Since $\e_1$ was chosen to have an angle of at most $\pi/3$ to the vector $\f$, we get
\[
|g_{1,1}| \geq \frac{1}{2}|\f|_2|\e_1|_2 - |\f_1-\f|_2 |\e_1|_2 \geq \frac{L}{2}\left(\frac{1}{2}|\f|_\infty - \sqrt{n} |\f_1-\f|_\infty \right).
\]

Using \eqref{eq:M_1}--\eqref{eq:M_2} the expression in brackets satisfies
\[
\frac{1}{2}|\f|_\infty - \sqrt{n}|\f_1-\f|_\infty \geq \frac{M P^3}{2} - \frac{M P^3}{8^n \sqrt{(n-1)!}} \geq \frac{MP^3}{4}, 
\]
so the condition \eqref{eq:M_1} holds with $M$ replaced by
\[
M' = \frac{L^4 M}{8} \ll M.
\]
Furthermore, we have
\[
|(\g_1)_i -(\g_2)_i| = |\e_i.(\f_1-\f_2)| \leq |\e_i|_2 |\f_1-\f_2|_2 \leq L \sqrt{n} \cdot \frac{MP^3}{8^n \sqrt{n!}} = \frac{M'(P')^3}{8^{n-1}\sqrt{(n-1)!}}, 
\]
so that \eqref{eq:M_2} also holds for $M'$.

Since 
\[
\dim \Sing(X_0 \cap H) = \dim \Sing(X_0) - 1,
\]
where $H$ is the projective hyperplane spanned by the vectors $\e_1,\dotsc,\e_{n-1}$, we may now assume by induction that for each $\t$ occurring in \eqref{eq:minor_slice}, the bound
\[
S_{\fm,\t} = O((P')^{n-5-\psi}) = O(P^{n-5-\psi})
\]
holds for some $\psi = \psi(\Delta)$, for an appropriate choice of $\phi$ and $\theta$. The implied constant depends on the quantity $\Vert F_\t \Vert_{P'}$, but we easily see that$\Vert F_\t \Vert_{P'} \ll \Vert F \Vert_P$ for any $\t \in [-P,P]^n$. Since the sum in \eqref{eq:minor_slice} includes at most $O(P)$ choices of $\t \in [-P,P]^n$, we immediately obtain the desired bound
\[
S_\fm = O(P^{n-4-\psi}),
\]
which is valid for the same choice of $\phi$ and $\theta$. This completes the proof of Proposition \ref{prop:minor_inhom}.

\appendix

\section{Mathematica code}
\label{sec:Math}

Here, we give the Mathematica code verifying our linear optimisation bounds in Section 
\ref{sec:minor}. 

\mbox{
\includegraphics[scale=0.7]{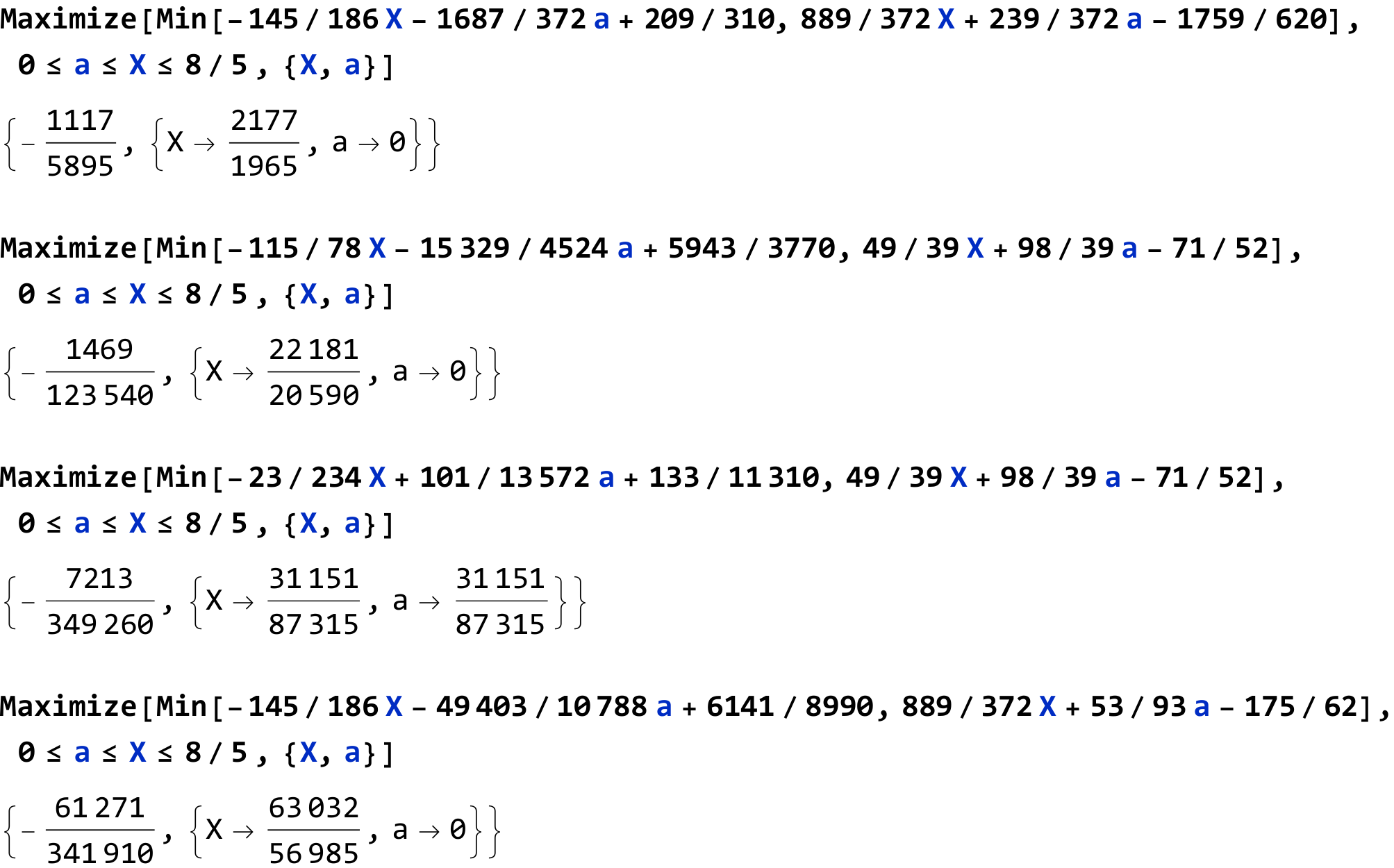}

}

\bibliographystyle{plain}
\end{document}